\documentclass[a4paper,11pt,reqno,noindent]{amsart}
\usepackage[centertags]{amsmath}
\usepackage{mathtools}
\usepackage{amsfonts,amssymb,amsthm,dsfont,cases,amscd,esint,enumerate}
\usepackage[T1]{fontenc} 
\usepackage[english]{babel}
\usepackage[applemac]{inputenc}
\usepackage{newlfont,cancel}
\usepackage{color}
\usepackage[body={15cm,21.5cm},centering]{geometry}
\usepackage{fancyhdr}
\theoremstyle{plain}
\newtheorem{theor10}{Theorem}
\newenvironment{theor1}
  {\pushQED{\qed}\begin{theor10}}
  {\popQED\end{theor10}}
\newtheorem{cor10}[theor10]{Corollary} 

\newtheorem{lem0}{Lemma}[section] 
\newenvironment{lem} 
  {\pushQED{\qed}\begin{lem0}}
  {\popQED\end{lem0}}
\theoremstyle{plain}
\newtheorem{theor0}[lem0]{Theorem}

\newtheorem{prop0}[lem0]{Proposition}
\newenvironment{prop}
  {\pushQED{\qed}\begin{prop0}}
  {\popQED\end{prop0}}
\newtheorem{cor0}[lem0]{Corollary}

\theoremstyle{definition}
\newtheorem{rems0}[lem0]{Remarks}

\newtheorem{rem0}[lem0]{Remark}

\numberwithin{equation}{section}

\newcommand{\rnabla}{{\langle\nabla\rangle}} 
\newcommand{\rv}{{\langle v\rangle}}
\newcommand{\rt}{{\langle t\rangle}}
\newcommand{\rs}{{\langle s\rangle}}
\newcommand{\rvs}{{\langle v_*\rangle}}

\newcommand{\Sp}{\mathbb S}
\newcommand{\e}{\varepsilon}

\newcommand{\R}{\mathbb R}

\newcommand{\Bc}{\mathcal B}
\newcommand{\Bo}{\overline{\mathcal B}}

\newcommand{\Id}{\operatorname{Id}}
\newcommand{\Tr}{\operatorname{tr}}

\newcommand{\sgn}{\operatorname{sgn}}
\newcommand{\LB}{{\operatorname{LB}}}

\newcommand{\loc}{{\operatorname{loc}}}

\newcommand{\Ld}{\operatorname{L}}

\newcommand{\3}{\operatorname{|\!|\!|}}
\newcommand{\step}[1]{\noindent \textit{Step} #1.}
\newcommand{\substep}[1]{\noindent \textit{Substep} #1.}

\newcommand{\pv}{\operatorname{p.}\!\operatorname{v.}}

\def\dbar{\,{{\mathchar'26\mkern-12mud}}}

\usepackage[colorlinks,linkcolor=black,citecolor=black,urlcolor=black]{hyperref}

\title[Well-posedness of the Lenard--Balescu equation]{Well-posedness of the Lenard--Balescu equation\\with smooth interactions}

\author[M. Duerinckx]{Mitia Duerinckx}
\address[Mitia Duerinckx]{Université Paris-Saclay, CNRS, Laboratoire de Mathématiques d'Orsay, Orsay, F-91400, France
\& Universit\'e Libre de Bruxelles, Département de Mathématique, Brussels, B-1050, Belgium
\& University of California, Los Angeles, Department of Mathematics, CA-90095, USA}
\email{mitia.duerinckx@u-psud.fr}
\author[R. Winter]{Raphael Winter}
\address[Raphael Winter]{École Normale Supérieure de Lyon, CNRS, Unité de Mathématiques Pures et Appliquée, Lyon, F-69364, France}
\email{raphael.winter@ens-lyon.fr}

\begin{document}
\selectlanguage{english}

\begin{abstract}
The Lenard--Balescu equation was formally derived in the 1960s as the fundamental description of the collisional process in a spatially homogeneous system of interacting particles. It can be viewed as correcting the standard Landau equation by taking into account collective screening effects.
Due to the reputed complexity of the Lenard--Balescu equation in case of Coulomb interactions, its mathematical theory has remained void apart from the linearized setting~\cite{Merchant-Liboff-73,Strain-07}.
In this contribution, we focus on the case of smooth interactions and we show that dynamical screening effects can then be handled perturbatively. Taking inspiration from the Landau theory, we establish global well-posedness close to equilibrium, local well-posedness away from equilibrium, and we discuss the convergence to equilibrium and the validity of the Landau approximation.
\end{abstract}

\maketitle

\setcounter{tocdepth}{1}
\tableofcontents

\section{Introduction}
In a spatially homogeneous particle system with mean-field interactions, the self-consistent Vlasov force vanishes, and the particle velocity density $F$ is predicted to satisfy to leading order the so-called Lenard--Balescu equation,
\begin{align}\label{eq:LB}
\partial_tF=\LB(F),
\end{align}
where the Lenard--Balescu operator is given by
\[\LB(F)\,:=\,\nabla\cdot\int_{\R^d}B(v,v-v_*;\nabla F)\,\big(F_*\nabla F-F\nabla_*F_*\big)\,dv_*,\]
with the notation $F=F(v)$, $F_*=F(v_*)$, $\nabla=\nabla_v$, and $\nabla_*=\nabla_{v_*}$,
in terms of the collision kernel
\begin{align}\label{eq:LB-kernel}
B(v,v-v_*;\nabla F)\,:=\,\int_{\R^d}(k\otimes k)\,\pi\widehat V(k)^2\tfrac{\delta(k\cdot(v-v_*))}{|\e(k,k\cdot v;\nabla F)|^2}\,\dbar k,
\end{align}
where $V$ stands for the particle interaction potential and where the dispersion function $\e$ is defined by
\begin{align}\label{eq:LB-disp}
\e(k,k\cdot v;\nabla F)\,:=\,1+\widehat V(k)\int_{\R^d}\tfrac{k\cdot\nabla F(v_*)}{k\cdot(v-v_*)-i0}\,dv_*.
\end{align}
Note that the collision kernel~\eqref{eq:LB-kernel} only makes sense provided that this dispersion function does not vanish, which is related to the linear Vlasov stability of $F$ in view of the Penrose criterion, cf.~\cite{Penrose}.
In the physically important case of 3D Coulomb interactions, \mbox{$\widehat V(k)=|k|^{-2}$}, the integral over $k$ in~\eqref{eq:LB-kernel} is logarithmically divergent at infinity, which corresponds to the contribution of collisions with small impact parameter, hence an appropriate cut-off $|k|\le\frac1{\delta}$ needs to be included at so-called Landau length~$\delta$. 
In 1D, particle systems undergo a kinetic blocking and the Lenard--Balescu operator is trivial ($B\equiv0$): we henceforth restrict to space dimension~$d\ge2$.
 
\medskip 
Equation~\eqref{eq:LB} was first derived in the early 60s independently by Guernsey~\cite{Guernsey-60,Guernsey-62}, Balescu~\cite{Balescu-60,Balescu-63}, and Lenard~\cite{Lenard-60} as the accurate kinetic description of homogeneous plasmas (see also~\cite[Appendix~A]{Nicholson}).
The collision kernel~\eqref{eq:LB-kernel} is obtained formally by resolving the long-time effects of particle correlations. We refer to~\cite{D-21a,DSR,VW-18,Winter-21,Nota-21-1,Nota-21-2} for first steps towards its rigorous justification;
see also~\cite{Bobylev-Potapenko-19} for a discussion of the validity of the spatial homogeneity assumption in plasma physics.
At a formal level, equation~\eqref{eq:LB} has the following expected physical properties:
\begin{enumerate}[$\bullet$]
\item it preserves mass, momentum, and kinetic energy,
\begin{equation}\label{eq:conserv}
\partial_t\int_{\R^d}\big(1,v,\tfrac12|v|^2\big)F=0;
\end{equation}
\item its steady states are Maxwellian distributions,
\[\LB(\mu_\beta)=0,\qquad\mu_\beta(v)\,:=\,(\tfrac\beta\pi)^{\frac d2}e^{-\beta|v|^2},\qquad0<\beta<\infty;\]
\item it satisfies an $H$-theorem, that is, the Boltzmann entropy is decreasing along the flow,
\begin{align}
\partial_t\int_{\R^d}F\log F
~=&~-\tfrac12\iint _{\R^d\times\R^d}FF_*\Big(\tfrac{\nabla F}F-\tfrac{\nabla_*F_*}{F_*}\Big)\cdot B(v,v-v_*;\nabla F)\Big(\tfrac{\nabla F}F-\tfrac{\nabla_*F_*}{F_*}\Big)\nonumber\\
~\leq&~0.\label{eq:H-princ}
\end{align}
\end{enumerate}
In particular, this equation formally describes the relaxation of the velocity density $F$ towards Maxwellian equilibrium.

\medskip
The main difficulty to study this equation stems from the nonlinearity and nonlocality of its collision kernel~\eqref{eq:LB-kernel}, which from the physical perspective express collective effects and dynamical screening.
Neglecting these would amount to replacing the dispersion function~$\e(k,k\cdot v;\nabla F)$ by a constant in the collision kernel $B(v,v-v_*;\nabla F)$: the latter would then reduce to the usual Landau kernel $B_L(v-v_*)$,
\begin{multline}\label{eq:expl-Landau}
B(v,v-v_*;\nabla F)~~\leadsto~~\int_{\R^d}(k\otimes k)\,\pi\widehat V(k)^2\,\delta(k\cdot(v-v_*))\,\dbar k\\
~=~\tfrac{L}{|v-v_*|}\Big(\Id-\tfrac{(v-v_*)\otimes(v-v_*)}{|v-v_*|^2}\Big)~=:~B_L(v-v_*),
\end{multline}
as follows indeed from a direct computation, see~\eqref{eq:expl-Landau-pr} below,
with explicit prefactor
\begin{equation}\label{eq:cst-L}
L\,:=\,\frac{\omega_{d-1}}{d\omega_d}\int_{\R^d}|k|\pi\widehat V(k)^2\dbar k,
\end{equation}
where $\omega_{n}$ stands for the volume of the $n$-dimensional unit ball.
In case of 3D Coulomb interactions, $\widehat V(k)=|k|^{-2}$, collective effects captured by the Lenard--Balescu dispersion function $\e(k,k\cdot v;\nabla F)$ are particularly relevant as they include Debye shielding: indeed, while Landau integrals~\eqref{eq:expl-Landau} and~\eqref{eq:cst-L} diverge logarithmically and require a cut-off both at small and large $k$ in that case, the Lenard--Balescu integral~\eqref{eq:LB-kernel} is formally convergent at small $k$.
As shown in~\cite{Strain-07}, the inclusion of small wavenumbers actually yields a huge difference between Lenard--Balescu and Landau operators: when evaluated at Maxwellian, the dispersion function $\e(k,k\cdot v;\nabla\mu_\beta)$ is not bounded away from $0$
and the collision case $B(v,v-v_*;\nabla \mu_\beta)$ displays exponential growth in velocity in some directions, in stark contrast with the Landau kernel~\eqref{eq:expl-Landau}, 
cf.~\cite[Theorem~6]{Strain-07}.
From the mathematical perspective, this makes the rigorous study of the Lenard--Balescu equation reputedly difficult in the Coulomb setting, see e.g.~\cite[p.64]{Alexandre-Villani-04}: even local well-posedness close to Maxwellian remains an open question as it would require fine control of this unbounded collision kernel along the flow.

\medskip
In the present contribution, we focus on the simpler case of a smooth interaction potential~$V$: we show that at Maxwellian equilibrium the dispersion function $\e(k,k\cdot v;\nabla\mu_\beta)$ is bounded away from $0$ in that case, see Lemma~\ref{lem:eps-bndd-R} below, which then allows to handle dynamical screening perturbatively and to compare to the well-studied Landau case~\eqref{eq:expl-Landau}.
In this spirit, the following main result states the global well-posedness for strong solutions close to Maxwellian.
The proof is inspired by Guo's work on the Landau equation~\cite{Guo-02}. Note that we expect the statement to hold for all $s>\frac32$, but we restrict to integer differentiability to avoid additional technicalities.

\begin{theor1}[Global well-posedness close to equilibrium]\label{th:global}
Let $V\in\Ld^1\cap \dot H^2(\R^d)$ be isotropic and positive definite, and assume $xV\in\Ld^2(\R^d)$.
For all $s\ge2$ and $0<\beta<\infty$, there is a constant $C_{V,\beta,s}$ large enough such that the following holds: for all initial data~$F^\circ\in\Ld^1(\R^d)$ of the form
\[F^\circ=\mu_\beta+\sqrt{\mu_\beta} f^\circ\ge0,\qquad f^\circ\in H^s(\R^d),\]
satisfying smallness and centering conditions,
\begin{equation}\label{eq:smallness-st}
\|f^\circ\|_{H^s(\R^d)}\,\le\,\tfrac1{C_{V,\beta,s}},\qquad\int_{\R^d}\big(1,v,\tfrac12|v|^2\big)\sqrt{\mu_\beta} f^\circ=0,
\end{equation}
there exists a unique global strong solution $F$ of the Lenard--Balescu equation~\eqref{eq:LB} with initial data $F^\circ$, in the form
\[F=\mu_\beta+\sqrt{\mu_\beta} f\ge0,\qquad f\in\Ld^\infty(\R^+;H^s(\R^d)),\]
and it satisfies for all $t\ge0$,
\begin{equation*}
\|f^t\|_{H^s(\R^d)}\,\lesssim_{V,\beta,s}\,\|f^\circ\|_{H^s(\R^d)}.\qedhere
\end{equation*}
\end{theor1}

As the Lenard--Balescu equation satisfies an $H$-theorem, cf.~\eqref{eq:H-princ}, solutions are expected to relax to Maxwellian equilibrium and we indeed establish the following convergence result. The proof is inspired by corresponding previous work on the Landau equation, see in particular~\cite{Toscani-Villani-99,Toscani-Villani-00,Guo-Strain-08}.

\begin{theor1}[Convergence to equilibrium]\label{cor:conv}
Under the assumptions of Theorem~\ref{th:global}, given \mbox{$s\ge2$} and $0<\beta<\infty$, let $F=\mu_\beta+\sqrt{\mu_\beta}f$ be the constructed unique global solution of the Lenard--Balescu equation~\eqref{eq:LB}.
It converges to equilibrium in the sense that the relative entropy decays to zero at stretched exponential rate,
\begin{equation}\label{eq:conventropy}
H(F^t|\mu_\beta)\,:=\,\int_{\R^d} F^t\log(\tfrac1{\mu_\beta}F^t )\, dv~\lesssim_{V,\beta}~\exp\Big(-\tfrac1{C_{V,\beta}}t^\frac25\Big).
\end{equation}
In addition, we have
\begin{equation*}
f^t\to0\qquad\text{in $\Ld^2(\R^d)$}\qquad\text{as $t\uparrow\infty$,}
\end{equation*}
for which quantitative estimates hold in case of compact initial data:
\begin{enumerate}[(i)]
\item If $\int_{\R^d}\rv^\ell|f^\circ|^2<\infty$ for some $\ell>0$, and if the smallness condition~\eqref{eq:smallness-st} holds for some large enough $C_{V,\beta,s}$ (further depending on $\ell$),
then we have for all $\delta<1$,
\[\int_{\R^d}|f^t|^2\,\lesssim_{V,\beta,\ell,\delta}\,\rt^{-\delta\ell}\int_{\R^d}\rv^\ell|f^\circ|^2.\]
\item If $\int_{\R^d}e^{K\rv^\theta}|f^\circ|^2<\infty$ for some $0<\theta<2$ and $K>0$, or for $\theta=2$ and some small enough~$K>0$ (only depending on $V$),
and if the smallness condition~\eqref{eq:smallness-st} holds for some large enough $C_{V,\beta,s}$ (further depending on $\theta,K$),
then we have
\[\int_{\R^d}|f^t|^2\,\lesssim_{V,\beta,\theta,K}\,\exp\Big(-\tfrac{K}{C_{V,\beta,\theta}}t^\frac\theta{\theta+1}\Big)\int_{\R^d}e^{K\rv^\theta}|f^\circ|^2.\qedhere\]
\end{enumerate}
\end{theor1}

Next, we establish local well-posedness away from equilibrium.
The proof is surprisingly involved: while the linearization of the dispersion function disappears when linearizing the Lenard--Balescu equation at equilibrium (see also~\cite{Strain-07}), this algebraic miracle does not occur away from equilibrium and the linearized operator then involves a new nonlocal term with the highest number of derivatives. In view of the nonlocality, however, this contribution is shown to be in fact of lower order, cf.~e.g.~\eqref{eq:lemB-3} in Lemma~\ref{lem:B}.
For simplicity, we restrict here to dimension~$d>2$, but we believe that this restriction is not essential.
Note that we expect the statement to hold for all $s>\frac52$, but we restrict to integer differentiability to avoid additional technicalities. Strikingly enough, one derivative is lost in the control of the solution with respect to its initial data.
\begin{theor1}[Local well-posedness away from equilibrium]\label{th:local}
Let $d>2$, let $V\in\Ld^1\cap\dot H^2(\R^d)$ be isotropic and positive definite, and assume $xV\in\Ld^2(\R^d)$.
For all $s\ge3$ and $m> d+7$, for all nonnegative initial data $F^\circ\in\Ld^1(\R^d)$,
provided the following properties hold for some $M>0$ and $v_0\in \R^d$,
\begin{gather}\label{eq:as-F0}
\|\rv^{\frac m2}\langle\nabla\rangle^{s+1}F^\circ \|_{\Ld^2(\R^d)}\,\le\,M,\\
\inf_{k,v}|\e(k,k\cdot v;\nabla F^\circ)|\,\ge\,\tfrac1{M},\qquad
\inf_{|v-v_0|\le\frac1M}F^\circ(v)\ge\tfrac{1}{M},\nonumber
\end{gather}
there exist $T>0$ (depending on $V,M,\beta,s,m$) and a unique strong solution $F$ of the Lenard--Balescu equation~\eqref{eq:LB} on the interval $[0,T]$ with weighted Sobolev norm
\begin{equation*}
\|\rv^{\frac m2}\langle\nabla\rangle^{s}F^t\|_{\Ld^2(\R^d)}\,\lesssim_{V,M,s,m}\,1.\qedhere
\end{equation*}
\end{theor1}

We emphasize that the existence of {\it global} smooth solutions is an open question even in the Landau case, cf.~\cite[Chapter~5, \S1.3(2)]{Villani-02}; we refer to~\cite{golse2019partial,desvillettes2020new} for recent advances on the topic. For the Lenard--Balescu equation, getting beyond the above local-in-time result would even bring further difficulties due to the possible degeneracy of the dispersion function away from equilibrium. 

\medskip
In the plasma physics literature, e.g.~\cite{Villani-02,Alexandre-Villani-04}, the Lenard--Balescu equation~\eqref{eq:LB} is often approximated by the simpler Landau equation,
\begin{equation}\label{eq:Landau}
\partial_t F_L=Q_L(F_L),
\end{equation}
where we recall that the Landau operator takes the form
\begin{eqnarray*}
Q_L(F)&:=&\nabla\cdot\int_{\R^d}B_L(v-v_*)\,\big(F_*\nabla F-F\nabla_*F_*\big)\,dv_*,\\
B_L(v-v_*)&:=&\tfrac{L}{|v-v_*|}\Big(\Id-\tfrac{(v-v_*)\otimes(v-v_*)}{|v-v_*|^2}\Big).
\end{eqnarray*}
In case of 3D Coulomb interactions, this approximation is formally justified by the logarithmic divergence of the integral~\eqref{eq:LB-kernel} at large wavenumbers; see e.g.~\cite{Lenard-60,Balescu-63} and~\cite[Part~1]{Decoster}.
Indeed, a cut-off $|k|\le\frac1{\delta}$ is included in that case in~\eqref{eq:LB-kernel} to remove the large-$k$ logarithmic divergence, but for small Landau length $\delta$ the large-$k$ regime dominates in the integral, and therefore, as $\e(k,k\cdot v;\nabla F)\to1$ for $|k|\uparrow\infty$, one formally recovers the Landau kernel to leading order~$O(\log\frac1\delta)$ in view of~\eqref{eq:expl-Landau}.
The following result provides a rigorous version of this heuristics in the limit of short-range interactions, cf.~\eqref{eq:Vscale}.
While formulated for global solutions constructed in Theorem~\ref{th:global}, the same obviously holds in the local-in-time setting of Theorem~\ref{th:local}, where the existence time is necessarily uniform in $\delta$ after the relevant time-rescaling.
Note that the Landau equation can also be obtained from the Boltzmann equation in the grazing collision limit, cf.~\cite{Alexandre-Villani-04}.

\begin{theor1}[Landau approximation]\label{cor:Landau}
Let $V\in\Ld^1\cap\dot H^2(\R^d)$ be isotropic and positive definite, and assume $xV\in\Ld^2(\R^d)$.
Given an exponent~$a< d$, we consider the rescaled potentials
\begin{equation} \label{eq:Vscale}
V_{\delta}(x) \,:=\, \delta^{-a}V(\tfrac x{\delta}),\qquad\delta>0.
\end{equation}
Given $s\ge2$ and $0<\beta<\infty$, there is a constant $C_{V,\beta,s}$ large enough such that the following holds: for all $\delta>0$, for all initial data $F^\circ\in\Ld^1(\R^d)$ of the form
\[F^\circ=\mu_\beta+\sqrt{\mu_\beta}f^\circ\ge0,\qquad f^\circ\in H^s(\R^d),\]
satisfying smallness and centering conditions~\eqref{eq:smallness-st},
there exists a unique global strong solution $F_\delta$ of the corresponding Lenard--Balescu equation~\eqref{eq:LB} with potential $V_{\delta}$ and initial data $F^\circ$, in the form
\[F_{\delta}=\mu_\beta+\sqrt{\mu_\beta} f_{\delta}\ge0,\qquad f_{\delta}\in\Ld^\infty(\R^+;H^s(\R^d)).\]
In addition, up to time rescaling $\tilde f_{\delta}(t):=f_{\delta}(\delta^{2a+1-d}t)$, we have
\begin{align}
\tilde f_{\delta}~\overset*\rightharpoonup~ f_L \qquad \text{in $\Ld^\infty(\R^+;H^s(\R^d))$} \qquad \text{as $\delta \downarrow 0$},
\end{align}
where $F_L=\mu_\beta+\sqrt{\mu_\beta}f_L$ solves the Landau equation~\eqref{eq:Landau} with initial data $F^\circ$ and explicit prefactor~\eqref{eq:cst-L}.
\end{theor1}

As our main result in Theorem~\ref{th:global} is inspired by Guo's work on the Landau equation in~\cite{Guo-02}, let us briefly discuss the specific challenges faced in the present contribution on the Lenard--Balescu equation.
An obvious difficulty is to ensure the non-degeneracy of the dispersion function $\e(k,k\cdot v;\nabla F)$, which is achieved in Lemma~\ref{lem:eps-bndd-R} close to Maxwellian equilibrium. There are however two more severe difficulties arising from the presence of the nonlinearity $|\e(k,k\cdot v;\nabla F)|^2$ in the collision kernel~\eqref{eq:LB-kernel}. The first is a structural issue: the collision kernel is no longer of convolution type, which prevents for instance higher derivatives of the kernel from having improved decay (compare~\eqref{eq:smooth-A} with~\cite[Lemma~3]{Guo-02}). We compensate for this by making use of the specific tensor structure of the kernel (see e.g.\@ the proof of~\eqref{eq:hnabAh}). A second difficulty is related to the high order of the nonlinearity. This requires a fine analysis of critical nonlinear terms, which we express carefully in terms of Radon transforms (see e.g.~\eqref{eq:Tabbrref} and the proof of Lemma~\ref{lem:B}): our analysis allows both to extract additional decay and to prevent the loss of regularity. As already mentioned, this analysis becomes particularly intricate in the case away from equilibrium, cf.~Theorem~\ref{th:local}.

\medskip
Besides the presence of the dispersion function $\e(k,k\cdot v;\nabla F)$, there are two additional differences from the Landau setting in~\cite{Guo-02}. On one hand, we focus here on the spatially homogeneous setting, which removes many difficulties faced in~\cite{Guo-02}. On the other hand, we include the 2D case (except in~Theorem~\ref{th:local}), which was excluded in~\cite{Guo-02}, the difficulty being that the kernel singularity $O(|v-v_*|^{-1})$ does not belong to $\Ld^2_\loc$ in that case. This is overcome by carefully separating the critical terms and estimating them using standard tools for singular integrals such as Calder\'on--Zygmund theory and the Hardy--Littlewood--Sobolev inequality (see e.g.\@ proof of Lemma~\ref{lem:B0}).

\subsection*{Notation}
\begin{enumerate}[$\bullet$]
\item We denote by $C\ge1$ any constant that only depends on the space dimension $d$.
We use the notation $\lesssim$ (resp.~$\gtrsim$) for $\le C\times$ (resp.~$\ge\frac1C\times$) up to such a multiplicative constant~$C$. We write $\simeq$ when both $\lesssim$ and $\gtrsim$ hold. We add subscripts to $C,\lesssim,\gtrsim,\simeq$ to indicate dependence on other parameters.
\item We denote by $\dbar k:=(2\pi)^{-d}dk$ the rescaled Lebesgue measure on momentum space.
\item For $a,b\in\R$ we write $a\wedge b:=\min\{a,b\}$, $a\vee b:=\max\{a,b\}$, and $\langle a\rangle:=(1+a^2)^{1/2}$.
\end{enumerate}

\section{Global well-posedness close to equilibrium} \label{sec:global}
This section is devoted to the proof of Theorem~\ref{th:global}.
First, we argue that we can restrict attention to the case $\beta=1$ by a scaling argument.
Indeed, setting
$F=\beta^{d/2}\tilde F_{\beta}(\beta^{1/2}\cdot)$,
the Lenard--Balescu equation~\eqref{eq:LB} for $F$ is transformed into
\[\beta^{\frac12}\partial_t\tilde F_{\beta}=\nabla\cdot\int_{\R^d} B_\beta(v,v-v_*;\nabla \tilde F_{\beta})\,\big(\tilde F_{\beta,*}\nabla \tilde F_{\beta}-\tilde F_{\beta}\nabla_*\tilde F_{\beta,*}\big)\,dv_*,\]
in terms of
\begin{eqnarray*}
B_\beta(v,v-v_*;\nabla \tilde F_{\beta})&:=&\int_{\R^d}(k\otimes k)\,\pi(\beta\widehat V(k))^2\tfrac{\delta(k\cdot(v-v_*))}{|\e_\beta(k,k\cdot v;\nabla \tilde F_{\beta})|^2}\,\dbar k,\\
\e_\beta(k,k\cdot v;\nabla \tilde F_{\beta})&:=&1+\beta\widehat V(k)\int_{\R^d}\tfrac{k\cdot\nabla \tilde F_{\beta}(v_*)}{k\cdot(v-v_*)-i0}\,dv_*.
\end{eqnarray*}
This means that the dilation $\tilde F_{\beta}$ satisfies the same equation~\eqref{eq:LB} up to rescaling time and replacing the interaction potential $V$ by $\beta V$. It is therefore sufficient to prove Theorem~\ref{th:global} in the case $\beta=1$, and we henceforth drop the subscript for notational simplicity.

\medskip
\medskip
In terms of $F=\mu+\sqrt{\mu}f$, noting that $B(v,v-v_*;\nabla F)\,(v-v_*)=0$, the Lenard--Balescu equation~\eqref{eq:LB} can be reformulated as the following equation on $f$,
\begin{equation}\label{eq:g-reform-re}
\partial_tf=L[f]+N(f), 
\qquad f|_{t=0}=f^\circ,
\end{equation}
in terms of the linear and nonlinear operators
\begin{eqnarray*}
L[g]\!\!&:=&\!\!(\nabla- v)\cdot\int_{\R^d}B(v,v-v_*;\nabla  \mu)\,\Big(\sqrt\mu_*(\nabla+ v)g-\sqrt{\mu}((\nabla+v)g)_*\Big)\,\sqrt\mu_*\,dv_*,\\
N(g)\!\!&:=&\!\!(\nabla- v)\cdot\int_{\R^d}B(v,v-v_*;\nabla F_g)\,\big(g_*\nabla g-g(\nabla g)_*
\big)\,\sqrt\mu_*\,dv_*\\
&&+\,(\nabla- v)\cdot\int_{\R^d}\Big(B(v,v-v_*;\nabla F_g)-B(v,v-v_*;\nabla \mu)\Big)\\
&&\hspace{4cm}\times\Big(\sqrt\mu_*(\nabla+ v)g-\sqrt{\mu}((\nabla+v)g)_*\Big)\,\sqrt\mu_*\,dv_*,
\end{eqnarray*}
where we use the short-hand notation
\[F_g\,:=\,\mu+\sqrt{\mu}g.\]
In the spirit of~\cite{Guo-02}, our approach to global well-posedness exploits the peculiar properties of the linearized operator $L$, which we can further split as
\begin{equation}\label{eq:defL-reform}
L[g]\,=\,(\nabla-v)\cdot A(\nabla+v)g-(\nabla-v)\cdot\big(\sqrt\mu\,\Bc_\circ[(\nabla+v)g]\big),
\end{equation}
in terms of the elliptic coefficient field
\begin{equation}\label{eq:def-A}
A(v)\,:=\,\int_{\R^d}B(v,v-v_*;\nabla\mu)\,\mu_*\,dv_*,
\end{equation}
and the linear operator
\begin{equation}\label{eq:def-Bc0}
\Bc_\circ[g](v)\,:=\,\int_{\R^d}B(v,v-v_*;\nabla\mu)\,\sqrt\mu_*g_*\,dv_*.
\end{equation}
Similarly, the nonlinear operator $N$ can be split as
\begin{multline*}
N(g)\,=\,(\nabla-v)\cdot\Bc(\nabla F_g)[g]\,\nabla g-(\nabla-v)\cdot\big(g\,\Bc(\nabla F_g)[\nabla g]\big)\\
+(\nabla-v)\cdot\Big(\Bc(\nabla F_g)[\sqrt\mu]-\Bc(\nabla \mu)[\sqrt\mu]\Big)(\nabla+v)g\\
-(\nabla-v)\cdot\bigg(\sqrt\mu\Big(\Bc(\nabla F_g)[(\nabla+v)g]-\Bc(\nabla \mu)[(\nabla+v)g]\Big)\bigg),
\end{multline*}
where for all scalar fields $F$ we define the linear operator $\Bc(\nabla F)$ as
\begin{equation}\label{eq:def-Bc}
\Bc(\nabla F)[g](v)\,:=\,\int_{\R^d}B(v,v-v_*;\nabla F)\,\sqrt\mu_*g_*\,dv_*.
\end{equation}
Note that with our notation,
\[A=\Bc_\circ[\sqrt\mu],\qquad\Bc_\circ[g]=\Bc(\nabla\mu)[g].\]
We emphasize that in the Landau case~\eqref{eq:expl-Landau} the operators $\Bc_\circ$ and $\Bc(\nabla F)$ coincide for all~$F$ and are of convolution type, which is not the case here and makes the Lenard--Balescu setting substantially more involved.

\medskip
Before investigating properties of $L$, our starting point is the following key lemma, stating that the dispersion function $\e(k,k\cdot v;\nabla F)$ is uniformly non-degenerate provided that the density $F$ is close enough to Maxwellian in a suitable Sobolev sense. This result is crucial to the perturbative handling of dynamical screening and for the comparison to the Landau case~\eqref{eq:expl-Landau}. Note that this is in sharp contrast with the degenerate behavior of the dispersion function in case of Coulomb interactions, cf.~\cite{Strain-07}.

\begin{lem}[Lenard--Balescu dispersion function]\label{lem:eps-bndd-R}
Let $V\in\Ld^1(\R^d)$ be positive definite.
\begin{enumerate}[(i)]
\item \emph{Non-degeneracy at Maxwellian:} For all $k,v\in\R^d$,
\[|\e(k,k\cdot v;\nabla\mu)|\,\simeq_V\,1.\]
\item \emph{Non-degeneracy close to Maxwellian:} Provided $g\in\Ld^2(\R^d)$ satisfies the following smallness condition, for some $r_0\ge0$, $\delta_0>0$, and some large enough constant $C_0$,
\[\|\rv^{-r_0}\rnabla^{\frac32+\delta_0}g\|_{\Ld^2(\R^d)}\le\tfrac1{C_0},\]
we have for all $k,v\in\R^d$,
\[|\e(k,k\cdot v;\nabla F_g)|\,\simeq_{V,\delta_0,r_0}\,1.\]
\item \emph{Boundedness:}
For all multi-indices $\alpha>0$, for all $\delta>0$, and $r\ge0$, we have
\[|\nabla_v^\alpha\e(k,k\cdot v;\nabla F_g)|\,\lesssim_{V,\alpha,\delta,r}\, 1+\|\rv^{-r}\rnabla^{|\alpha|+\frac32+\delta}g\|_{\Ld^2(\R^d)}.\qedhere\]
\end{enumerate}
\end{lem}

\begin{proof}
We split the proof into three steps.

\medskip
\step1 Proof of~(i).\\
Setting $\hat k:=\frac{k}{|k|}$, we have by definition, cf.~\eqref{eq:LB-disp}
\begin{equation*}
\e(k,k\cdot v;\nabla\mu)\,=\,1-2\widehat V(k)\int_{\R^d}\tfrac{\hat k\cdot v_*}{\hat k\cdot(v-v_*)-i0}\,\mu(v_*)\,dv_*,
\end{equation*}
and the Plemelj formula yields
\begin{multline*}
\e(k,k\cdot v;\nabla\mu)\,=\,1-2\widehat V(k)\pv\int_{\R^d}\tfrac{\hat k\cdot v_*}{\hat k\cdot(v-v_*)}\,\mu(v_*)\,dv_*\\
-2i\pi\widehat V(k)\int_{\R^d}(\hat k\cdot v_*)\,\delta(\hat k\cdot(v-v_*))\,\mu(v_*)\,dv_*.
\end{multline*}
Splitting integrals over $v_*\in\R^d$ as integrals over $v_*\in\hat k\R\oplus\hat k^\bot$, setting $v_k:=\hat k\cdot v$, and noting that $\int_{\hat k^\bot}\mu=(\frac1\pi)^\frac12$, this can be rewritten as
\begin{equation*}
\e(k,k\cdot v;\nabla\mu)\,=\,1+2\widehat V(k)\Big(1-(\tfrac1\pi)^{\frac12}v_k\pv\int_{\R}\tfrac{1}{v_k-y}\,e^{- y^2}dy\Big)
-2i\pi \widehat V(k)(\tfrac1\pi)^{\frac12}v_ke^{- v_k^2}.
\end{equation*}
Using the integral formula
\[\pv\int_\R\tfrac{1}{x-y}\,e^{- y^2}\,dy\,=\,2\pi^\frac12e^{- x^2}\int_0^{x}e^{ y^2}dy,\]
and further setting for abbreviation
\[w_{k}:=\sqrt2\, v_k,\qquad\text{and}\qquad H(x)\,:=\,\frac{\frac1{x}e^{\frac12x^2}-\int_0^{x}e^{
\frac12y^2}dy}{\frac1{x}e^{\frac12x^2}},\]
the above becomes
\begin{equation*}
\e(k,k\cdot v;\nabla\mu)\,=\,1+2\widehat V(k)H(w_{k})
-i(2\pi)^\frac12 \widehat V(k)w_{k}e^{- \frac12w_{k}^2}.
\end{equation*}
Taking the modulus, we then find
\begin{equation*}
|\e(k,k\cdot v;\nabla\mu)|^2\,=\,\big(1+2\widehat V(k)H(w_{k})\big)^2
+2\pi\widehat V(k)^2w_{k}^2e^{-w_{k}^2}.
\end{equation*}
Noting that
\[\begin{array}{ll}
H(x)\ge0,&\text{for $|x|\le1$},\\
H(x)\ge-\frac2{|x|}-|x|e^{2-\frac12x^2},&\text{for $|x|\ge1$},
\end{array}\]
we easily deduce for all $k,v\in\R^d$,
\[e^{-C\widehat V(k)^2}\,\lesssim\,|\e(k,k\cdot v;\nabla \mu)|^2\,\lesssim\,1+\widehat V(k)^2.\]

\medskip
\step2 Proof of~(ii).\\
By definition, with $F_g=\mu+\sqrt\mu g$, we find
\[|\e(k,k\cdot v;\nabla F_g)|\,\ge\,|\e(k,k\cdot v;\nabla \mu)|-\widehat V(k)\Big|\int_{\R^d}\tfrac{\hat k\cdot\nabla_*(\sqrt\mu g)_*}{\hat k\cdot(v-v_*)-i0}\,dv_*\Big|.\]
By the Sobolev inequality, we deduce for all $\delta_0>0$,
\[|\e(k,k\cdot v;\nabla f)|\,\ge\,|\e(k,k\cdot v;\nabla\mu)|-C_{\delta_0}\|V\|_{\Ld^1(\R^d)}\Big\|\int_{\R^d}\tfrac{\hat k\cdot\nabla_*(\sqrt\mu g)_*}{\hat k\cdot(\cdot- v_*)-i0}\,dv_*\Big\|_{H^{\frac12+\delta_0}(\hat k\R)}.\]
Thus, splitting again integrals over $v_*\in\R^d$ as integrals over $v_*\in\hat k\R\oplus\hat k^\bot$, appealing to the $\Ld^2$ boundedness of the Hilbert transform, and using the fast decay of~$\mu$, we get for all~$r\ge0$,
\begin{eqnarray*}
|\e(k,k\cdot v;\nabla f)|&\ge&|\e(k,k\cdot v;\nabla\mu)|-C_{V,\delta_0}\Big\|\int_{\hat k^\bot}\hat k\cdot\nabla(\sqrt\mu g)\Big\|_{H^{\frac12+\delta_0}(\hat k\R)}\\
&\ge&|\e(k,k\cdot v;\nabla\mu)|-C_{V,\delta_0,r}\|\rv^{-r}\rnabla^{\frac32+\delta_0}g\|_{\Ld^2(\R^d)}.
\end{eqnarray*}
Combined with~(i), this proves~(ii).

\medskip
\step3 Proof of~(iii).\\
By definition, cf.~\eqref{eq:LB-disp}, we find for $|\alpha|\ge1$,
\begin{equation}\label{eq:diff-eps-form}
\nabla^\alpha_v\e(k,k\cdot v;\nabla F_g)\,=\,\widehat V(k)\int_{\R^d}\tfrac{\hat k\cdot\nabla_*\nabla_*^\alpha F_g(v_*)}{\hat k\cdot(v-v_*)-i0}dv_*.
\end{equation}
Arguing as in Step~2,
the conclusion~(iii) follows from the Sobolev inequality and from the boundedness of the Hilbert transform.
\end{proof}

We now turn to properties of the linearized operator $L$, cf.~\eqref{eq:g-reform-re}, and we start with the following estimates on the elliptic coefficient field $A$, cf.~\eqref{eq:def-A}. In particular, note that it is not uniformly elliptic.
From a similar argument as for Lemma~\ref{lem:B0}(ii) below, we could in fact show that in~\eqref{eq:smooth-A} higher-order derivatives have stronger decay (yet under stronger assumptions on $V$), but it is not needed in this work.

\begin{lem}[Properties of $A$]\label{lem:A}
Let $V\in\Ld^1\cap\dot H^\frac12(\R^d)$ be isotropic and positive definite.
\begin{enumerate}[(i)]
\item \emph{Coercivity and boundedness:} For all $v,e\in\R^d$,
\begin{equation*}
e\cdot A(v)e \,\simeq_V\,\langle v\rangle^{-1} |P_v^\bot e|^2+\langle v\rangle^{-3}|P_ve|^2,
\end{equation*}
in terms of orthogonal projections $P_v$ and $P_v^\bot$ on $v\R$ and $v^\bot$,
\begin{equation}\label{eq:def-projv}
P_v:=\tfrac{v}{|v|}\otimes\tfrac{v}{|v|},\qquad P^\bot_v=\Id-P_v.
\end{equation}
\item \emph{Smoothness:}
$A$ belongs to $C^\infty_b(\R^d)$ and satisfies for all $v\in\R^d$ and $\alpha\ge0$,
\begin{equation}\label{eq:smooth-A}
|\nabla^\alpha A(v)|\,\lesssim_{V,\alpha}\,\langle v\rangle^{-1},\qquad |\nabla^\alpha (A(v)  v)| \,\lesssim_{V,\alpha}\, \langle v\rangle^{-2}.
\end{equation}
In particular, for all vector fields $h_1,h_2$, and $\alpha\ge0$,
\begin{equation}\label{eq:hnabAh}
\Big|\int_{\R^d}h_1\cdot(\nabla^\alpha A)h_2\Big|\,\lesssim_{V,\alpha}\,\Big(\int_{\R^d}h_1\cdot A\,h_1\Big)^\frac12\Big(\int_{\R^d}h_2\cdot A\,h_2\Big)^\frac12.\qedhere
\end{equation}
\end{enumerate}
\end{lem}

Before turning to the proof of this lemma, we note that it motivates the definition of the following weighted norm, which is adapted to the dissipation structure of the operator~$L$: for any vector field $h$,
\begin{equation}\label{eq:def-L2A}
\|h\|_{\Ld^2_A(\R^d)}\,:=\,\Big(\int_{\R^d}h\cdot A \,h\Big)^\frac12.
\end{equation}
Lemma~\ref{lem:A}(i) above states that this norm is equivalent to the following,
\begin{equation}\label{eq:equiv-L2A}
\|h\|_{\Ld^2_A(\R^d)}\,\simeq_V\,\|\rv^{-\frac12}P_v^\bot h\|_{\Ld^2(\R^d)}+\|\rv^{-\frac32}P_v h\|_{\Ld^2(\R^d)},
\end{equation}
and~\eqref{eq:hnabAh} can now be rewritten as
\begin{equation*}
\Big|\int_{\R^d}h_1\cdot(\nabla^\alpha A)h_2\Big|\,\lesssim_{V,\alpha}\,\|h_1\|_{\Ld^2_A(\R^d)}\|h_2\|_{\Ld^2_A(\R^d)}.
\end{equation*}

\begin{proof}[Proof of Lemma~\ref{lem:A}]
We split the proof into four steps.

\medskip
\step1 Proof of~\eqref{eq:expl-Landau}, that is,
\begin{equation}\label{eq:expl-Landau-pr}
\int_{\R^d}(k\otimes k)\,\pi\widehat V(k)^2\,\delta(k\cdot v)\,\dbar k
\,=\,\tfrac{L}{|v|}P_{v}^\bot,
\end{equation}
where we use the notation~\eqref{eq:def-projv} and where the constant $L$ is defined in~\eqref{eq:cst-L}.

\medskip\noindent
The integral is computed as
\begin{eqnarray*}
\int_{\R^d}(k\otimes k)\,\pi\widehat V(k)^2\,\delta(k\cdot v)\,\dbar k
&=&\tfrac1{|v|}\int_{\R^d}(k\otimes k)\,\pi\widehat V(k)^2\,\delta(k\cdot \tfrac{v}{|v|})\,\dbar k\\
&=&\tfrac{1}{|v|}(2\pi)^{-d}\int_{v^\bot}(k\otimes k)\,\pi\widehat V(k)^2\,dk.
\end{eqnarray*}
By symmetry, as $\widehat V$ is isotropic, this integral is a multiple of the orthogonal projection~$P^\bot_{v}$.
This proves~\eqref{eq:expl-Landau-pr} with multiplicative constant
\begin{eqnarray*}
L&:=&\tfrac1{d-1}\Tr\Big((2\pi)^{-d}\int_{v^\bot}(k\otimes k)\,\pi\widehat V(k)^2\,dk\Big)\\
&=&\tfrac1{d-1}(2\pi)^{-d}\int_{v^\bot}|k|^2\,\pi\widehat V(k)^2\,dk,
\end{eqnarray*}
which can be rewritten as follows, using radial coordinates and the isotropy of the integrand,
\begin{eqnarray*}
L&=&\tfrac1{d-1}|\Sp^{d-2}|(2\pi)^{-d}\int_0^\infty r^d\pi\widehat V(r)^2\,dr\\
&=&\tfrac1{d-1}\tfrac{|\Sp^{d-2}|}{|\Sp^{d-1}|}(2\pi)^{-d}\int_{\R^d} |k|\pi\widehat V(k)^2\,dk.
\end{eqnarray*}
Using $|\Sp^{d-2}|=(d-1)\omega_{d-1}$ and $|\Sp^{d-1}|=d\omega_d$, we recover the definition~\eqref{eq:cst-L} of $L$, which is finite provided $V\in\dot H^{1/2}(\R^d)$.

\medskip
\step2 Proof of~(i).\\
By definition of $A$ and of the collision kernel $B$, cf.~\eqref{eq:def-A} and~\eqref{eq:LB-kernel}, we note that~$A(v)$ is a positive definite symmetric matrix for all $v\in\R^d$, with
\begin{equation}\label{eq:eAe}
e\cdot A(v)e\,=\,\iint_{\R^d\times\R^d}(e\cdot k)^2\,\pi\widehat V(k)^2\tfrac{\delta(k\cdot(v-v_*))}{|\e(k,k\cdot v;\nabla\mu)|^2}\,\dbar k\,\mu_*\,dv_*\,\ge\,0.
\end{equation}
Further recalling the definition of the dispersion function $\e$, cf.~\eqref{eq:LB-disp}, and using that $V$ and~$\mu$ are both isotropic, we note that $A$ satisfies $A(Rv)=RA(v)R'$ for all orthogonal transformations~$R\in O_d(\R)$.
From this, we can deduce that for all $v\in\R^d$ the spectrum of~$A(v)$ consists of a simple eigenvalue $\lambda_1(v)$ associated with the eigenvector $v$, and a multiple eigenvalue~$\lambda_2(v)$ associated with the eigenspace $v^\bot$. In other words, in terms of projections $P_v,P_v^\bot$, cf.~\eqref{eq:def-projv}, we have
\begin{equation}\label{eq:Av-eigen}
A(v)\,=\,\lambda_1(v)P_v+\lambda_2(v)P_v^\bot,
\end{equation}
where in view of~\eqref{eq:eAe} the eigenvalues are computed as
\begin{eqnarray*}
\lambda_1(v)&:=&\iint_{\R^d\times\R^d}\big(k\cdot \tfrac{v}{|v|}\big)^2\,\pi\widehat V(k)^2\tfrac{\delta(k\cdot(v-v_*))}{|\e(k,k\cdot v;\nabla\mu)|^2}\,\dbar k\,\mu_*\,dv_*,\\
\lambda_2(v)&:=&\tfrac1{d-1}\big(\Tr A(v)-\lambda_1(v)\big).
\end{eqnarray*}
It remains to establish asymptotic estimates for these eigenvalues.
Using Lemma~\ref{lem:eps-bndd-R}(i) and the explicit computation~\eqref{eq:expl-Landau-pr}, the identity~\eqref{eq:eAe} becomes for all $v,e\in\R^d$,
\begin{eqnarray}
e\cdot A(v)e
&\simeq_V&\iint_{\R^d\times\R^d}|e\cdot k|^2\,\pi\widehat V(k)^2\,\delta(k\cdot(v-v_*))\,\dbar k\,\mu_*\,dv_*\nonumber\\
&\simeq_V&e\cdot\bigg(\int_{\R^d}\tfrac{1}{|v-v_*|}\,\Big(\Id-\tfrac{(v-v_*)\otimes(v-v_*)}{|v-v_*|^2}\Big)\,\mu_*\,dv_*\bigg)e,\label{eq:A-Landau}
\end{eqnarray}
that is, we are essentially reduced to the Landau case,
and the above formulas for eigenvalues then yield, as e.g.~in~\cite[Section~2]{Degond-Lemou-97},
\begin{eqnarray*}
\lambda_1(v)&\simeq_V&\int_{\R^d}\tfrac1{|v_*|}\Big(1-\big(\tfrac{v\cdot v_*}{|v||v_*|}\big)^2\Big)\mu(v-v_*)\,dv_*,\\
\lambda_2(v)&\simeq_V&\tfrac1{d-1}\int_{\R^d}\tfrac1{|v_*|}\Big(d-2+\big(\tfrac{v\cdot v_*}{|v||v_*|}\big)^2\Big)\mu(v-v_*)\,dv_*.
\end{eqnarray*}
Using the fast decay of $\mu$, we get
\[\lambda_1(v)\simeq_V\langle v\rangle^{-3},\qquad \lambda_2(v)\simeq_V\langle v\rangle^{-1},\]
and the conclusion~(i) follows.

\medskip
\step3 Proof of~\eqref{eq:smooth-A}.\\
Estimates $|A(v)|\lesssim\langle v\rangle^{-1}$ and $|A(v)v|\lesssim\langle v\rangle^{-2}$ follow from~(i).
Let now $|\alpha|\ge1$.
By definition of~$A$ and of the collision kernel $B$, cf.~\eqref{eq:def-A} and~\eqref{eq:LB-kernel},
Leibniz' rule yields
\begin{multline*}
\nabla^\alpha A(v)
\,=\,\sum_{\gamma\le\alpha}\binom\alpha\gamma\iint_{\R^d\times\R^d}(k\otimes k)\,\pi\widehat V(k)^2\,\delta(k\cdot(v-v_*))\nabla^\gamma\big(\tfrac1{|\e(k,k\cdot v;\nabla \mu)|^2}\big)\dbar k\\
\times(\nabla^{\alpha-\gamma}\mu)_*\,dv_*,
\end{multline*}
hence, by Lemma~\ref{lem:eps-bndd-R},
\begin{eqnarray}
|\nabla^\alpha A(v)|
&\lesssim_{V,\alpha}&\iint_{\R^d\times\R^d}|k|^2\widehat V(k)^2\delta(k\cdot(v-v_*))\,\dbar k\,\rvs^{|\alpha|}\mu_*\,dv_*\nonumber\\
&\lesssim_{V}&\int_{\R^d}|v-v_*|^{-1}\rvs^{|\alpha|}\mu_*\,dv_*\nonumber\\
&\lesssim_{\alpha}&\langle v\rangle^{-1}.\label{eq:pr-bnd-diffalphaA}
\end{eqnarray}
Next, we turn to derivatives of $A(v)v$: in view of~\eqref{eq:Av-eigen}, we find
\begin{equation*}
A(v)v\,=\,\lambda_1(v)v\,=\,\tfrac{v}{|v|^2}\iint_{\R^d\times\R^d}(k\cdot v)^2\,\pi\widehat V(k)^2\tfrac{\delta(k\cdot(v-v_*))}{|\e(k,k\cdot v;\nabla\mu)|^2}\dbar k\,\mu_*\,dv_*,
\end{equation*}
or alternatively, using the Dirac mass $\delta(k\cdot(v-v_*))$ to replace the factor $(k\cdot v)^2$ by $(k\cdot v_*)^2$,
\begin{equation*}
A(v)v\,=\,\tfrac{v}{|v|^2}\iint_{\R^d\times\R^d}\,\pi\widehat V(k)^2\tfrac{\delta(k\cdot(v-v_*))}{|\e(k,k\cdot v;\nabla\mu)|^2}\dbar k\,(k\cdot v_*)^2\mu_*\,dv_*.
\end{equation*}
Leibniz' rule then yields
\begin{multline*}
\nabla^\alpha(A(v)v)\,=\,\sum_{\gamma+\gamma'\le\alpha}\binom{\alpha}{\gamma,\gamma'}\nabla^\gamma(\tfrac{v}{|v|^2})\iint_{\R^d\times\R^d}\pi\widehat V(k)^2\delta(k\cdot(v-v_*))\nabla^{\gamma'}\big(\tfrac1{|\e(k,k\cdot v;\nabla\mu)|^2}\big)\,\dbar k\\
\times\nabla_*^{\alpha-\gamma-\gamma'}\big((k\cdot v)^2\mu\big)_*\,dv_*,
\end{multline*}
hence, bounding the last factor by $|\nabla_*^{\alpha-\gamma-\gamma'}((k\cdot v)^2\mu)_*|\lesssim_{\alpha}|k|^2\rvs^{|\alpha|+2}$ and appealing to Lemma~\ref{lem:eps-bndd-R},
\begin{equation*}
|\nabla^\alpha(A(v)v)|\,\lesssim_{V,\alpha}\,\sum_{\gamma+\gamma'\le\alpha}|v|^{-|\gamma|-1}\iint_{\R^d\times\R^d}|k|^2\widehat V(k)^2\delta(k\cdot(v-v_*))\,\dbar k\,
\rvs^{|\alpha|+2}\mu_*\,dv_*,
\end{equation*}
which yields as in~\eqref{eq:pr-bnd-diffalphaA}, for $|v|>1$,
\[|\nabla^\alpha(A(v)v)|\,\lesssim_{V,\alpha}\,|v|^{-2}\]
Since incidentally~\eqref{eq:pr-bnd-diffalphaA} entails $|\nabla^\alpha(A(v)v)|\lesssim_\alpha1$ for $|v|\le1$, the conclusion~\eqref{eq:smooth-A} follows.

\medskip
\substep{4} Proof of~\eqref{eq:hnabAh}.\\
In view of~(i), it suffices to prove for all $\alpha\ge0$,
\begin{multline}\label{eq:hnabAh-re}
\Big|\int_{\R^d}h_1\cdot(\nabla^\alpha A)h_2\Big|\,\lesssim_{V,\alpha}\,\Big(\|\rv^{-\frac12}P_v^\bot h_1\|_{\Ld^2(\R^d)}+\|\rv^{-\frac32}P_v h_1\|_{\Ld^2(\R^d)}\Big)\\
\times\Big(\|\rv^{-\frac12}P_v^\bot h_2\|_{\Ld^2(\R^d)}+\|\rv^{-\frac32}P_v h_2\|_{\Ld^2(\R^d)}\Big).
\end{multline}
We start by decomposing the vector fields $h_1,h_2$ with respect to $P_v,P_v^\bot$,
\begin{multline*}
\int_{\R^d}h_1\cdot(\nabla^\alpha A)h_2\,=\,
\int_{\R^d}(P_v^\bot h_1)\cdot(\nabla^\alpha A)(P_v^\bot h_2)
+\int_{\R^d}(P_v h_1)\cdot(\nabla^\alpha A)(P_v h_2)\\
+\int_{\R^d}(P_v^\bot h_1)\cdot(\nabla^\alpha A)(P_vh_2)
+\int_{\R^d}(P_vh_1)\cdot(\nabla^\alpha A)(P_v^\bot h_2).
\end{multline*}
By definition of $P_v$, we can write
\[(\nabla^\alpha A)(P_vh_1)\,=\,\big(\tfrac{v\cdot h_1}{|v|^2}\big)(\nabla^\alpha A)v\,=\,\big(\tfrac{v\cdot h_1}{|v|^2}\big)\Big(\nabla^\alpha(Av)-\sum_{j\in\alpha}(\nabla^{\alpha-j} A)e_j\Big),\]
and similarly
\begin{multline*}
(P_vh_1)\cdot(\nabla^\alpha A)(P_vh_2)\,=\,\big(\tfrac{v\cdot h_1}{|v|^2}\big)\big(\tfrac{v\cdot h_2}{|v|^2}\big)\Big(v\cdot \nabla^\alpha(Av)\\
-\sum_{j\in\alpha}e_j\cdot\nabla^{\alpha-j}(Av)
+\sum_{j,l\in\alpha}e_j\cdot(\nabla^{\alpha-j-l}A)e_l\Big).
\end{multline*}
By~\eqref{eq:smooth-A}, we deduce
\begin{multline*}
\Big|\int_{\R^d}h_1\cdot(\nabla^\alpha A)h_2\Big|\,\lesssim_{V,\alpha}\,
\int_{\R^d}\rv^{-1}|P_v^\bot h_1||P_v^\bot h_2|
+\int_{\R^d}\rv^{-3}|P_v h_1||P_v h_2|\\
+\int_{\R^d}\rv^{-2}|P_v^\bot h_1||P_vh_2|
+\int_{\R^d}\rv^{-2}|P_v h_1||P_v^\bot h_2|,
\end{multline*}
and the claim~\eqref{eq:hnabAh-re} follows.
\end{proof}

Next, we study boundedness and regularity properties of the linear operator~$\Bc_\circ$, cf.~\eqref{eq:def-Bc0}.
Item~(ii) below expresses a crucial gain of differentiability. As is clear from the proof, this gain could actually be improved in higher dimensions $d>2$ by iterating our argument below (yet under stronger assumptions on $V$), but it is not needed in this work.
We emphasize that the operator $\Bc_\circ$ is not of convolution type,
which is an important difference from the Landau case~\eqref{eq:expl-Landau}.

\begin{lem}[Properties of $\Bc_\circ$]\label{lem:B0}
Let $V\in\Ld^1\cap \dot H^\frac12(\R^d)$ be isotropic and positive definite.
\begin{enumerate}[(i)]
\item\emph{Boundedness:} For all $r\ge0$,
\begin{equation*}
\|\rv^r\sqrt{\mu}\Bc_\circ[g]\|_{\Ld^2(\R^d)}
\,\lesssim_{V,r}\, \|\rv^{-r}g\|_{\Ld^2(\R^d)},
\end{equation*}
\item\emph{Improved regularity:} Further assume $V\in\dot H^2(\R^d)$ and $xV\in\Ld^2(\R^d)$.
Then, for all~$\alpha>0$ and~$r\ge0$,
\begin{equation*}
\|\rv^r\sqrt{\mu}\nabla^\alpha\Bc_\circ[g]\|_{\Ld^2(\R^d)}
\,\lesssim_{V,\alpha,r}\,\sum_{\gamma<\alpha}\|\rv^{-r}\nabla^\gamma g\|_{\Ld^2(\R^d)}.
\qedhere
\end{equation*}
\end{enumerate}
\end{lem}

\begin{proof}
We split the proof into two steps.

\medskip
\step1 Proof of~(i).\\
By definition of~$\Bc_\circ$ and of the collision kernel $B$, cf.~\eqref{eq:def-Bc0} and~\eqref{eq:LB-kernel}, Lemma~\ref{lem:eps-bndd-R}(i) yields
\begin{equation}\label{eq:homogvvst-1}
|B(v,v-v_*;\nabla \mu)|\,\lesssim_V\,\int_{\R^d}|k|^2\widehat V(k)^2\delta(k\cdot(v-v_*))\,\dbar k\,\lesssim\,|v-v_*|^{-1},
\end{equation}
hence, by the Hardy--Littlewood--Sobolev inequality, for all $d<q<\infty$ and $r\ge0$,
\[\|\Bc_\circ[g]\|_{\Ld^q(\R^d)}~\lesssim_q~\|\sqrt{\mu}g\|_{\Ld^{\frac{dq}{d+q(d-1)}}(\R^d)}.\]
Combined with the fast decay of $\sqrt\mu$, this proves~(i).

\medskip
\step2 Proof of~(ii).\\
By Leibniz' rule, we can decompose
\begin{equation}\label{eq:diff-B0-decomp}
\nabla^\alpha\Bc_\circ[g]\,=\,\sum_{\gamma\le\alpha}\binom\alpha\gamma T^\alpha_{\gamma}[g],
\end{equation}
in terms of
\begin{equation*}
T^\alpha_{\gamma}[g]\,:=\,\iint_{\R^d\times\R^d}(k\otimes k)\,\pi\widehat V(k)^2\,\delta(k\cdot(v-v_*))
\nabla^{\alpha-\gamma}_v\big(\tfrac1{|\e(k,k\cdot v;\nabla \mu)|^2}\big)\,\dbar k\,\nabla_*^{\gamma}(\sqrt\mu_*g_*)\,dv_*.
\end{equation*}
The conclusion~(ii) follows provided we show for all $\gamma\le\alpha$ and $r\ge0$,
\begin{equation}\label{eq:L2bnd-diff-B0-re}
\|\rv^r\sqrt\mu\, T^\alpha_\gamma[g]\|_{\Ld^2(\R^d)}\,\lesssim_{V,\alpha,r}\,\sum_{\gamma<\alpha}\|\rv^{-r}\nabla^\gamma g\|_{\Ld^2(\R^d)}.
\end{equation}
We split the proof into three further substeps. Note that the case of dimension $d=2$ is critical and requires a particular care below.

\medskip
\substep{2.1} Proof of~\eqref{eq:L2bnd-diff-B0-re} for $\gamma<\alpha$.\\
By Lemma~\ref{lem:eps-bndd-R}, we get
\begin{eqnarray*}
|T^\alpha_{\gamma}[g]|&\lesssim_{V,\alpha}&\iint_{\R^d\times\R^d}|k|^2\widehat V(k)^2\,\delta(k\cdot(v-v_*))\dbar k\,|\nabla_*^{\gamma}(\sqrt\mu_*g_*)|\,dv_*\\
&\lesssim_V&\int_{\R^d}|v-v_*|^{-1}|\nabla_*^{\gamma}(\sqrt\mu_*g_*)|\,dv_*,
\end{eqnarray*}
and~\eqref{eq:L2bnd-diff-B0-re} follows as in Step~1 provided $\gamma<\alpha$.

\medskip
\substep{2.2} Proof of~\eqref{eq:L2bnd-diff-B0-re} for $\gamma=\alpha$ in case $d>2$.\\
Given $e_j\le\alpha$, successive integrations by parts yield
\begin{eqnarray*}
T^\alpha_{\alpha}[g]
&=&
\iint_{\R^d\times\R^d}\tfrac{k_j(k\otimes k)\,\pi\widehat V(k)^2}{|\e(k,k\cdot v;\nabla \mu)|^2}\delta'(k\cdot(v-v_*))\,\dbar k\,\nabla_*^{\alpha-e_j}(\sqrt\mu_*g_*)\,dv_*\\
&=&
-\sum_l\iint_{\R^d\times\R^d}\tfrac{(v-v_*)_l}{|v-v_*|^2}\,\nabla_{k_l}\Big(\tfrac{k_j(k\otimes k)\pi\widehat V(k)^2}{|\e(k,k\cdot v;\nabla \mu)|^2}\Big)\delta(k\cdot(v-v_*))\,\dbar k\,\nabla_*^{\alpha-e_j}(\sqrt\mu_*g_*)\,dv_*.
\end{eqnarray*}
Expanding the $k$-gradient in the last term, and noting that the $k$-gradient of the dispersion function can be written as
\begin{equation}\label{eq:form-nabk-eps}
\nabla_{k}\big(\e(k,k\cdot v;\nabla \mu)\big)\,=\,R(k,k\cdot v)+r(k,k\cdot v)v,
\end{equation}
in terms of
\begin{eqnarray*}
R(k,k\cdot v)&:=&\nabla\widehat V(k)\int_{\R^d}\tfrac{k\cdot\nabla\mu_*}{k\cdot(v-v_*)-i0}\,dv_*
+\widehat V(k)\big(\Id-\tfrac{k\otimes k}{|k|^2}\big)\int_{\R^d}\tfrac{\nabla \mu_*}{k\cdot(v-v_*)-i0}\,dv_*\\
&&\hspace{2cm}-\tfrac{\widehat V(k)}{|k|^{2}}\int_{\R^d}\tfrac{(k\otimes k):\nabla^2 \mu_*}{k\cdot(v-v_*)-i0}\,v_*\,dv_*,\\
r(k,k\cdot v)&:=&\tfrac{\widehat V(k)}{|k|^{2}}\int_{\R^d}\tfrac{(k\otimes k):\nabla^2 \mu_*}{k\cdot(v-v_*)-i0}\,dv_*,
\end{eqnarray*}
we are led to
\begin{multline}\label{eq:represent-Talphalph}
T_{\alpha}^\alpha[g]\,=\,
\int_{\R^d}K_j^0(v,v-v_*)\nabla_*^{\alpha-j}(\sqrt\mu_* g_*)\,dv_*\\
+\sum_lv_l\int_{\R^d}K_{j,l}^1(v,v-v_*)\nabla_*^{\alpha-j}(\sqrt\mu_* g_*)\,dv_*,
\end{multline}
where the kernels $K^0,K^1$ are defined by
\begin{eqnarray*}
K_j^0(v,w)&:=&-\sum_l\tfrac{w_l}{|w|^2}\int_{\R^d}\tfrac{\nabla_{k_l}(k_j(k\otimes k)\pi\widehat V(k)^2)}{|\e(k,k\cdot v;\nabla\mu)|^2}\delta(k\cdot w)\,\dbar k\\
&&+\,2\sum_l\tfrac{w_l}{|w|^2}\int_{\R^d}\tfrac{k_j(k\otimes k)\pi\widehat V(k)^2}{|\e(k,k\cdot v;\nabla\mu)|^4}\Re\big(\overline{\e(k,k\cdot v;\nabla\mu)}R_l(k,k\cdot v)\big)\,\delta(k\cdot w)\,\dbar k,\\
K_{j,l}^1(v,w)&:=&2\,\tfrac{w_l}{|w|^2}\int_{\R^d}\tfrac{k_j(k\otimes k)\pi\widehat V(k)^2}{|\e(k,k\cdot v;\nabla\mu)|^4}\Re\big(\overline{\e(k,k\cdot v;\nabla\mu)}r(k,k\cdot v)\big)\,\delta(k\cdot w)\,\dbar k.
\end{eqnarray*}
By Lemma~\ref{lem:eps-bndd-R}, as the integrability assumptions on $V$ ensure $\int_{\R^d}|k|^2|\nabla\widehat V(k)|\widehat V(k)\,\dbar k<\infty$, a direct estimate yields
\begin{equation}\label{eq:est-K0K1-sing}
|K_j^0(v,w)|+|K^1_{j,l}(v,w)|\,\lesssim_V\,|w|^{-2}.
\end{equation}
In case $d>2$, the Hardy--Littlewood--Sobolev inequality then gives for all $\frac d2<q<\infty$,
\[\|\rv^{-1}T^\alpha_{\alpha}[g]\|_{\Ld^q(\R^d)}\,\lesssim_{V,q}\,\|\nabla^{\alpha-j}(\sqrt\mu g)\|_{\Ld^\frac{dq}{d+q(d-2)}(\R^d)},\]
hence~\eqref{eq:L2bnd-diff-B0-re} follows for $\gamma=\alpha$.

\medskip
\substep{2.3} Proof of~\eqref{eq:L2bnd-diff-B0-re} for $\gamma=\alpha$ in case $d=2$.\\
Starting point is again the representation~\eqref{eq:represent-Talphalph}, but we note that in view of~\eqref{eq:est-K0K1-sing} the kernels~$K^0,K^1$ are singular in 2D and require a finer treatment.
By Lemma~\ref{lem:eps-bndd-R}, we note that for all~$v,w\in\R^2$,
\begin{gather*}
K^0_j(v,s w)=s^{-2}K^0_j(v, w),\quad K^1_{j,l}(v,s w)=s^{-2}K^1_{j,l}(v, w),\qquad\text{for all $s>0$},\\
\int_{|w'|=1}K_j^0(v,w')\,dw'=\int_{|w'|=1}K_{j,l}^1(v,w')\,dw'=0,\\
\sup_{v'}\sup_{|w'|=1}\Big(|K_j^0(v',w')|+|K^1_{j,l}(v',w')|\Big)\,\lesssim_V\,1.
\end{gather*}
Standard Calder\'on--Zygmund singular integral theory in form of~\cite[Theorem~2]{Calderon-Zygmund} can then be applied and yields for all $1<q<\infty$,
\[\|\rv^{-1}T_\alpha^\alpha[g]\|_{\Ld^q(\R^2)}\,\lesssim_{V,q}\,\|\nabla^{\alpha-j}(\sqrt\mu g)\|_{\Ld^q(\R^2)},\]
hence~\eqref{eq:L2bnd-diff-B0-re} follows for $\gamma=\alpha$.
\end{proof}

Next, we prove boundedness and regularity properties of the linear operator~$\Bc(\nabla F)$ for any scalar field $F$, cf.~\eqref{eq:def-Bc}.
Since this result is crucial to the proof of Theorem~\ref{th:global}, let us give some motivation. Formally, for a solution $f$ of~\eqref{eq:g-reform-re}, we have for all $\alpha\geq0$,
\begin{multline*}
\tfrac12 \partial_t  \|\nabla^\alpha f\|^2_{\Ld^2(\R^d)} \,=\, -\int_{\R^d}  (\nabla+ v)\nabla^\alpha f \cdot \nabla^\alpha \Big(\Bc[\sqrt\mu+f]\,(\nabla+ v)f-(\sqrt\mu+f)\Bc[(\nabla+ v)f]\Big) \\
-\sum_{e_j\leq \alpha}\binom{\alpha}{e_j}\int_{\R^d}  e_j\nabla^\alpha f  \cdot  \nabla^{\alpha-e_j} \Big(\Bc[\sqrt\mu+f]\,(\nabla+ v)f-(\sqrt\mu+f)\Bc[(\nabla+ v)f]\Big).
\end{multline*}
The following lemma allows us to control precisely this right-hand side by the $H^{|\alpha|}$ norm of~$f$ and by the dissipation.
The right-hand side terms become critical when all the derivatives in $\nabla^\alpha,\nabla^{\alpha-e_j}$ fall either on $(\nabla+v)f$ or on $\Bc[(\nabla+v)f]$. To this end, the following lemma provides \eqref{eq:lemB-1} and~\eqref{eq:lemB-2}--\eqref{eq:lemB-3} for the respective cases.
An important part of the proof consists of noting that the Sobolev embedding always only needs to be applied on 1D sets in the direction of the wavenumber $k$, thus reducing the loss of derivatives.
Note that the last bound~\eqref{eq:lemB-3} is particularly involved, but is in fact not needed in the present section: it will only be used in the proof of local well-posedness away from equilibrium in Section~\ref{sec:away-equ}. The 2D case in~\eqref{eq:lemB-3} could actually be included but would substantially lengthen the proof and is omitted for simplicity.

\begin{lem}[Properties of $\Bc$]\label{lem:B}
Let $V\in\Ld^1\cap\dot H^\frac12(\R^d)$ be isotropic and positive definite.
Given \mbox{$g,h\in\Ld^2_\loc(\R^d)$}, provided $g$ satisfies the following smallness condition, for some \mbox{$r_0\ge0$}, $\delta_0>0$, and some large enough constant $C_0$,
\[\|\rv^{-r_0}\rnabla^{\frac32+\delta_0}g\|_{\Ld^2(\R^d)}\le\tfrac1{C_0},\]
we have for all vector fields $h_1,h_2$, for all $\alpha\ge0$ and $r\ge0$,
\begin{multline}\label{eq:lemB-1}
\Big|\int_{\R^d}h_1\cdot\big(\nabla^\alpha\Bc(\nabla F_g)[h]\big)\,h_2\Big|\,\lesssim_{V,\alpha,r}\,
\|h_1\|_{\Ld^2_A(\R^d)}\|h_2\|_{\Ld^2_A(\R^d)}\|h\|_{H^{|\alpha|+1}(\R^d)}\\
\times\Big(1+\|g\|_{H^{|\alpha|+1}(\R^d)}^{|\alpha|}
+\mathds1_{\alpha>0}\|\rv^{-r}\rnabla^{|\alpha|+2}g\|_{\Ld^2(\R^d)}
\Big),
\end{multline}
where alternatively we can exchange one derivative of $g,h$ for one derivative of $h_2$,
\begin{multline}\label{eq:lemB-2}
\Big|\int_{\R^d}h_1\cdot \big(\nabla^\alpha\Bc(\nabla F_g)[h]\big)\,h_2\Big|
\,\lesssim_{V,\alpha,r}\,
\|h_1\|_{\Ld^2_A(\R^d)}\|\rnabla h_2\|_{\Ld^2_A(\R^d)}\\
\times\bigg(\|h\|_{H^{(|\alpha|-1)\vee1}(\R^d)}\Big(1+\|g\|_{H^{|\alpha|}(\R^d)}^{|\alpha|}+\big(1+\mathds1_{|\alpha|\ge2}\|g\|_{H^{|\alpha|}(\R^d)}\big)\|\rv^{-r}\rnabla^{|\alpha|+1}g\|_{\Ld^2(\R^d)}\Big)\\
+\|\rv^{-r}\rnabla^{|\alpha|}h\|_{\Ld^2(\R^d)}\big(1+\|g\|_{H^{|\alpha|}(\R^d)}^{|\alpha|}\big)\bigg).
\end{multline}
In case of dimension $d>2$, provided $V\in\Ld^1\cap \dot H^2(\R^d)$ and $xV\in\Ld^2(\R^d)$, we can further improve on the norm of $g$: for all $\alpha\ge0$ and $r\ge0$,
\begin{align}\label{eq:lemB-3}
&\Big|\int_{\R^d}h_1\cdot \big(\nabla^\alpha\Bc(\nabla F_g)[h]\big)\,h_2\Big|
\,\lesssim_{V,\alpha,r}\,\|h_1\|_{\Ld^2_A(\R^d)}\|\rnabla h_2\|_{\Ld^2_A(\R^d)}\\
&\hspace{7cm}\times\|h\|_{H^{|\alpha|\vee2}(\R^d)}
\Big(1+\|g\|_{H^{|\alpha|\vee3}(\R^d)}^{|\alpha|\vee2}\Big).\nonumber\qedhere
\end{align}
\end{lem}

\begin{proof}
We split the proof into three main steps. Note that we establish slightly improved versions of~\eqref{eq:lemB-1}--\eqref{eq:lemB-3}, while in the statement we reduce for instance to integer differentiability.

\medskip
\step1 Proof of~\eqref{eq:lemB-1}.\\
By definition of $\Bc$ and of the collision operator $B$, cf.~\eqref{eq:def-Bc} and~\eqref{eq:LB-kernel}, and by Leibniz' rule, we can decompose
\begin{equation}\label{eq:decomp-Ggamalph}
\nabla^\alpha\Bc(\nabla F_g)[h]\,=\,\sum_{\gamma\le\alpha}\binom\alpha\gamma G_\gamma^\alpha(g,h),
\end{equation}
in terms of
\begin{multline}\label{eq:def-Ggamalph}
G_\gamma^\alpha(g,h)\,:=\,
\iint_{\R^d\times\R^d}(k\otimes k)\,\pi\widehat V(k)^2\delta(k\cdot(v-v_*))\nabla_v^{\alpha-\gamma}\big(\tfrac{1}{|\e(k,k\cdot v;\nabla F_g)|^2}\big)\\
\times\nabla_*^{\gamma}(\sqrt\mu_*h_*)\,dv_*\dbar k.
\end{multline}
By Faà di Bruno's formula, the derivatives of $|\e(k,k\cdot v;\nabla F_g)|^{-2}$ are evaluated as follows, for~$\gamma<\alpha$,
\begin{multline}\label{eq:diff-1/eps0}
\big|\nabla^{\alpha-\gamma}_v\big(\tfrac1{|\e(k,k\cdot v;\nabla F_g)|^2}\big)\big|\\
\,\lesssim_\alpha\,\sum_{n=1}^{|\alpha-\gamma|}\tfrac{1}{|\e(k,k\cdot v;\nabla F_g)|^{n+2}}
\sum_{\gamma_1+\ldots+\gamma_n=\alpha-\gamma\atop\gamma_1,\ldots,\gamma_n>0}~\prod_{j=1}^n|\nabla^{\gamma_j}\e(k,k\cdot v;\nabla F_g)|,
\end{multline}
By Lemma~\ref{lem:eps-bndd-R}(ii)--(iii), we deduce for all $\delta>0$ and $r\ge0$,
\begin{equation*}
\big|\nabla^{\alpha-\gamma}_v\big(\tfrac1{|\e(k,k\cdot v;\nabla F_g)|^2}\big)\big|
\,\lesssim_{V,\alpha,\delta,r}\,1+\|\rv^{-r}\rnabla^{|\alpha-\gamma|+\frac32+\delta}g\|_{\Ld^2(\R^d)}^{|\alpha-\gamma|}.
\end{equation*}
More precisely, first separating the term with $n=1$ in~\eqref{eq:diff-1/eps0}, we get the following refined estimate, for all $\delta>0$ and $r\ge0$,
\begin{multline}\label{eq:prepre-bnd-G}
\big|\nabla^{\alpha-\gamma}_v\big(\tfrac1{|\e(k,k\cdot v;\nabla F_g)|^2}\big)\big|\\
\,\lesssim_{V,\alpha,\delta,r}\,1+\|\rv^{-r}\rnabla^{|\alpha-\gamma|+\frac12+\delta}g\|_{\Ld^2(\R^d)}^{|\alpha-\gamma|}+\mathds1_{\gamma<\alpha}\|\rv^{-r}\rnabla^{|\alpha-\gamma|+\frac32+\delta}g\|_{\Ld^2(\R^d)},
\end{multline}
and thus, inserting this into~\eqref{eq:def-Ggamalph},
\begin{multline}\label{eq:pre-bnd-G}
|G_\gamma^\alpha(g,h)|\,\lesssim_{V,\alpha,\delta,r}\,
\int_{\R^d}|v-v_*|^{-1}|\nabla_*^{\gamma}(\sqrt\mu_*h_*)|\,dv_*\\
\times\Big(1+\|\rv^{-r}\rnabla^{|\alpha-\gamma|+\frac12+\delta}g\|_{\Ld^2(\R^d)}^{|\alpha-\gamma|}
+\mathds1_{\gamma<\alpha}\|\rv^{-r}\rnabla^{|\alpha-\gamma|+\frac32+\delta}g\|_{\Ld^2(\R^d)}
\Big).
\end{multline}
Writing
\begin{equation}\label{eq:v-vs-bnd}
|v-v_*|^{-1}\le\rv^{-1}\big(\rvs+|v-v_*|^{-1}\big),
\end{equation}
and appealing to the Sobolev inequality to compensate for the fact that $v_*\mapsto|v-v_*|^{-1}$ does not belong to $\Ld^2_\loc(\R^d)$ in dimension $d=2$, we deduce for all $\delta>0$ and $r\ge0$,
\begin{multline}\label{eq:bnd-Ggambet-easy0}
|G_\gamma^\alpha(g,h)|\,\lesssim_{V,\alpha,\delta,r}\,
\rv^{-1}\|\rv^{-r}\rnabla^{|\gamma|+\delta}h\|_{\Ld^2(\R^d)}\\
\times\Big(1+\|\rv^{-r}\rnabla^{|\alpha-\gamma|+\frac12+\delta}g\|_{\Ld^2(\R^d)}^{|\alpha-\gamma|}
+\mathds1_{\gamma<\alpha}\|\rv^{-r}\rnabla^{|\alpha-\gamma|+\frac32+\delta}g\|_{\Ld^2(\R^d)}
\Big),
\end{multline}
hence, for all vector fields $h_1,h_2$,
\begin{multline*}
\Big|\int_{\R^d}h_1\cdot G_\gamma^\alpha(g,h)\,h_2\Big|\,\lesssim_{V,\alpha,\delta,r}\,\|\rv^{-\frac12}h_1\|_{\Ld^2(\R^d)}\|\rv^{-\frac12}h_2\|_{\Ld^2(\R^d)}\|\rv^{-r}\rnabla^{|\gamma|+\delta}h\|_{\Ld^2(\R^d)}\\
\times\Big(1+\|\rv^{-r}\rnabla^{|\alpha-\gamma|+\frac12+\delta}g\|_{\Ld^2(\R^d)}^{|\alpha-\gamma|}
+\mathds1_{\gamma<\alpha}\|\rv^{-r}\rnabla^{|\alpha-\gamma|+\frac32+\delta}g\|_{\Ld^2(\R^d)}
\Big).
\end{multline*}
Now note that the very same bound~\eqref{eq:bnd-Ggambet-easy0} holds for $G_\gamma^\alpha(g,h)$ replaced by $G_\gamma^\alpha(g,h)v$ or by~$v\cdot G_\gamma^\alpha(g,h)v$ since the products with $v$ can be replaced by products with $v_*$ in the integral~\eqref{eq:def-Ggamalph} thanks to the Dirac mass~$\delta(k\cdot(v-v_*))$. The above estimate can therefore be improved exactly as in the proof of~\eqref{eq:hnabAh}: decomposing $h_1,h_2$ with respect to $P_v,P_v^\bot$, and suitably estimating the different terms, we deduce
\begin{multline}\label{eq:improved-Gbnd0}
\Big|\int_{\R^d}h_1\cdot G_\gamma^\alpha(g,h)\,h_2\Big|\,\lesssim_{V,\alpha,\delta,r}\,\|h_1\|_{\Ld^2_A(\R^d)}\|h_2\|_{\Ld^2_A(\R^d)}\|\rv^{-r}\rnabla^{|\gamma|+\delta}h\|_{\Ld^2(\R^d)}\\
\times\Big(1+\|\rv^{-r}\rnabla^{|\alpha-\gamma|+\frac12+\delta}g\|_{\Ld^2(\R^d)}^{|\alpha-\gamma|}
+\mathds1_{\gamma<\alpha}\|\rv^{-r}\rnabla^{|\alpha-\gamma|+\frac32+\delta}g\|_{\Ld^2(\R^d)}
\Big).
\end{multline}
Inserting this into~\eqref{eq:decomp-Ggamalph}, we get
\begin{multline*}
\Big|\int_{\R^d}h_1\cdot\big(\nabla^\alpha\Bc(\nabla F_g)[h]\big)h_2\Big|\,\lesssim_{V,\alpha,\delta,r}\,
\|h_1\|_{\Ld^2_A(\R^d)}\|h_2\|_{\Ld^2_A(\R^d)}\|\rv^{-r}\rnabla^{|\alpha|+\delta}h\|_{\Ld^2(\R^d)}\\
\times\Big(1+\|\rv^{-r}\rnabla^{|\alpha|+\frac12+\delta}g\|_{\Ld^2(\R^d)}^{|\alpha|}
+\mathds1_{\alpha>0}\|\rv^{-r}\rnabla^{|\alpha|+\frac32+\delta}g\|_{\Ld^2(\R^d)}
\Big),
\end{multline*}
and the conclusion~\eqref{eq:lemB-1} follows.

\medskip
\step2 Proof of~\eqref{eq:lemB-2}.\\
In view of~\eqref{eq:decomp-Ggamalph}, it suffices to prove suitable improvements of~\eqref{eq:improved-Gbnd0} in case $\gamma=\alpha$ and in case $|\gamma|\le1$, which we do in the following three further substeps.

\medskip
\substep{2.1} Case $\gamma=\alpha$: prove that for all $\delta>0$ and $r\ge0$,
\begin{equation}\label{eq:Galphalph-impr}
\Big|\int_{\R^d}h_1\cdot G_\alpha^\alpha(g,h)\,h_2\Big|\,\lesssim_{V,\alpha,\delta,r}\,
\|h_1\|_{\Ld^2_A(\R^d)}\|\rnabla^\delta h_2\|_{\Ld^2_A(\R^d)}\|\rv^{-r}\rnabla^{|\alpha|}h\|_{\Ld^2(\R^d)}.
\end{equation}
From~\eqref{eq:pre-bnd-G} and~\eqref{eq:v-vs-bnd}, we find
\begin{equation*}
|G_\alpha^\alpha(g,h)|\,\lesssim_{V,\alpha,\delta,r}\,
\rv^{-1}\bigg(\int_{\R^d}\rvs|\nabla_*^{\alpha}(\sqrt\mu_*h_*)|\,dv_*+\int_{\R^d}|v-v_*|^{-1}|\nabla_*^{\alpha}(\sqrt\mu_*h_*)|\,dv_*\bigg).
\end{equation*}
Appealing again to the Sobolev inequality to compensate for the fact that $v_*\mapsto|v-v_*|^{-1}$ does not belong to $\Ld^2_\loc(\R^d)$ in dimension $d=2$, we deduce for all $\delta>0$ and $r\ge0$,
\begin{equation*}
\Big|\int_{\R^d}h_1\cdot G_\alpha^\alpha(g,h)\,h_2\Big|\,\lesssim_{V,\alpha,\delta,r}\,
\|\rv^{-\frac12}h_1\|_{\Ld^2(\R^d)}\|\rv^{-\frac12}\rnabla^\delta h_2\|_{\Ld^2(\R^d)}\|\rv^{-r}\rnabla^{|\alpha|}h\|_{\Ld^2(\R^d)}.
\end{equation*}
Again improving on the weights $\rv^{-\frac12}$ as in~\eqref{eq:improved-Gbnd0}, the claim~\eqref{eq:Galphalph-impr} follows.

\medskip
\substep{2.2} Case $|\gamma|=1$: prove that for all $\delta>0$ and $r\ge0$,
\begin{multline}\label{eq:Galph1-impr}
\Big|\int_{\R^d}h_1\cdot G_{e_j}^\alpha(g,h)\,h_2\Big|\,\lesssim_{V,\alpha,\delta,r}\,
\|h_1\|_{\Ld^2_A(\R^d)}\|\rnabla^\delta h_2\|_{\Ld^2_A(\R^d)}\|\rv^{-r}\rnabla h\|_{\Ld^2(\R^d)}\\
\times\Big(1+\|\rv^{-r}\rnabla^{|\alpha|-\frac12+\delta}g\|_{\Ld^2(\R^d)}^{|\alpha|}+\|\rv^{-r}\rnabla^{|\alpha|+\frac12+\delta}g\|_{\Ld^2(\R^d)}\Big).
\end{multline}
We start from~\eqref{eq:prepre-bnd-G} in the following form, for all $\delta>0$ and $r\ge0$,
\begin{equation*}
\big|\nabla^{\alpha-e_j}_v\big(\tfrac1{|\e(k,k\cdot v;\nabla F_g)|^2}\big)\big|
\lesssim_{V,\alpha,\delta,r}
1+\|\rv^{-r}\rnabla^{|\alpha|-\frac12+\delta}g\|_{\Ld^2(\R^d)}^{|\alpha|}+\|\rv^{-r}\rnabla^{|\alpha|+\frac12+\delta}g\|_{\Ld^2(\R^d)}.
\end{equation*}
Inserting this into~\eqref{eq:def-Ggamalph} and arguing as in~\eqref{eq:Galphalph-impr}, the claim~\eqref{eq:Galph1-impr} follows.

\medskip
\substep{2.3} Case $\gamma=0$: prove that for all $\delta>0$ and $r\ge0$,
\begin{multline}\label{eq:Galph0-impr}
\Big|\int_{\R^d}h_1\cdot G_0^\alpha(g,h)\,h_2\Big|
\,\lesssim_{V,\alpha,\delta,r}\,
\|h_1\|_{\Ld^2_A(\R^d)}\|\rnabla^{\frac12+\delta}h_2\|_{\Ld^2_A(\R^d)}\|\rv^{-r}\rnabla^{\frac12+\delta} h\|_{\Ld^2(\R^d)}\\
\hspace{-4cm}\times\bigg(1+\|\rv^{-r}\rnabla^{|\alpha|}g\|_{\Ld^2(\R^d)}^{|\alpha|}\\
+\Big(1+\mathds1_{|\alpha|\ge2}\|\rv^{-r}\rnabla^2g\|_{\Ld^2(\R^d)}\Big)\|\rv^{-r}\rnabla^{|\alpha|+1}g\|_{\Ld^2(\R^d)}\bigg).
\end{multline}
Starting from~\eqref{eq:diff-1/eps0}, and applying again Lemma~\ref{lem:eps-bndd-R}(ii)--(iii), but now separating both the terms with $n=1$ and with $n=2$, we get for all $\delta>0$ and $r\ge0$,
\begin{multline*}
\big|\nabla^{\alpha}_v\big(\tfrac1{|\e(k,k\cdot v;\nabla F_g)|^2}\big)\big|
\,\lesssim_{V,\alpha,\delta,r}\,
1+\|\rv^{-r}\rnabla^{|\alpha|-\frac12+\delta}g\|_{\Ld^2(\R^d)}^{|\alpha|}
+|\nabla^\alpha\e(k,k\cdot v;\nabla F_g)|\\
+\mathds1_{|\alpha|\ge2}\sum_{e_j\le\alpha}|\nabla^{\alpha-e_j}\e(k,k\cdot v;\nabla F_g)||\nabla_j\e(k,k\cdot v;\nabla F_g)|.
\end{multline*}
Recalling~\eqref{eq:diff-eps-form} and applying again Lemma~\ref{lem:eps-bndd-R}(ii)--(iii), this yields
\begin{multline*}
\big|\nabla^{\alpha}_v\big(\tfrac1{|\e(k,k\cdot v;\nabla F_g)|^2}\big)\big|
\,\lesssim_{V,\alpha,\delta,r}\,1+\|\rv^{-r}\rnabla^{|\alpha|-\frac12+\delta}g\|_{\Ld^2(\R^d)}^{|\alpha|}
+\|\rv^{-r}\rnabla^{|\alpha|+\frac12+\delta}g\|_{\Ld^2(\R^d)}\\
+\Big|\int_{\R^d}\tfrac{k\cdot(\nabla\nabla^\alpha (\sqrt\mu g))_{**}}{k\cdot(v-v_{**})-i0}dv_{**}\Big|\\
+\mathds1_{|\alpha|\ge2}\Big(1+\|\rv^{-r}\rnabla^{|\alpha|+\frac12+\delta}g\|_{\Ld^2(\R^d)}\Big)\sum_{j}\Big|\int_{\R^d}\tfrac{k\cdot(\nabla\nabla_j(\sqrt\mu g))_{**}}{k\cdot(v-v_{**})-i0}dv_{**}\Big|.
\end{multline*}
Inserting this into~\eqref{eq:def-Ggamalph} and arguing as in~\eqref{eq:bnd-Ggambet-easy0}, we get for all vector fields $h_1,h_2$, for all~$\delta>0$ and $r\ge0$,
\begin{multline}\label{eq:bnd-Gh1h20}
\Big|\int_{\R^d}h_1\cdot G_0^\alpha(g,h)\,h_2\Big|
\,\lesssim_{V,\alpha}\,
\|\rv^{-\frac12}h_1\|_{\Ld^2(\R^d)}\|\rv^{-\frac12}h_2\|_{\Ld^2(\R^d)}\\
\times\Big(1+\|\rv^{-r}\rnabla^{|\alpha|-\frac12+\delta}g\|_{\Ld^2(\R^d)}^{|\alpha|}+\|\rv^{-r}\rnabla^{|\alpha|+\frac12+\delta}g\|_{\Ld^2(\R^d)}\Big)\|\rv^{-r}\rnabla^\delta h\|_{\Ld^2(\R^d)}\\
+T_\alpha(g,h;h_1,h_2)\\
+\mathds1_{|\alpha|\ge2}\Big(1+\|\rv^{-r}\rnabla^{|\alpha|+\frac12+\delta} g\|_{\Ld^2(\R^d)}\Big)\sum_jT_j(g,h;h_1,h_2),
\end{multline}
where we have set for abbreviation
\begin{multline}\label{eq:def-Tabbr0}
T_\alpha(g,h;h_1,h_2)\,:=\,
\int_{\R^d}|k|^2\widehat V(k)^2\\
\times\bigg(\iint_{\R^d\times\R^d}\delta(k\cdot(v-v_*))
|h_1||h_2|\sqrt\mu_*|h_*|\Big|\int_{\R^d}\tfrac{k\cdot(\nabla\nabla^\alpha (\sqrt\mu g))_{**}}{k\cdot(v-v_{**})-i0}dv_{**}\Big|\,dvdv_*\bigg)\dbar k.
\end{multline}
It remains to estimate the latter.
Recall the notation $\hat k={k}/{|k|}$.
Splitting integrals over~$v,v_*\in\R^d$ as integrals over $v,v_*\in \hat k\R\oplus\hat k^\bot$, and smuggling in a power of the weight~\mbox{$\langle\hat k\cdot v\rangle$}, we get from Hölder's inequality, for all $r_0\ge0$,
\begin{multline} \label{eq:Tabbrref}
\iint_{\R^d\times\R^d}\delta(k\cdot(v-v_*))
|h_1||h_2|\sqrt\mu_*|h_*|\Big|\int_{\R^d}\tfrac{k\cdot(\nabla\nabla^\alpha(\sqrt\mu g))_{**}}{k\cdot(v-v_{**})-i0}dv_{**}\Big|\,dvdv_*\\
\,\lesssim\,\tfrac1{|k|}\|\langle\hat k\cdot v\rangle^{-r_0}h_1\|_{\Ld^2(\hat k^\bot;\Ld^2(\hat k\R))}\|\langle\hat k\cdot v\rangle^{-r_0}h_2\|_{\Ld^2(\hat k^\bot;\Ld^\infty(\hat k\R))}\|\langle\hat k\cdot v\rangle^{2r_0}\sqrt\mu h\|_{\Ld^1(\hat k^\bot;\Ld^\infty(\hat k\R))}\\
\times\Big\|\int_{\R^d}\tfrac{\hat k\cdot(\nabla\nabla^\alpha(\sqrt\mu g))_{**}}{\hat k\cdot(\cdot-v_{**})-i0}dv_{**}\Big\|_{\Ld^2(\hat k\R)}.
\end{multline}
Using the Sobolev inequality on $\hat k\R$, the fast decay of $\mu$, and the boundedness of the Hilbert transform in $\Ld^2(\hat k\R)$, we deduce for all $r_0,r\ge0$ and $\delta>0$,
\begin{multline}\label{eq:preest-Tabbr0}
\iint_{\R^d\times\R^d}\delta(k\cdot(v-v_*))
|h_1||h_2|\sqrt\mu_*|h_*|\Big|\int_{\R^d}\tfrac{k\cdot(\nabla\nabla^\alpha(\sqrt\mu g))_{**}}{k\cdot(v-v_{**})-i0}dv_{**}\Big|\,dvdv_*\\
\,\lesssim_{r_0,r}\,\tfrac1{|k|}\|\langle\hat k\cdot v\rangle^{-r_0}h_1\|_{\Ld^2(\R^d)}\|\langle\hat k\cdot v\rangle^{-r_0}\rnabla^{\frac12+\delta}h_2\|_{\Ld^2(\R^d)}\\
\times\|\rv^{-r}\rnabla^{|\alpha|+1}g\|_{\Ld^2(\R^d)}\|\rv^{-r}\rnabla^{\frac12+\delta}h\|_{\Ld^2(\R^d)}.
\end{multline}
Choosing $r_0>\frac12$, we note that
\begin{equation*}
\int_{\Sp^{d-1}}\langle\hat k\cdot v\rangle^{-2r_0}d\sigma(\hat k)\,=\,|\Sp^{d-2}|\int_0^\pi\big(1+|v|^2(\cos\theta)^2\big)^{-r_0}(\sin\theta)\,d\theta
\,\lesssim_{r_0}\,\rv^{-1},
\end{equation*}
and therefore, as $V$ is isotropic,
\begin{eqnarray}
\lefteqn{\int_{\R^d}|k|\widehat V(k)^2\|\langle\hat k\cdot v\rangle^{-r_0}h_1\|_{\Ld^2(\R^d)}\|\langle\hat k\cdot v\rangle^{-r_0}\rnabla^{\frac12+\delta} h_2\|_{\Ld^2(\R^d)}\,\dbar k}\nonumber\\
&\lesssim&\Big(\int_{\R^d}\langle\hat k\cdot v\rangle^{-2r_0}|k|\widehat V(k)^2|h_1(v)|^2\,dv\,\dbar k\Big)^\frac12\nonumber\\
&&\hspace{2cm}\times\Big(\iint_{\R^d\times\R^d}\langle\hat k\cdot v\rangle^{-2r_0}|k|\widehat V(k)^2|\rnabla^{\frac12+\delta} h_2(v)|^2\,dv\,\dbar k\Big)^\frac12\nonumber\\
&\lesssim_{V,r}&\|\rv^{-\frac12}h_1\|_{\Ld^2(\R^d)}\|\rv^{-\frac12}\rnabla^{\frac12+\delta} h_2\|_{\Ld^2(\R^d)}.\label{eq:weight-intK}
\end{eqnarray}
Inserting~\eqref{eq:preest-Tabbr0} into~\eqref{eq:def-Tabbr0}, and using the latter estimate, we obtain for all $r\ge0$ and~$\delta>0$,
\begin{multline}\label{eq:bnd-Talph}
T_\alpha(g,h;h_1,h_2)\,\lesssim_{V,\alpha,\delta,r}\,
\|\rv^{-\frac12}h_1\|_{\Ld^2(\R^d)}\|\rv^{-\frac12}\rnabla^{\frac12+\delta} h_2\|_{\Ld^2(\R^d)}\\
\times\|\rv^{-r}\rnabla^{|\alpha|+1}g\|_{\Ld^2(\R^d)}\|\rv^{-r}\rnabla^{\frac12+\delta}h\|_{\Ld^2(\R^d)}.
\end{multline}
Inserting this into~\eqref{eq:bnd-Gh1h20} and improving on the weights $\rv^{-\frac12}$ as in~\eqref{eq:improved-Gbnd0}, the claim~\eqref{eq:Galph0-impr} follows.

\medskip
\step3 Proof of~\eqref{eq:lemB-3}.\\
In view of~\eqref{eq:decomp-Ggamalph}, \eqref{eq:improved-Gbnd0}, and~\eqref{eq:Galphalph-impr}, it suffices to prove suitable improvements of~\eqref{eq:Galph1-impr} and~\eqref{eq:Galph0-impr} for $G_\gamma^\alpha(g,h)$ with $|\gamma|\le1$. We split the proof into three further substeps. Note that we focus here on the case of dimension $d>2$ for simplicity.

\medskip
\substep{3.1} Case $|\gamma|=1$: prove that for all $\delta>0$ and $r\ge0$,
\begin{multline}\label{eq:Galph1-impr-re}
\Big|\int_{\R^d}h_1\cdot G_{e_j}^\alpha(g,h)\,h_2\Big|
\,\lesssim\,
\|h_1\|_{\Ld^2_A(\R^d)}\|\langle\nabla\rangle^{\frac12+\delta}h_2\|_{\Ld^2_A(\R^d)}\|\rv^{-r}\langle\nabla\rangle^{\frac32+\delta}h\|_{\Ld^2(\R^d)}\\
\times\Big(1+\|\rv^{-r}\rnabla^{|\alpha|}g\|_{\Ld^2(\R^d)}^{|\alpha|}\Big).
\end{multline}
Starting from~\eqref{eq:diff-1/eps0}, and arguing as for~\eqref{eq:prepre-bnd-G}, but now separating the term with $n=1$ and writing it in form of~\eqref{eq:diff-eps-form}, we get for all  $r\ge0$,
\begin{equation*}
\big|\nabla^{\alpha-e_j}_v\big(\tfrac1{|\e(k,k\cdot v;\nabla F_g)|^2}\big)\big|
\lesssim_{V,\alpha,\delta,r}
1+\|\rv^{-r}\rnabla^{|\alpha|}g\|_{\Ld^2(\R^d)}^{|\alpha|}
+\Big|\int_{\R^d}\tfrac{k\cdot(\nabla\nabla^{\alpha-e_j} F_g)_{**}}{k\cdot(v-v_{**})-i0}dv_{**}\Big|.
\end{equation*}
Inserting this into~\eqref{eq:def-Ggamalph} and arguing as in~\eqref{eq:Galphalph-impr}, we obtain
\begin{multline*}
\Big|\int_{\R^d}h_1\cdot G_{e_j}^\alpha(g,h)\,h_2\Big|
\,\lesssim\,T^j_\alpha(g,h;h_1,h_2)\\
+\|h_1\|_{\Ld^2_A(\R^d)}\|\langle\nabla\rangle^\delta h_2\|_{\Ld^2_A(\R^d)}\|\rv^{-r}\langle\nabla\rangle h\|_{\Ld^2(\R^d)}\Big(1+\|\rv^{-r}\rnabla^{|\alpha|}g\|_{\Ld^2(\R^d)}^{|\alpha|}\Big),
\end{multline*}
where we have set for abbreviation
\begin{multline*}
T^j_\alpha(g,h;h_1,h_2)\,:=\,
\int_{\R^d}|k|^2\widehat V(k)^2\\
\times\bigg(\iint_{\R^d\times\R^d}\delta(k\cdot(v-v_*))
|h_1||h_2||\nabla_{j,*}(\sqrt\mu_*h_*)|\Big|\int_{\R^d}\tfrac{k\cdot(\nabla\nabla^{\alpha-e_j} (\sqrt\mu g))_{**}}{k\cdot(v-v_{**})-i0}dv_{**}\Big|\,dvdv_*\bigg)\dbar k.
\end{multline*}
Repeating the proof of~\eqref{eq:bnd-Talph} to estimate $T^j_\alpha(g,h;h_1,h_2)$, and improving on the weights as in~\eqref{eq:improved-Gbnd0}, the claim~\eqref{eq:Galph1-impr-re} follows.

\medskip
\substep{3.2} Case $\gamma=0$ with $d>2$: prove that for all $\delta>0$ and $r\ge0$,
\begin{multline}\label{eq:bnd-Gh1h20-re}
\Big|\int_{\R^d}h_1\cdot G_0^\alpha(g,h)\,h_2\Big|
\,\lesssim_{V,\alpha,\delta,r}\,|S_\alpha(g,h;h_1,h_2)|\\
+\Big(1+\|\rv^{-r}\rnabla^{|\alpha|\vee(\frac52+\delta)}g\|_{\Ld^2(\R^d)}^{|\alpha|}\Big)\\
\times\|h_1\|_{\Ld^2_A(\R^d)}\|\rnabla^{\frac12+\delta} h_2\|_{\Ld^2_A(\R^d)}\|\rv^{-r}\rnabla^{\frac12+\delta}h\|_{\Ld^2(\R^d)},
\end{multline}
where we have set for abbreviation
\begin{multline}\label{eq:def-Sabbr0}
S_\alpha(g,h;h_1,h_2)\,:=\,\iiint_{\R^d\times\R^d\times\R^d}(k\cdot h_1)(k\cdot h_2)\,\pi\widehat V(k)^3\tfrac{\delta(k\cdot(v-v_*))}{|\e(k,k\cdot v;\nabla F_g)|^2}\\
\times\Re\Big(\tfrac{1}{\e(k,k\cdot v;\nabla F_g)}\int_{\R^d}\tfrac{k\cdot(\nabla\nabla^\alpha (\sqrt\mu g))_{**}}{k\cdot(v-v_{**})-i0}dv_{**}\Big)\sqrt\mu_*h_*\,dvdv_*\dbar k,
\end{multline}
The estimation of the latter is postponed to the next substep.

\medskip\noindent
Starting from~\eqref{eq:diff-1/eps0}, and applying again Lemma~\ref{lem:eps-bndd-R}(ii)--(iii), but now setting aside the term with $n=1$ and separately estimating the terms with $n=2$, we find for all~$r\ge0$,
\begin{multline*}
\Big|\nabla^{\alpha}_v\big(\tfrac1{|\e(k,k\cdot v;\nabla F_g)|^2}\big)+\tfrac2{|\e(k,k\cdot v;\nabla F_g)|^2}\Re\tfrac{\nabla_v^\alpha\e(k,k\cdot v;\nabla F_g)}{\e(k,k\cdot v;\nabla F_g)}\Big|\\
\,\lesssim_{V,\alpha,\delta,r}\,
1+\|\rv^{-r}\rnabla^{|\alpha|}g\|_{\Ld^2(\R^d)}^{|\alpha|}\\
+\mathds1_{|\alpha|\ge2}\sum_{e_j\le\alpha}|\nabla^{\alpha-e_j}\e(k,k\cdot v;\nabla F_g)||\nabla_j\e(k,k\cdot v;\nabla F_g)|.
\end{multline*}
Recalling~\eqref{eq:diff-eps-form}, and applying again Lemma~\ref{lem:eps-bndd-R}(ii)--(iii), we get
\begin{multline*}
\bigg|\nabla^{\alpha}_v\big(\tfrac1{|\e(k,k\cdot v;\nabla F_g)|^2}\big)+\tfrac{2\widehat V(k)}{|\e(k,k\cdot v;\nabla F_g)|^2}\Re\Big(\tfrac{1}{\e(k,k\cdot v;\nabla F_g)}\int_{\R^d}\tfrac{k\cdot(\nabla\nabla^\alpha (\sqrt\mu g))_{**}}{k\cdot(v-v_{**})-i0}dv_{**}\Big)\bigg|\\
\,\lesssim_{V,\alpha,\delta,r}\,1+\|\rv^{-r}\rnabla^{|\alpha|}g\|_{\Ld^2(\R^d)}^{|\alpha|}\\
+\mathds1_{|\alpha|\ge2}\Big(1+\|\rv^{-r}\langle\nabla\rangle^{\frac52+\delta} g\|_{\Ld^2(\R^d)}\Big)\sum_{j}\Big|\int_{\R^d}\tfrac{k\cdot(\nabla\nabla^{\alpha-e_j}(\sqrt\mu g))_{**}}{k\cdot(v-v_{**})-i0}dv_{**}\Big|.
\end{multline*}
Inserting this into~\eqref{eq:def-Ggamalph} and arguing as in~\eqref{eq:bnd-Ggambet-easy0}, we deduce for all vector fields $h_1,h_2$,
\begin{align}\label{eq:bnd-Gh1h20}
&\Big|\int_{\R^d}h_1\cdot G_0^\alpha(g,h)\,h_2\Big|
\,\lesssim_{V,\alpha,\delta,r}\,|S_\alpha(g,h;h_1,h_2)|\\
&\hspace{2cm}+\mathds1_{|\alpha|\ge2}\Big(1+\|\rv^{-r}\langle\nabla\rangle^{\frac52+\delta} g\|_{\Ld^2(\R^d)}\Big)\sum_jT_{\alpha-e_j}(g,h;h_1,h_2)\nonumber\\
&\quad+\|\rv^{-\frac12}h_1\|_{\Ld^2(\R^d)}\|\rv^{-\frac12}h_2\|_{\Ld^2(\R^d)}
\|\rv^{-r}\langle\nabla\rangle^\delta h\|_{\Ld^2(\R^d)}\Big(1+\|\rv^{-r}\rnabla^{|\alpha|}g\|_{\Ld^2(\R^d)}^{|\alpha|}\Big),\nonumber
\end{align}
where $S_\alpha(g,h;h_1,h_2)$ is defined in~\eqref{eq:def-Sabbr0} and where we recall~\eqref{eq:def-Tabbr0},
\begin{multline*}
T_{\alpha-e_j}(g,h;h_1,h_2)\,:=\,\int_{\R^d}|k|^2\widehat V(k)^2\\
\times\bigg(\iint_{\R^d\times\R^d}\delta(k\cdot(v-v_*))|h_1||h_2|\sqrt\mu_*|h_*|\Big|\int_{\R^d}\tfrac{k\cdot(\nabla\nabla^{\alpha-e_j}(\sqrt\mu g))_{**}}{k\cdot(v-v_{**})-i0}dv_{**}\Big|\,dvdv_*\bigg)\dbar k.
\end{multline*}
Recalling the estimation~\eqref{eq:bnd-Talph} on the latter, inserting it into~\eqref{eq:bnd-Gh1h20}, and improving on the weights as in~\eqref{eq:improved-Gbnd0}, the claim~\eqref{eq:bnd-Gh1h20-re} follows.

\medskip
\substep{3.3} Estimation of $S_\alpha(g,h;h_1,h_2)$ with $d>2$: prove that for all $\delta>0$ and $r\ge0$,
\begin{multline}\label{eq:est-Salpghhh-d>}
|S_\alpha(g,h;h_1,h_2)|
\,\lesssim_{V,r,r_0,\delta}\,
\|h_1\|_{\Ld^2_A(\R^d)}\|\langle\nabla\rangle^\frac12 h_2\|_{\Ld^2_A(\R^d)}\|\rv^{-r}\langle\nabla\rangle^{\frac32+\delta}h\|_{\Ld^2(\R^d)}\\
\times\|\rv^{-r}\rnabla^{|\alpha|}g\|_{\Ld^2(\R^d)}\Big(1+\|\rv^{-r}\langle\nabla\rangle^{\frac52+\delta}g\|_{\Ld^2(\R^d)}\Big).
\end{multline}
Starting point is the following integration by parts,
\[\int_{\R^d}\tfrac{k\cdot(\nabla\nabla^\alpha (\sqrt\mu g))_{**}}{k\cdot(v-v_{**})-i0}dv_{**}\,=\,|k|^2\nabla_{k}\cdot\Big(\int_{\R^d}\tfrac{v-v_{**}}{|v-v_{**}|^2}\tfrac{(\nabla^{\alpha} (\sqrt\mu g))_{**}}{k\cdot(v-v_{**})-i0}dv_{**}\Big).\]
Inserting this into definition~\eqref{eq:def-Sabbr0} for $S_\alpha(g,h;h_1,h_2)$, integrating by parts with respect to~$k$, and reorganizing the integrals, we are led to
\begin{equation}\label{eq:1strewr-S}
S_\alpha(g,h;h_1,h_2)\,=\,
-\Re\int_{\R^d}\bigg(\iint_{\R^d\times\R^d}H^j_{h;h_1,h_2}(k,v)\tfrac{(v-v_{**})_j}{|v-v_{**}|^2}\tfrac{(\nabla^{\alpha} (\sqrt\mu g))_{**}}{k\cdot(v-v_{**})-i0}\,dvdv_{**}\bigg)\dbar k,
\end{equation}
where we have set for abbreviation
\begin{equation*}
H^j_{h;h_1,h_2}(k,v)\,:=\,\nabla_{k_j}\bigg(\tfrac{|k|^2(k\cdot h_1)(k\cdot h_2)\,\pi\widehat V(k)^3}{\e(k,k\cdot v;\nabla F_g)|\e(k,k\cdot v;\nabla F_g)|^2}\Big(\int_{\R^d}\delta(k\cdot(v-v_*))\sqrt\mu_*h_*\,dv_*\Big)\bigg).
\end{equation*}
Evaluating the $k$-derivative and integrating by parts, the latter expression takes on the following guise,
\begin{multline*}
H^j_{h;h_1,h_2}(k,v)\,=\,
\tfrac{\big(\nabla_{k_j}-\frac{k_j}{|k|^2}\big)\big(|k|^2(k\cdot h_1)(k\cdot h_2)\,\pi\widehat V(k)^3\big)}{\e(k,k\cdot v;\nabla F_g)|\e(k,k\cdot v;\nabla F_g)|^2}\Big(\int_{\R^d}\delta(k\cdot(v-v_*))\sqrt\mu_*h_*\,dv_*\Big)\\
-\tfrac{2|k|^2(k\cdot h_1)(k\cdot h_2)\,\pi\widehat V(k)^3}{\e(k,k\cdot v;\nabla F_g)^2|\e(k,k\cdot v;\nabla F_g)|^2}\Big(\nabla_{k_j}\e(k,k\cdot v;\nabla F_g)\Big)\Big(\int_{\R^d}\delta(k\cdot(v-v_*))\sqrt\mu_*h_*\,dv_*\Big)\\
-\tfrac{|k|^2(k\cdot h_1)(k\cdot h_2)\,\pi\widehat V(k)^3}{|\e(k,k\cdot v;\nabla F_g)|^4}\Big(\nabla_{k_j}\overline{\e(k,k\cdot v;\nabla F_g)}\Big)\Big(\int_{\R^d}\delta(k\cdot(v-v_*))\sqrt\mu_*h_*\,dv_*\Big)\\
+ \sum_{l=1}^d \tfrac{k_l(k\cdot h_1)(k\cdot h_2)\,\pi\widehat V(k)^3}{\e(k,k\cdot v;\nabla F_g)|\e(k,k\cdot v;\nabla F_g)|^2}\Big(\int_{\R^d}\delta(k\cdot(v-v_*))(v-v_*)_j\nabla_{v_{*,l}}(\sqrt\mu h)_*\,dv_*\Big).
\end{multline*}
Note that another integration by parts yields as in~\eqref{eq:form-nabk-eps},
\begin{multline*}
\nabla_{k_j}\e(k,k\cdot v;\nabla F_g)\,=\,\big(\nabla_j\widehat V(k)-\tfrac{k_j}{|k|^2}\widehat V(k)\big)\int_{\R^d}\tfrac{\hat k\cdot\nabla F_g(v_*)}{\hat k\cdot(v-v_*)-i0}\,dv_*\\
+\tfrac1{|k|}\widehat V(k)\int_{\R^d}\tfrac{\nabla_j F_g(v_*)}{\hat k\cdot(v-v_*)-i0}\,dv_*
+\tfrac{1}{|k|}\widehat V(k)\int_{\R^d}(v-v_*)_j\tfrac{(\hat k\cdot\nabla)^2 F_g(v_*)}{\hat k\cdot(v-v_*)-i0}\,dv_*,
\end{multline*}
and thus, using the Sobolev inequality and the $\Ld^2$ boundedness of the Hilbert transform, for all $\delta>0$ and $r\ge0$,
\begin{equation*}
|\nabla_{k}\e(k,k\cdot v;\nabla F_g)|\,\lesssim_{r,\delta}\,\rv\Big(\tfrac{1}{|k|}\widehat V(k)+|\nabla\widehat V(k)|\Big)\Big(1+\|\rv^{-r}\langle\nabla\rangle^{\frac52+\delta}g\|_{\Ld^2(\R^d)}\Big).
\end{equation*}
Inserting this into~\eqref{eq:1strewr-S}, together with the above identity for $H^j_{h;h_1,h_2}$, and appealing to Lemma~\ref{lem:eps-bndd-R}(ii)--(iii), we are led to
\begin{multline*}
|S_\alpha(g,h;h_1,h_2)|\lesssim_{V,r,\delta}\,
\Big(1+\|\rv^{-r}\langle\nabla\rangle^{\frac52+\delta}g\|_{\Ld^2(\R^d)}\Big)
\int_{\R^d}\Big(|k|\widehat V(k)^3+|k|^2\widehat V(k)^2|\nabla\widehat V(k)|\Big)\\
\hspace{2cm}\times\bigg\{\int_{\R^d}|h_1||h_2|\Big(\int_{\R^d}\delta(\hat k\cdot(v-v_*))\,\langle v_*\rangle|(\langle\nabla\rangle(\sqrt\mu h))_*|\,dv_*\Big)\\
\hspace{5cm}\times\Big|\rv\int_{\R^d}\tfrac{v-v_{*}}{|v-v_{*}|^2}\tfrac{(\nabla^{\alpha} (\sqrt\mu g))_{*}}{\hat k\cdot(v-v_{*})-i0}dv_*\Big|\,dv\bigg\}\,\dbar k.
\end{multline*}
Smuggling in a power of the weight~$\langle\hat k\cdot v\rangle$, and noting that the Sobolev inequality yields for all $\delta>0$ and $r,r_0\ge0$,
\[\int_{\R^d}\delta(\hat k\cdot(v-v_*))\,\langle v_*\rangle^{r_0}|(\langle\nabla\rangle(\sqrt\mu h))_*|\,dv_*\,\lesssim_{r,r_0,\delta}\,\|\rv^{-r}\langle\nabla\rangle^{\frac32+\delta}h\|_{\Ld^2(\R^d)},\]
the above becomes
\begin{multline}\label{eq:pre-est-Salph}
|S_\alpha(g,h;h_1,h_2)|
\,\lesssim_{V,r,r_0,\delta}\,
\Big(1+\|\rv^{-r}\langle\nabla\rangle^{\frac52+\delta}g\|_{\Ld^2(\R^d)}\Big)\|\rv^{-r}\langle\nabla\rangle^{\frac32+\delta}h\|_{\Ld^2(\R^d)}\\
\times\int_{\R^d}\Big(|k|\widehat V(k)^3+|k|^2\widehat V(k)^2|\nabla\widehat V(k)|\Big)\,U_\alpha^k\big(g;\langle\hat k\cdot v\rangle^{-r_0}h_1,\langle\hat k\cdot v\rangle^{-r_0}h_2\big)\,\dbar k,
\end{multline}
where we have set for abbreviation,
\begin{equation}\label{eq:def-U}
U_\alpha^k(g;h_1,h_2)\,:=\,\int_{\R^d}|h_1||h_2|\Big|\rv\int_{\R^d}\tfrac{v-v_{*}}{|v-v_{*}|^2}\tfrac{(\nabla^{\alpha} (\sqrt\mu g))_{*}}{\hat k\cdot(v-v_{*})-i0}dv_*\Big|\,dv.
\end{equation}
We turn to the estimation of this quantity.
Let $k\in\R^d\setminus\{0\}$ be fixed.
Decomposing velocity variables $v\in\R^d$ into $v=(v_0,v_1)$ with $v_0=v\cdot\hat k$ and $v_1=v-v_0\hat k\in\hat k^\bot$, we can write
\begin{multline*}
U_\alpha^k(g;h_1,h_2)\,=\,\int_\R\int_{\R^{d-1}}|h_1(v_0,v_1)||h_2(v_0,v_1)|\\
\times\Big|\langle(v_0,v_1)\rangle\int_\R\int_{\R^{d-1}}\tfrac{(v_0-v_{0*},v_1-v_{1*})}{|v_0-v_{0*}|^2+|v_1-v_{1*}|^2}\tfrac{(\nabla^{\alpha} (\sqrt\mu g))_*}{(v_0-v_{0*})-i0}dv_{1*}dv_{0*}\Big|\,dv_1dv_0.
\end{multline*}
Estimating $\langle (v_0,v_1)\rangle\le1+|v_0|+|v_1|$, and using $\frac{v_0-v_{0*}}{(v_0-v_{0*})-i0}=1$, we are led to
\begin{multline}\label{eq:bnd-main-CZRiesz-int-decomp}
U_\alpha^k(g;h_1,h_2)\,\lesssim\,
\int_\R\int_{\R^{d-1}}|h_1(v_0,v_1)||h_2(v_0,v_1)|\\
\times\bigg\{\Big|\int_\R\int_{\R^{d-1}}\sigma(v_0-v_{0*},v_1-v_{1*})(\nabla^{\alpha} (\sqrt\mu g))_*\,dv_{1*}dv_{0*}\Big|\\
\hspace{3cm}+\Big|\int_\R\int_{\R^{d-1}}\sigma(v_0-v_{0*},v_1-v_{1*}) (v_{0*},v_{1*})(\nabla^{\alpha} (\sqrt\mu g))_*\,dv_{1*}dv_{0*}\Big|\\
\hspace{3.55cm}+\Big|\int_\R\int_{\R^{d-1}}(v_1-v_{1*})\otimes\sigma(v_0-v_{0*},v_1-v_{1*})(\nabla^{\alpha} (\sqrt\mu g))_*\,dv_{1*}dv_{0*}\Big|\bigg\}\\
\hspace{-3cm}+\int_{\R^{d}}|h_1||h_2|\Big(\int_{\R^d}\tfrac{\langle v_*\rangle|(\nabla^{\alpha} (\sqrt\mu g))_*|}{|v-v_*|^2}\,dv_*\Big)\,dv\\
+\int_{\R^{d}}|h_1||h_2|\Big(\int_{\R^d}\tfrac{|(\nabla^{\alpha} (\sqrt\mu g))_*|}{|v-v_*|}\,dv_*\Big)\,dv,
\end{multline}
in terms of the symbol
\[\sigma(v_0,v_1)\,:=\,\tfrac{1}{v_0-i0}\tfrac{v_1}{|v_0|^2+|v_1|^2}.\]
We denote by $\tilde f(w_0,v_1)$ the Fourier transform of a function $f(v_0,v_1)$ in its first variable.
As the Fourier transforms of $\frac1{v_0-i0}$ and $\frac1{v_0^2+1}$ are given by $i\pi(1-\sgn(w_0))$ and $\pi e^{-|w_0|}$, respectively,
the Fourier transform of $\sigma$ in its first variable takes on the following explicit guise,
\[\tilde\sigma(w_0,v_1)\,=\,2i\pi^2\tfrac{v_1}{|v_1|}\int_{w_0}^\infty\,e^{-|w_0'||v_1|}\,dw_0'\,=\,2i\pi^2\tfrac{v_1}{|v_1|^2}\times\left\{\begin{array}{lll}
e^{-|w_0||v_1|}&:&w_0>0,\\
2-e^{-|w_0||v_1|}&:&w_0<0,
\end{array}\right.\]
hence,
\[|\tilde\sigma(w_0,v_1)|\le4\pi^2\tfrac{1}{|v_1|}.\]
Taking Fourier transform in the $v_0$-variables, using Parseval's identity, inserting this bound on the symbol $\tilde\sigma$, using the triangle inequality and again Parseval's identity, we find for any $f\in C^\infty_c(\R^d)$ and $1\le p\le\infty$,
\begin{eqnarray*}
\lefteqn{\Big\|\int_{\R}\int_{\R^{d-1}}\sigma(v_0-v_{0*},v_1-v_{1*})\,f(v_{0*},v_{1*})\,dv_{1*}dv_{0*}\Big\|_{\Ld^p_{v_1}(\R^{d-1};\Ld^2_{v_0}(\R))}}\\
&\simeq&\Big\|\int_{\R^{d-1}}\tilde \sigma(w_0,v_1-v_{1*})\,\tilde f(w_{0},v_{1*})\,dv_{1*}\Big\|_{\Ld^p_{v_1}(\R^{d-1};\Ld^2_{w_0}(\R))}\\
&\lesssim&\Big\|\int_{\R^{d-1}}\tfrac{|\tilde f(w_{0},v_{1*})|}{|v_1-v_{1*}|}dv_{1*}\Big\|_{\Ld^p_{v_1}(\R^{d-1};\Ld^2_{w_0}(\R))}\\
&\lesssim&\Big\|\int_{\R^{d-1}}\tfrac{\|f(\cdot,v_{1*})\|_{\Ld^2(\R)}}{|v_1-v_{1*}|}dv_{1*}\Big\|_{\Ld^p_{v_1}(\R^{d-1})},
\end{eqnarray*}
and therefore, by the Hardy--Littlewood--Sobolev inequality, in dimension $d>2$, for all~\mbox{$d-1<p<\infty$} and $q=\frac{(d-1)p}{d-1+(d-2)p}$,
\begin{multline*}
\Big\|\int_{\R}\int_{\R^{d-1}}\sigma(v_0-v_{0*},v_1-v_{1*})\,f(v_{0*},v_{1*})\,dv_{1*}dv_{0*}\Big\|_{\Ld^p_{v_1}(\R^{d-1};\Ld^2_{v_0}(\R))}\\
\,\lesssim_{p}\,\|f(v_0,v_1)\|_{\Ld^q_{v_1}(\R^{d-1};\Ld^2_{v_0}(\R))}.
\end{multline*}
Similarly, we find
\begin{multline*}
\Big\|\int_{\R}\int_{\R^{d-1}}(v_1-v_{1*})\otimes\sigma(v_0-v_{0*},v_1-v_{1*})\,f(v_{0*},v_{1*})\,dv_{1*}dv_{0*}\Big\|_{\Ld^\infty_{v_1}(\R^{d-1};\Ld^2_{v_0}(\R))}\\
\,\lesssim\,\|f(v_0,v_1)\|_{\Ld^1_{v_1}(\R^{d-1};\Ld^2_{v_0}(\R))}.
\end{multline*}
Applying these two bounds to estimate the first terms in~\eqref{eq:bnd-main-CZRiesz-int-decomp}, using the Hardy--Littlewood--Sobolev inequality for the last two terms, and using the Sobolev inequality and the decay of $\mu$ to reduce to $\Ld^2$ norms, we easily deduce in dimension $d>2$, for all $r\ge0$,
\begin{equation*}
U_\alpha^k(g;h_1,h_2)\,\lesssim_r\,\|h_1\|_{\Ld^2(\R^d)}\|\langle\nabla\rangle^\frac12 h_2\|_{\Ld^2(\R^d)}\|\rv^{-r}\rnabla^{|\alpha|}g\|_{\Ld^2(\R^d)}.
\end{equation*}
Inserting this into~\eqref{eq:pre-est-Salph}, we are led to
\begin{multline*}
|S_\alpha(g,h;h_1,h_2)|
\,\lesssim_{V,r,r_0,\delta}\,
\Big(1+\|\rv^{-r}\langle\nabla\rangle^{\frac52+\delta}g\|_{\Ld^2(\R^d)}\Big)\\
\times\|\rv^{-r}\rnabla^{|\alpha|}g\|_{\Ld^2(\R^d)}\|\rv^{-r}\langle\nabla\rangle^{\frac32+\delta}h\|_{\Ld^2(\R^d)}\\
\times\int_{\R^d}\Big(|k|\widehat V(k)^3+|k|^2\widehat V(k)^2|\nabla\widehat V(k)|\Big)\,
\|\langle\hat k\cdot v\rangle^{-r_0}h_1\|_{\Ld^2(\R^d)}\|\langle\hat k\cdot v\rangle^{-r_0}\langle\nabla\rangle^\frac12 h_2\|_{\Ld^2(\R^d)}\,\dbar k.
\end{multline*}
Now evaluating the $k$-integral as in~\eqref{eq:weight-intK}, with $r_0>\frac12$, and noting that the integrability assumptions on $V$ ensure $\int_{\R^d}|k|\widehat V(k)^3+\int_{\R^d}|k|^2\widehat V(k)^2|\nabla\widehat V(k)|<\infty$, this becomes
\begin{multline*}
|S_\alpha(g,h;h_1,h_2)|
\,\lesssim_{V,r,r_0,\delta}\,
\Big(1+\|\rv^{-r}\langle\nabla\rangle^{\frac52+\delta}g\|_{\Ld^2(\R^d)}\Big)\|\rv^{-r}\rnabla^{|\alpha|}g\|_{\Ld^2(\R^d)}\\
\times\|\rv^{-r}\langle\nabla\rangle^{\frac32+\delta}h\|_{\Ld^2(\R^d)}
\|\rv^{-\frac12}h_1\|_{\Ld^2(\R^d)}\|\rv^{-\frac12}\langle\nabla\rangle^\frac12 h_2\|_{\Ld^2(\R^d)}.
\end{multline*}
Further improving on the weights $\rv^{-\frac12}$ as in~\eqref{eq:improved-Gbnd0}, the claim~\eqref{eq:est-Salpghhh-d>} follows.
\end{proof}

With the above estimates at hand, we may now establish the local well-posedness of the Lenard--Balescu equation~\eqref{eq:g-reform-re} close to Maxwellian equilibrium.

\begin{prop}[Local well-posedness close to equilibrium]\label{prop:loc-wp-equ}
Let $V\in\Ld^1\cap\dot H^\frac12(\R^d)$ be isotropic and positive definite.
For all $s\ge2$, there is a constant $C_{V,s}$ large enough such that the following holds: for all initial data $f^\circ\in H^s(\R^d)$ satisfying the positivity and smallness conditions,
\[F_{f^\circ}=\mu+\sqrt\mu f^\circ\ge0,\qquad\|f^\circ\|_{H^s(\R^d)}\le \tfrac1{C_{V,s}},\]
there exists $T>\frac1{C_{V,s}}$ and a unique strong solution $f\in \Ld^\infty([0,T];H^s(\R^d))$ of the Lenard--Balescu equation~\eqref{eq:g-reform-re} on $[0,T]$, and it still satisfies $F_f=\mu+\sqrt\mu f\ge0$.
\end{prop}

\begin{proof}
We proceed by constructing a sequence $\{f_n\}_n$ of approximate solutions: we choose the $0$th-order approximation as the initial data,
\[f_0:=f^\circ,\]
and next, for all $n\ge0$, given the $n$th-order approximation $f_n$, we iteratively define $f_{n+1}$ as the solution of the following linear Cauchy problem,
\begin{equation}\label{eq:g-n-approx}
\left\{\begin{array}{l}
\partial_tf_{n+1}=(\nabla- v)\cdot\Big(\Bc_n[\sqrt\mu+f_{n}]\,(\nabla+ v)f_{n+1}-(\sqrt\mu+f_{n+1})\Bc_n[(\nabla+ v)f_{n}]\Big),\\
f_{n+1}|_{t=0}=f^\circ,
\end{array}\right.
\end{equation}
where we have set for abbreviation
\[\Bc_n[g]\,:=\,\Bc(\nabla F_{f_n})[g]\,=\,\int_{\R^d}B(v,v-v_*;\nabla F_{f_n})\,\sqrt\mu_*g_*\,dv_*.\]
Noting that definitions of $\Bc_n$ and of the collision kernel $B$ ensure $\Bc_n[vg]=\Bc_n[g]v$ for all scalar fields $g$, cf.~\eqref{eq:LB-kernel},
we can rewrite equation~\eqref{eq:g-n-approx} as the following linear parabolic equation,
\begin{equation}\label{eq:g-n-approx-rewr}
\left\{\begin{array}{l}
\partial_tf_{n+1}=\nabla\cdot A_n\nabla f_{n+1}-b_n\cdot \nabla f_{n+1}-\nabla\cdot (b_n'f_{n+1})+c_nf_{n+1}\\
\hspace{7cm}-(\nabla-v)\cdot(\sqrt\mu(b_n+b_n')),\\
f_{n+1}|_{t=0}=f^\circ,
\end{array}\right.
\end{equation}
in terms of the coefficients
\begin{eqnarray}
A_n&:=&\Bc_n[\sqrt\mu+f_{n}],\label{eq:def-Anbn}\\
b_n&:=&A_nv~=~\Bc_n[v(\sqrt\mu+f_{n})],\nonumber\\
b_n'&:=&\Bc_n[\nabla (\sqrt\mu+f_{n})],\nonumber\\
c_n&:=&v\cdot b_n'.\nonumber
\end{eqnarray}
Let $s\ge2$ be fixed, let $C_0$ be a constant to be chosen large enough, and assume
\begin{equation}\label{eq:ass-f0-reg}
F_{f^\circ}=\mu+\sqrt\mu f^\circ\ge0,\qquad \|f^\circ\|_{H^s(\R^d)}\le\tfrac1{4C_0}.
\end{equation}
We shall control the sequence $(f_n)_n$ by induction. Given $n\ge0$, assume that $f_n$ is well-defined on some time interval with $\mu+\sqrt\mu f_n\ge0$, and define the maximal time
\begin{equation}\label{eq:def-Tn}
T_n\,:=\,1\wedge\max\Big\{T\ge0:\|f_n\|_{\Ld^\infty([0,T];H^s(\R^d))}^2+\tfrac1{C_0}\|\nabla\rnabla^sf_n\|_{\Ld^2([0,T];\Ld^2_A(\R^d))}^2\le \tfrac1{C_0^2}\Big\}.
\end{equation}
By definition of $\Bc_n$ and of the collision kernel $B$, cf.~\eqref{eq:LB-kernel}, choosing $C_0$ large enough, Lemma~\ref{lem:eps-bndd-R}(ii) ensures that coefficients $A_n,b_n,b_n',c_n$ are all well-defined on the time interval~$[0,T_n]$.
In addition, the assumption $\mu+\sqrt\mu f_n\ge0$ ensures $A_n\ge0$.
Up to regularizing coefficients and initial data, linear parabolic theory ensures that~\eqref{eq:g-n-approx-rewr} admits a unique smooth solution on $[0,T_n]$. We shall estimate its~$H^s$~norm uniformly with respect to regularization, and we split the proof into three steps.

\medskip
\step{1} Prove that for all $\alpha\ge0$, on the time interval $[0,T_n]$,
\begin{multline}\label{eq:estHs-An}
\int_{\R^d}(\nabla^\alpha g)\nabla^\alpha(\nabla\cdot A_n\nabla g)\,\le\,
-\tfrac1{C_V}\|\nabla\nabla^\alpha g\|_{\Ld^2_A(\R^d)}^2\\
\hspace{-5cm}+C_{V,\alpha}\|\nabla\nabla^\alpha g\|_{\Ld^2_A(\R^d)}\|g\|_{H^{|\alpha|\vee2}(\R^d)}\\
\times\bigg(1+\|f_n\|_{H^{|\alpha|}(\R^d)}^{|\alpha|+1}+\|\nabla\rnabla^{|\alpha|}f_n\|_{\Ld^2_A(\R^d)}\Big(1+\|f_n\|_{H^{|\alpha|}(\R^d)}^{|\alpha|}\Big)\bigg).
\end{multline}
In view of~\eqref{eq:def-Tn}, provided $C_0$ is large enough,
we note that Lemma~\ref{lem:eps-bndd-R}(ii) entails on the time interval $[0,T_n]$,
\[\Bc_n[\sqrt\mu]\,\simeq_V\, A,\]
and therefore
\[\int_{\R^d}(\nabla\nabla^\alpha g)\cdot \Bc_n[\sqrt\mu](\nabla\nabla^\alpha g)\,\simeq_V\,\|\nabla\nabla^\alpha g\|_{\Ld^2_A(\R^d)}^2.\]
Decomposing $A_n=\Bc_n[\sqrt\mu]+\Bc_n[f_n]$, we deduce
\[\int_{\R^d}(\nabla\nabla^\alpha g)\cdot A_n(\nabla\nabla^\alpha g)\,\ge\,\tfrac1{C_V}\|\nabla\nabla^\alpha g\|_{\Ld^2_A(\R^d)}^2-\Big|\int_{\R^d}(\nabla\nabla^\alpha g)\cdot \Bc_n[f_n](\nabla\nabla^\alpha g)\Big|,\]
and thus, appealing to~\eqref{eq:lemB-1} in Lemma~\ref{lem:B} to estimate the last term,
\[\int_{\R^d}(\nabla\nabla^\alpha g)\cdot A_n(\nabla\nabla^\alpha g)\,\ge\,\|\nabla\nabla^\alpha g\|_{\Ld^2_A(\R^d)}^2\Big(\tfrac1{C_V}-C_V\|f_n\|_{H^1(\R^d)}\Big).\]
In view of~\eqref{eq:def-Tn}, choosing $C_0\ge2C_V^2$, this yields on $[0,T_n]$,
\[\int_{\R^d}(\nabla\nabla^\alpha g)\cdot A_n(\nabla\nabla^\alpha g)\,\ge\,\tfrac1{2C_V}\|\nabla\nabla^\alpha g\|_{\Ld^2_A(\R^d)}^2.\]
By Leibniz' rule and an integration by parts, we then find
\begin{multline*}
\int_{\R^d}(\nabla^\alpha g)\nabla^\alpha(\nabla\cdot A_n\nabla g)\,\le\,
-\tfrac1{2C_V}\|\nabla\nabla^\alpha g\|_{\Ld^2_A(\R^d)}^2\\
-\sum_{0<\gamma\le\alpha}\binom\alpha\gamma\int_{\R^d}(\nabla\nabla^\alpha g)\cdot (\nabla^\gamma A_n)\,(\nabla\nabla^{\alpha-\gamma} g).
\end{multline*}
Now using Lemma~\ref{lem:B} to estimate the last sum, appealing to~\eqref{eq:lemB-1} or~\eqref{eq:lemB-2} depending on the value of $\gamma$, the claim~\eqref{eq:estHs-An} follows.

\medskip
\step{2} Prove that for all $\alpha\ge0$, on the time interval $[0,T_n]$,
\begin{multline}\label{eq:estHs-bn}
\bigg|\int_{\R^d}(\nabla^\alpha g)\nabla^\alpha\Big(-b_n\cdot\nabla g-\nabla\cdot(b_n'g)+c_ng-(\nabla-v)\cdot(\sqrt\mu(b_n+b_n'))\Big)\bigg|\\
\,\lesssim_{V,\alpha}\,
\Big(\|\nabla\rnabla^{|\alpha|}g\|_{\Ld^2_A(\R^d)}+\|g\|_{H^{|\alpha|}(\R^d)}\Big)\|g\|_{H^{|\alpha|\vee2}(\R^d)}\\
\times\bigg(1+\|f_n\|_{H^{|\alpha|\vee2}(\R^d)}^{|\alpha|+1}+\|\nabla\rnabla^{|\alpha|}f_n\|_{\Ld^2_A(\R^d)}\Big(1+\|f_n\|_{H^{|\alpha|\vee2}(\R^d)}^{|\alpha|}\Big)\bigg).
\end{multline}
By Leibniz' rule, we decompose
\begingroup\allowdisplaybreaks
\begin{eqnarray*}
\int_{\R^d}(\nabla^\alpha g)\nabla^\alpha(b_n\cdot\nabla g)
&=&\sum_{\gamma\le\alpha}\binom\alpha\gamma\int_{\R^d}(\nabla^\alpha g)(\nabla^\gamma b_n)\cdot(\nabla\nabla^{\alpha-\gamma} g),\\
\int_{\R^d}(\nabla^\alpha g)\nabla^\alpha\nabla\cdot(b_n' g)
&=&-\sum_{\gamma\le\alpha}\binom\alpha\gamma\int_{\R^d}(\nabla\nabla^\alpha g)\cdot(\nabla^\gamma b_n')(\nabla^{\alpha-\gamma} g),\\
\int_{\R^d}(\nabla^\alpha g)\nabla^\alpha(c_n g)
&=&\sum_{\gamma\le\alpha}\binom\alpha\gamma\int_{\R^d}(\nabla^\gamma c_n)(\nabla^\alpha g)(\nabla^{\alpha-\gamma} g),\\
\int_{\R^d}(\nabla^\alpha g)\nabla^\alpha(\nabla-v)\cdot(\sqrt\mu(b_n+b_n'))
&=&-\sum_{e_j\le\alpha}\int_{\R^d}((\nabla+v)\nabla^\alpha g)\cdot\nabla^\alpha(\sqrt\mu(b_n+b_n'))\\
&&\hspace{-0.4cm}-\sum_{e_j\le\alpha}\binom\alpha{e_j}\int_{\R^d}(e_j\nabla^\alpha g)\cdot\nabla^{\alpha-e_j}(\sqrt\mu(b_n+b_n')).
\end{eqnarray*}
\endgroup
Recalling the definition of $b_n,b_n',c_n$, cf.~\eqref{eq:def-Anbn}, and using Lemma~\ref{lem:B} to estimate the different terms, appealing again to~\eqref{eq:lemB-1} or~\eqref{eq:lemB-2} depending on the value of $\gamma$, the claim~\eqref{eq:estHs-bn} follows easily.

\medskip
\step3 Conclusion.\\
Combining~\eqref{eq:estHs-An} and~\eqref{eq:estHs-bn}, and taking advantage of the dissipation term in~\eqref{eq:estHs-An} to absorb factors~$\|\rnabla^{|\alpha|+1} g\|_{\Ld^2_A(\R^d)}$, we deduce for the solution $f_{n+1}$ of~\eqref{eq:g-n-approx-rewr} (up to regularization),
\begin{multline}\label{eq:Hs-evol-fn1}
\partial_t\|f_{n+1}\|_{H^s(\R^d)}^2
+\tfrac1{C_{V,s}}\|\nabla\rnabla^{s}f_{n+1}\|_{\Ld^2_A(\R^d)}^2\\
\,\le\,C_{V,s}\|f_{n+1}\|_{H^{s}(\R^d)}^2\Big(1+\|\nabla\rnabla^{s}f_n\|_{\Ld^2_A(\R^d)}^2\Big)\Big(1+\|f_n\|_{H^{s}(\R^d)}^{2(s+1)}\Big),
\end{multline}
hence, after time integration,
\begin{multline*}
\|f_{n+1}^t\|_{H^s(\R^d)}^2+\tfrac1{C_{V,s}}\|\nabla\rnabla^sf_{n+1}\|_{\Ld^2([0,t];\Ld^2_A(\R^d))}^2\\
\,\le\,
2\|f^\circ\|_{H^s(\R^d)}^2
\exp\bigg(2C_{V,s}\Big(t+\|\nabla\rnabla^{s}f_n\|_{\Ld^2([0,t];\Ld^2_A(\R^d))}^2\Big)\Big(1+\|f_n\|_{\Ld^\infty([0,t];H^{s}(\R^d))}^{2(s+1)}\Big)\bigg).
\end{multline*}
For $t\le T_n$, in view of~\eqref{eq:ass-f0-reg} and~\eqref{eq:def-Tn}, this entails
\begin{equation}\label{eq:est-Hs}
\|f_{n+1}^t\|_{H^s(\R^d)}^2+\tfrac1{C_{V,s}}\|\nabla\rnabla^sf_{n+1}\|_{\Ld^2([0,t];\Ld^2_A(\R^d))}^2
\,\le\,
\tfrac18C_0^{-2}
\exp\Big(4C_{V,s}\big(t+C_0^{-1}\big)\Big).
\end{equation}
As this bound holds uniformly with respect to regularization of coefficients and of initial data, well-posedness for~\eqref{eq:g-n-approx-rewr} in $\Ld^\infty([0,T_n];H^s(\R^d))$ without regularization follows by an approximation argument.

\medskip\noindent
Next, we check that the solution satisfies $F_{f_{n+1}}=\mu+\sqrt\mu f_{n+1}\ge0$.
As equation~\eqref{eq:g-n-approx} or~\eqref{eq:g-n-approx-rewr} for $f_{n+1}$ is equivalent to the following,
\[\partial_tF_{f_{n+1}}\,=\,\nabla\cdot\Big(\Bc_n[\tfrac1{\sqrt\mu}F_{f_n}]\nabla F_{f_{n+1}}-F_{f_{n+1}}\,\Bc_n[\tfrac1{\sqrt\mu}\nabla F_{f_n}]\Big).\]
this positivity statement follows from the maximum principle.

\medskip\noindent
Finally, choosing $C_0\ge1+2C_{V,s}$, the above $H^s$ estimate~\eqref{eq:est-Hs} for $f_{n+1}$ given $f_n$ entails for all $t\le T_n\wedge \frac1{2C_{V,s}}$,
\begin{equation*}
\|f_{n+1}^t\|_{H^s(\R^d)}^2+\tfrac1{C_0}\|\nabla\rnabla^sf_{n+1}\|_{\Ld^2([0,t];\Ld^2_A(\R^d))}^2
\,\le\,C_0^{-2},
\end{equation*}
which means $T_{n+1}\ge T_n\wedge(2C_{V,s})^{-1}$.
Therefore, setting $T_0:=1\wedge(2C_{V,s})^{-1}$,
we obtain by iteration a sequence $(f_n)_n$ satisfying the sequence of approximate equations~\eqref{eq:g-n-approx} on the time interval $[0,T_0]$ with
\[\|f_n\|_{\Ld^\infty([0,T_0];H^s(\R^d))}\,\le\,C_0^{-1}.\]
The conclusion now follows easily by a compactness argument, passing to the limit $n\uparrow\infty$ in approximate equations~\eqref{eq:g-n-approx}.
\end{proof}

We now turn to the proof of Theorem~\ref{th:global}, that is, we extend the above into a global-in-time result.
Taking inspiration from Guo's argument in~\cite{Guo-02} for the Landau equation, the proof relies heavily on detailed properties of the linearized operator $L$ and on refined estimates on the nonlinear operator $N$.

\begin{proof}[Proof of Theorem~\ref{th:global}]
We split the proof into five main steps.

\medskip
\step{1} Energy dissipation norm.\\
We define the following weighted norm, which is adapted to the dissipation structure of the linear operator $L$,
\begin{equation}\label{eq:def-3norm}
\3g\3^2\,:=\,\|\nabla g\|_{\Ld^2_A(\R^d)}^2+\|vg\|_{\Ld^2_A(\R^d)}^2\,=\,\int_{\R^d}\nabla g\cdot A\nabla g+vg\cdot Avg,
\end{equation}
where we recall that $A$ is the elliptic coefficient field defined in~\eqref{eq:def-A}, and we show that it satisfies
\begin{multline}\label{eq:dissipation} 
\| \rv^{-\frac12} g\|_{\Ld^2(\R^d)}+\| \rv^{-\frac32}\nabla g\|_{\Ld^2(\R^d)}\\
\,\lesssim\,\| \rv^{-\frac12} g\|_{\Ld^2(\R^d)}+\| \rv^{-\frac12} P_v^\bot \nabla g\|_{\Ld^2(\R^d)}+\| \rv^{-\frac32}  P_v \nabla g\|_{\Ld^2(\R^d)}\,\simeq_V\,\3g\3.
\end{multline}
where we recall the definition~\eqref{eq:def-projv} of orthogonal projections $P_v,P_v^\bot$.

\medskip\noindent
By definition of $\3\cdot\3$, Lemma~\ref{lem:A}(i) yields
\begin{eqnarray*}
\| \rv^{-\frac12} P_v^\bot \nabla g\|_{\Ld^2(\R^d)}^2+\| \rv^{-\frac32}  P_v \nabla g\|_{\Ld^2(\R^d)}^2
&=&\int_{\R^d}\rv^{-1}|P_v^\bot\nabla g|^2+\rv^{-3}|P_v\nabla g|^2\\
&\simeq_V&\int_{\R^d}\nabla g\cdot A\nabla g,
\end{eqnarray*}
so it remains to prove
\begin{equation}\label{eq:remain-Aequiv3norm}
\int_{\R^d}vg\cdot Avg\,\lesssim_V\,\|\rv^{-\frac12}g\|_{\Ld^2(\R^d)}^2\,\lesssim_V\,\int_{\R^d}vg\cdot Avg+\|\rv^{-\frac32}\nabla g\|_{\Ld^2(\R^d)}^2.
\end{equation}
The first inequality follows from Lemma~\ref{lem:A}(i) in form of $v\cdot Av\lesssim_V\rv^{-1}$, and we now turn to the second inequality.
For that purpose, we choose a cut-off function $\chi\in C^\infty_c(2B)$ with $\chi|_B=1$ and $|\nabla\chi|\le2$, and we decompose
\[\|\langle v\rangle^{-\frac12}g\|_{\Ld^2(\R^d)}^2\,\le\,\int_{2B}|\chi g|^2+\int_{\R^d\setminus B}\langle v\rangle^{-1}|g|^2.\]
By Poincaré's inequality on $2B$, using the properties of $\chi$, we deduce
\begin{eqnarray*}
\|\langle v\rangle^{-\frac12}g\|_{\Ld^2(\R^d)}^2&\lesssim&\int_{2B}|\nabla(\chi g)|^2+\int_{\R^d\setminus B}\langle v\rangle^{-1}|g|^2\\
&\lesssim&\int_{2B}|\nabla g|^2+\int_{\R^d\setminus B}\langle v\rangle^{-1}|g|^2.
\end{eqnarray*}
Now using Lemma~\ref{lem:A}(i) in form of $v\cdot A(v)v\gtrsim_V|v|^2\langle v\rangle^{-3}\simeq\langle v\rangle^{-1}$ on $\R^d\setminus B$, the second inequality in~\eqref{eq:remain-Aequiv3norm} follows.

\medskip
\step2 Control on dissipation rate: prove that
\begin{equation}\label{eq:L'-lower}
-\int_{\R^d}gL[g]\,\gtrsim_V\,\3(\Id-\pi_0)[g]\3^2,
\end{equation}
where $\pi_0$ stands for the orthogonal projection of $\Ld^2(\R^d)$ onto its finite-dimensional subspace
\begin{equation}\label{eq:def-S0}
S_0\,:=\,\big\{\sqrt\mu\big(a+b\cdot v+c|v|^2\big)~:~a,c\in\R,~b\in\R^d\big\},
\end{equation}
or equivalently, $\pi_0[g]:=\sum_i(\int_{\R^d}w_ig)w_i$ in terms of an orthonormal basis $\{w_i\}_{1\le i\le d+2}$ of the subspace~$S_0$.

\medskip\noindent
Since the definition of $B$ yields $B(v,v-v_*;\nabla F)=B(v_*,v_*-v;\nabla F)$, cf.~\eqref{eq:LB-kernel}, we have by definition of $L$, cf.~\eqref{eq:g-reform-re},
\begin{eqnarray}
-\int_{\R^d}gL[g]
&=&\tfrac12\iint_{\R^d\times\R^d}\Big(\sqrt\mu_*(\nabla+v)g-\sqrt\mu((\nabla+v)g)_*\Big)\label{eq:positivity-L'}\\
&&\hspace{1.8cm}\cdot\, B(v,v-v_*;\nabla\mu)\Big(\sqrt\mu_*(\nabla+v)g-\sqrt\mu((\nabla+v)g)_*\Big)dvdv_*\nonumber\\
&\ge&0.\nonumber
\end{eqnarray}
As $B(v,v-v_*;\nabla\mu)(v-v_*)=0$, this expression clearly vanishes for $g\in S_0$ (and in fact only for $g\in S_0$), and it remains to prove a quantitative version of this fact in form of the claimed lower bound~\eqref{eq:L'-lower}. For that purpose, we argue by contradiction: if~\eqref{eq:L'-lower} was failing,
there would exist a sequence $(g_n)_n$ such that
\begin{equation}\label{eq:gn-contrad}
0\le-\int_{\R^d}g_nL[g_n]\le\tfrac1n,\qquad\3g_n\3=1,\qquad\pi_0[g_n]=0.
\end{equation}
We split the argument into two further substeps.

\medskip
\substep{2.1} Prove that up to an extraction the sequence $(g_n)_n$ in~\eqref{eq:gn-contrad} converges weakly to some~$g$ in $H^1_\loc(\R^d)$ with
\begin{equation}\label{eq:prop-glim}
-\int_{\R^d}gL[g]=0,\qquad\3g\3=1,\qquad\pi_0[g]=0.
\end{equation}
In view of~\eqref{eq:dissipation},
the property $\3g_n\3=1$ in~\eqref{eq:gn-contrad} ensures that
for all~$M>0$ the sequence of restrictions $(g_n|_{B_M})_n$ is bounded in~$H^1(B_M)$. By weak compactness, there exists~\mbox{$g\in H^1_\loc(\R^d)$} such that up to an extraction~$(g_n)_n$ converges weakly to~$g$ in~$H^1_\loc(\R^d)$.
Passing to the limit in~\eqref{eq:gn-contrad}, we infer
\begin{equation}\label{eq:g-contrad-0}
\3g\3\le1,\qquad\pi_0[g]=0.
\end{equation}
It remains to pass to the limit in the relation $0\le-\int_{\R^d}g_nL[g_n]\le\frac1n$, which requires some more care.
By definition of $L$, cf.~\eqref{eq:defL-reform}, we can decompose
\[-\int_{\R^d}g_nL[g_n]
\,=\,\3g_n\3^2+2\int_{\R^d}vg_n\cdot A\nabla g_n
-\int_{\R^d}\sqrt\mu(\nabla+v)g_n\cdot \Bc_\circ[(\nabla+v)g_n].\]
Using that $\3g_n\3=1$, and noting that $2\int_{\R^d}vg_n\cdot A\nabla g_n=-\int_{\R^d}g_n^2\nabla\cdot(Av)$ follows from integration by parts, we deduce
\begin{equation}\label{eq:gn/g-lim}
-\int_{\R^d}g_nL[g_n]
\,=\,1-\int_{\R^d} g_n^2\,\nabla\cdot(Av)
-\int_{\R^d}\sqrt\mu(\nabla+v)g_n\cdot \Bc_\circ[(\nabla+v)g_n].
\end{equation}
Along the extraction, as $g_n$ converges weakly to $g$ in $H^1_\loc(\R^d)$, hence strongly in $\Ld^2_\loc(\R^d)$ by Rellich's theorem, we deduce for all~\mbox{$M\ge1$},
\begin{eqnarray}
\int_{|v|\le M}g_n^2\,\nabla\cdot(Av)&\xrightarrow{n\uparrow\infty}&\int_{|v|\le M}g^2\,\nabla\cdot(Av),\label{eq:conv-truncM}\\
\int_{|v|\le M}\sqrt\mu(\nabla+v)g_n\cdot \Bc_\circ^M[(\nabla+v)g_n]&\xrightarrow{n\uparrow\infty}&\int_{|v|\le M}\sqrt\mu(\nabla+v)g\cdot \Bc_\circ^M[(\nabla+v)g],\nonumber
\end{eqnarray}
in terms of the truncated operator
\[\Bc_\circ^M[h]\,:=\,\int_{|v_*|\le M}B(v,v-v_*;\nabla\mu)\sqrt\mu_*h_*\,dv_*.\]
Next, we estimate truncation errors: we show for all $M\ge1$,
\begin{eqnarray}
\sup_{n}\Big|\int_{|v|> M}g_n^2\,\nabla\cdot(Av)\Big|&\lesssim&\tfrac1M, \label{Remainder1}\\
\sup_{n}\Big|\int_{|v|> M}\sqrt\mu(\nabla+v)g_n\cdot \Bc_\circ[(\nabla+v)g_n]\Big|&\lesssim&\tfrac1M, \label{Remainder2} \\
\sup_{n}\Big|\int_{|v|\le M}\sqrt\mu(\nabla+v)g_n\cdot (\Bc_\circ-\Bc_\circ^M)[(\nabla+v)g_n]\Big|&\lesssim&\tfrac1M. \label{Remainder3}
\end{eqnarray}
The estimate~\eqref{Remainder1} follows from Lemma~\ref{lem:A} in form of
\begin{eqnarray*}
\Big|\int_{|v|> M}g_n^2\,\nabla\cdot(Av)\Big|
&\lesssim_V&\int_{|v|> M}\langle v\rangle^{-2}g_n^2\\
&\lesssim_V&\int_{|v|> M}\langle v\rangle^{-1}(vg_n\cdot Avg_n)\\
&\lesssim&\tfrac1M\3g_n\3^2\,=\,\tfrac1M.
\end{eqnarray*}
Next, appealing to Lemma~\ref{lem:B0}(i) and to~\eqref{eq:dissipation},
\begin{multline*}
\Big|\int_{|v|> M}\sqrt\mu(\nabla+v)g_n\cdot \Bc_\circ[(\nabla+v)g_n]\Big|\\
\,\lesssim\,\tfrac1M\|\rv^{-\frac32}(\nabla+v)g_n\|_{\Ld^2(\R^d)}\|\rv^{\frac52}\sqrt\mu\Bc_\circ[(\nabla+v)g_n]\|_{\Ld^2(\R^d)}\\
\,\lesssim_V\,\tfrac1M\|\rv^{-\frac32}(\nabla+v)g_n\|_{\Ld^2(\R^d)}^2
\,\lesssim_V\,\tfrac1M\3g_n\3^2\,=\,\tfrac1M,
\end{multline*}
that is,~\eqref{Remainder2}.
Finally, rewriting
\begin{align}
(\Bc_\circ-\Bc_\circ^M)[(\nabla+v)g_n] =  \Bc_\circ[\mathds1_{|v|>M}(\nabla+v)g_n],
\end{align}
we similarly find
\begin{multline*}
\Big|\int_{|v|\le M}\sqrt\mu(\nabla+v)g_n\cdot (\Bc_\circ-\Bc_\circ^M)[(\nabla+v)g_n]\Big|\\
\,\lesssim\,\|\rv^{-\frac32}(\nabla+v)g_n\|_{\Ld^2(\R^d)}\|\rv^{\frac32}\sqrt\mu \Bc_\circ[\mathds1_{|v|>M}(\nabla+v)g_n]\|_{\Ld^2(\R^d)}\\
\,\lesssim_V\,\tfrac1M\|\rv^{-\frac32}(\nabla+v)g_n\|_{\Ld^2(\R^d)}^2
\,\lesssim_V\,\tfrac1M\3g_n\3^2\,=\,\tfrac1M,
\end{multline*}
that is,~\eqref{Remainder3}.

\medskip\noindent
Combined with~\eqref{eq:gn/g-lim} and~\eqref{eq:conv-truncM}, the above estimates~\eqref{Remainder1}--\eqref{Remainder3} allow to pass to the limit in~\eqref{eq:gn/g-lim}, and thus, recalling that~\eqref{eq:gn-contrad} yields $-\int_{\R^d}g_nL[g_n]\to0$, we find
\begin{eqnarray*}
0
&=&1-\int_{\R^d} g^2\,\nabla\cdot(Av)-\int_{\R^d}\sqrt\mu(\nabla+v)g\cdot \Bc_\circ[(\nabla+v)g]\\
&=&1-\3g\3^2-\int_{\R^d}gL[g].
\end{eqnarray*}
Combining this with the positivity~\eqref{eq:positivity-L'} and with~\eqref{eq:g-contrad-0}, we conclude
\[-\int_{\R^d}gL[g]=0,\qquad\3g\3=1,\]
hence~\eqref{eq:prop-glim}.

\medskip
\substep{2.2} Conclusion.\\
It remains to show that for a function $h\in H^1_\loc(\R^d)$ with $\3h\3 <\infty$, the following equivalence holds: 
\begin{equation}\label{eq:S0equiv}
h\in S_0  ~~\Longleftrightarrow~~ -\int_{\R^d}hL[h]=0.
\end{equation}
Indeed, using this equivalence, we would deduce from~\eqref{eq:prop-glim} that $g\in S_0$, so that the two other conditions in~\eqref{eq:prop-glim} would become contradictory, thus yielding the conclusion.

\medskip\noindent
It remains to prove the claimed equivalence~\eqref{eq:S0equiv}. The implication `$\Rightarrow$' is clear and we now focus on the reverse. For that purpose, we note that, for all $v,v_*,e\in\R^d$ with $v\ne v_*$, the kernel $B(v,v-v_*;\nabla\mu)$ satisfies by definition,
\[e\cdot B(v,v-v_*;\nabla\mu)e\,\ge\,0,\]
and, using Lemma~\ref{lem:eps-bndd-R} and the explicit computation~\eqref{eq:expl-Landau-pr},
\begin{eqnarray*}
e\cdot  B(v,v-v_*;\nabla\mu)e=0 &\Longleftrightarrow&\int_{\R^d}|e\cdot k|^2\pi\widehat V(k)^2\tfrac{\delta(k\cdot(v-v_*))}{|\e(k,k\cdot v;\nabla \mu)|^2}\dbar k=0\\
&\Longleftrightarrow&e\cdot\Big(\Id-\tfrac{(v-v_*)\otimes(v-v_*)}{|v-v_*|^2}\Big)e=0\\
&\Longleftrightarrow&e \parallel v-v_*.
\end{eqnarray*}
In view of~\eqref{eq:positivity-L'}, we deduce that the condition $-\int_{\R^d}hL[h]=0$ implies almost everywhere
\[\sqrt\mu_*(\nabla+v)h=\sqrt\mu((\nabla+v)h)_*,\]
which implies $h\in S_0$ by integration.

\medskip
\step3 Control on dissipation rate for differentiated equation: prove that for all $\alpha>0$,
\begin{equation}\label{eq:lb-nabsL}
-\int_{\R^d}(\nabla^\alpha g)\nabla^\alpha L[g]\,\ge\,\tfrac1{C_V}\3\nabla^\alpha g\3^2-C_{V,\alpha}\3g\3_{|\alpha|-1}^2,
\end{equation}
where we have set for abbreviation,
\begin{equation}\label{eq:def-3norm-diff}
\3g\3_s^2\,:=\,\sum_{|\gamma|\le s}\3\nabla^\gamma g\3^2.
\end{equation}
We split the proof into three further substeps.

\medskip
\substep{3.1} Prove that
 \begin{align}\label{eq:bnd-pi0-nab}
\3 \pi_0 [\nabla g] \3 \,\lesssim\, \3g\3. 
 \end{align}
Recall that $S_0$ is a finite-dimensional space, cf.~\eqref{eq:def-S0}, and that the orthogonal projection~$\pi_0$ can be written as $\pi_0[g]:=\sum_i(\int_{\R^d}w_ig)w_i$ for some orthonormal basis $\{w_i\}_{1\le i\le d+2}$ of~$S_0$.  It is therefore sufficient to show that for all $w\in S_0$ we have
\begin{align}\label{eq:bnd-pi0-nab-pre}
\Big| \int_{\R^d}w\nabla g  \Big| \,\lesssim_V\,\|w\|_{\Ld^2(\R^d)} \3g\3,
\end{align}
which happens to be a direct consequence of~\eqref{eq:dissipation}
as all norms are equivalent on $S_0$.

\medskip
\substep{3.2} Prove that for all $\alpha>0$,
\begin{equation}\label{eq:Lcommutator}
\Big|\int_{\R^d}(\nabla^\alpha g)\,[\nabla^\alpha,L]g\Big| \,\lesssim_{V,\alpha}\, \3\nabla^\alpha g\3\3g\3_{|\alpha|-1}.
\end{equation}
By definition of $L$, cf.~\eqref{eq:defL-reform}, we decompose
\[\int_{\R^d}(\nabla^\alpha g)\,[\nabla^\alpha,L]g\,=\,T_1^\alpha(g)+T_2^\alpha(g),\]
in terms of
\begin{eqnarray*}
T_1^\alpha(g)&:=&\int_{\R^d}(\nabla^\alpha g)\,\big[\nabla^\alpha,(\nabla-v)\cdot (A(\nabla+v))\big]g,\\
T_2^\alpha(g)&:=&-\int_{\R^d}(\nabla^\alpha g)\,\big[\nabla^\alpha,(\nabla-v)\cdot(\sqrt\mu\Bc_\circ(\nabla+v))\big]g.
\end{eqnarray*}
We start with estimating the first term $T_1^\alpha(g)$.
By Leibniz' rule, we compute
\begin{multline*}
T_1^\alpha(g)\,=\,-\sum_{\gamma\le\alpha\atop\gamma\ne0}\binom\alpha\gamma\bigg(\int_{\R^d}((\nabla+v)\nabla^\alpha g)\cdot (\nabla^\gamma A)((\nabla+v)\nabla^{\alpha-\gamma}g)\\
+\sum_{e_j\le\gamma}\int_{\R^d}(e_j\nabla^\alpha g)\cdot (\nabla^{\gamma-e_j}A)((\nabla+v)\nabla^{\alpha-\gamma}g)
+\sum_{e_j\le\gamma}\int_{\R^d}((\nabla+v)\nabla^\alpha g)\cdot (\nabla^{\gamma-e_j}A)(e_j\nabla^{\alpha-\gamma}g)\\
+\sum_{e_j+e_l\le\gamma}\int_{\R^d}(e_{j}\nabla^\alpha g)\cdot (\nabla^{\gamma-e_j-e_l}A)e_{l}\nabla^{\alpha-\gamma}g
\bigg),
\end{multline*}
hence, by~\eqref{eq:hnabAh} and~\eqref{eq:dissipation},
\[|T_1^\alpha(g)|\,\lesssim_{V,\alpha}\,\sum_{\gamma<\alpha}\3\nabla^\alpha g\3\3\nabla^{\gamma}g\3\]
We turn to the second term $T_2^\alpha(g)$. By Leibniz' rule, we can write
\begin{multline*}
T_2^\alpha(g)\,=\,\int_{\R^d}((\nabla+v)\nabla^\alpha g)\,\cdot\nabla^\alpha\big(\sqrt\mu\Bc_\circ[(\nabla+v)g]\big)\\
+\sum_{e_j\le\alpha}\int_{\R^d}(e_j\nabla^\alpha g)\,\cdot\nabla^{\alpha-e_j}\big(\sqrt\mu\Bc_\circ[(\nabla+v)g]\big)\\
-\int_{\R^d}((\nabla+v)\nabla^\alpha g)\,\cdot \sqrt\mu\Bc_\circ[(\nabla+v)\nabla^\alpha g],
\end{multline*}
or alternatively, further integrating by parts in the last right-hand side term, for any $j_0\in\alpha$,
\begin{multline*}
T_2^\alpha(g)\,=\,\int_{\R^d}((\nabla+v)\nabla^\alpha g)\,\cdot\nabla^\alpha\big(\sqrt\mu\Bc_\circ[(\nabla+v)g]\big)\\
+\sum_{e_j\le\alpha}\int_{\R^d}(e_j\nabla^\alpha g)\,\cdot\nabla^{\alpha-e_j}\big(\sqrt\mu\Bc_\circ[(\nabla+v)g]\big)\\
+\int_{\R^d}((\nabla+v)\nabla^{\alpha-j_0} g)\,\cdot \nabla_{j_0}\big(\sqrt\mu\Bc_\circ[(\nabla+v)\nabla^\alpha g]\big)
+\int_{\R^d}(e_{j_0}\nabla^{\alpha-j_0} g)\,\cdot \sqrt\mu\Bc_\circ[(\nabla+v)\nabla^\alpha g].
\end{multline*}
Now appealing to Lemma~\ref{lem:B0} and to~\eqref{eq:dissipation}, we deduce
\begin{equation*}
|T_2^\alpha(g)|\,\lesssim_{V,\alpha}\,
\sum_{\gamma<\alpha}\3\nabla^\alpha g\3\3\nabla^\gamma g\3,
\end{equation*}
and the claim~\eqref{eq:Lcommutator} follows.

\medskip
\substep{3.3} Proof of~\eqref{eq:lb-nabsL}.\\
Let $\alpha>0$. Decomposing
\begin{equation*}
-\int_{\R^d}(\nabla^\alpha g)\nabla^\alpha L[g]\,=\,-\int_{\R^d}(\nabla^\alpha g)L[\nabla^\alpha g]+\int_{\R^d}(\nabla^\alpha g)[\nabla^\alpha, L]g,
\end{equation*}
and applying~\eqref{eq:L'-lower} and~\eqref{eq:Lcommutator}, we find
\begin{eqnarray*}
-\int_{\R^d}(\nabla^\alpha g)\nabla^\alpha L[g]&\ge&\tfrac1{C_V}\3(\Id-\pi_0)[\nabla^\alpha g]\3^2-C_{V,\alpha}\3\nabla^\alpha g\3\3 g\3_{|\alpha|-1}\\
&\ge&\tfrac1{2C_V}\3\nabla^\alpha g\3^2-C_V\3\pi_0[\nabla^\alpha g]\3^2-C_{V,\alpha}\3g\3^2_{|\alpha|-1}.
\end{eqnarray*}
Combined with~\eqref{eq:bnd-pi0-nab}, this proves~\eqref{eq:lb-nabsL}.

\medskip
\step4 Control on the nonlinearity: provided that $g$ satisfies the following smallness condition, for some $r_0\ge0$, $\delta_0>0$, and some large enough constant $C_0$,
\begin{equation}\label{eq:smallness-cond-RE}
\|\rv^{-r_0}\rnabla^{\frac32+\delta_0}g\|_{\Ld^2(\R^d)}\le\tfrac1{C_0}
\end{equation}
we prove for all $\alpha\ge0$,
\begin{equation}\label{eq:bnd-N}
\Big|\int_{\R^d}(\nabla^\alpha g)(\nabla^\alpha N(g))\Big|\,\lesssim_{V,\alpha}\,\3g\3_{|\alpha|}^2\|g\|_{H^{|\alpha|\vee2}}\Big(1+\|g\|_{H^{|\alpha|}}^{|\alpha|}\Big),
\end{equation}
where we recall the notation~\eqref{eq:def-3norm-diff}.
The subtle point in this estimate is to control the nonlinearity in terms of the energy dissipation norm. Note that $|\alpha|\vee2$ could be replaced by $|\alpha|\vee(\frac32+\delta_0)$ for any $\delta_0>0$, but the two are equivalent for our purposes as we focus on integer differentiability.

\medskip\noindent
By definition of $N$, cf.~\eqref{eq:g-reform-re}, we decompose
\begin{equation*}
N(g)\,=\,N_1(g)+N_2(g)+N_3(g)+N_4(g)
\end{equation*}
in terms of
\begin{eqnarray*}
N_1(g)&:=&(\nabla- v)\cdot\Bc(\nabla F_g)[g]\,\nabla g,\\
N_2(g)&:=&-(\nabla- v)\cdot\big(g\,\Bc(\nabla F_g)[\nabla g]\big),\\
N_3(g)&:=&(\nabla- v)\cdot\Big(\Bc(\nabla F_g)[\sqrt\mu]-\Bc(\nabla\mu)[\sqrt\mu]\Big)(\nabla+v)g,\\
N_4(g)&:=&-(\nabla- v)\cdot\bigg(\sqrt{\mu}\Big(\Bc(\nabla F_g)[(\nabla+v)g]-\Bc(\nabla\mu)[(\nabla+v)g]\Big)\bigg).
\end{eqnarray*}
We focus on $N_1$, while the estimation of $N_2$, $N_3$, and $N_4$ is similar and is skipped for shortness.
By Leibniz' rule and an integration by parts, we find
\begin{multline*}
\int_{\R^d}(\nabla^\alpha g)(\nabla^\alpha N_1(g))
\,=\,-\sum_{\gamma\le\alpha}\binom\alpha\gamma \int_{\R^d}((\nabla+v)\nabla^\alpha g)\cdot(\nabla^\gamma\Bc(\nabla F_g)[g])\,\nabla\nabla^{\alpha-\gamma} g\\
-\sum_{\gamma+e_j\le\alpha}\binom\alpha{\gamma,e_j}\int_{\R^d}(e_j\nabla^\alpha g)\cdot (\nabla^{\gamma}\Bc(\nabla F_g)[g])\,\nabla\nabla^{\alpha-\gamma-e_j} g.
\end{multline*}
By Lemma~\ref{lem:B}, appealing to~\eqref{eq:lemB-1} or~\eqref{eq:lemB-2} depending on the value of $\gamma$, and recalling~\eqref{eq:dissipation}, we obtain~\eqref{eq:bnd-N} for $N_1$.

\medskip 
\step5 Conclusion.\\
Given $s\ge2$, let $f\in\Ld^\infty([0,T];H^{s}(\R^d))$ be a local strong solution of the Lenard--Balescu equation~\eqref{eq:g-reform-re} on $[0,T]$ as given by Proposition~\ref{prop:loc-wp-equ}. Up to smoothing the initial data and shortening a bit the time interval $[0,T]$, the solution can be assume to have more smoothness, and we may then write on $[0,T]$ for all $0\le r\le s$,
\[\tfrac12\partial_t\sum_{|\alpha|\le r}\|\nabla^\alpha f\|_{\Ld^2(\R^d)}^2-\sum_{|\alpha|\le r}\int_{\R^d}(\nabla^\alpha f)\nabla^\alpha L[f]\,=\,\sum_{|\alpha|\le r}\int_{\R^d}(\nabla^\alpha f)\nabla^\alpha N(f).\]
hence, inserting~\eqref{eq:lb-nabsL} and~\eqref{eq:bnd-N}, and using that for $\alpha=0$ we have $-\int_{\R^d}fL[f]\ge0$, we deduce
\begin{multline}\label{eq:appl-step23}
\tfrac12\partial_t\sum_{|\alpha|\le r}\|\nabla^\alpha f\|_{\Ld^2(\R^d)}^2+\tfrac1{C_V}\sum_{0<|\alpha|\le r}\3\nabla^\alpha f\3^2\,\le\,C_{V,s}\sum_{|\alpha|\le r-1}\3\nabla^\alpha f\3^2\\
+C_{V,s}\Big(1+\|f\|_{H^{s}(\R^d)}^{s}\Big)\|f\|_{H^{s}(\R^d)}\sum_{|\alpha|\le r}\3\nabla^\alpha f\3^2.
\end{multline}
By Proposition~\ref{prop:loc-wp-equ}, assuming $\|f^\circ\|_{H^s(\R^d)}\le\frac12\wedge(8C_VC_{V,s})^{-1}$, and choosing $T$ small enough, we get
\begin{equation}\label{eq:smallness-Hs}
\|f\|_{H^s(\R^d)}\le1\wedge(4C_VC_{V,s})^{-1}\qquad\text{on $[0,T]$,}
\end{equation}
which allows to absorb the last right-hand side term of~\eqref{eq:appl-step23} into the dissipation term. This yields on $[0,T]$ for all $0\le r\le s$,
\begin{equation*}
\tfrac12\partial_t\sum_{|\alpha|\le r}\|\nabla^\alpha f\|_{\Ld^2(\R^d)}^2+\tfrac1{2C_V}\sum_{0<|\alpha|\le r}\3\nabla^\alpha f\3^2\,\le\,C_{V,s}\sum_{|\alpha|\le r-1}\3\nabla^\alpha f\3^2+\tfrac1{2C_V}\3f\3^2,
\end{equation*}
and thus, writing a telescoping sum over $1\le r\le s$ and using the dissipation,
\begin{equation}\label{eq:telescop-Hs}
\tfrac12\partial_t\sum_{r=1}^s(2+2C_VC_{V,s})^{-r}\sum_{|\alpha|\le r}\|\nabla^\alpha f\|_{\Ld^2(\R^d)}^2
\,\le\,\tfrac1{C_V}\3f\3^2.
\end{equation}
Now for $r=0$, rather appealing to~\eqref{eq:L'-lower}, the inequality~\eqref{eq:appl-step23} becomes
\begin{equation*}
\tfrac12\partial_t\|f\|_{\Ld^2(\R^d)}^2+\tfrac1{C_V}\3(\Id-\pi_0)[f]\3^2\,\le\,C_{V,s}\Big(1+\|f\|_{H^{s}(\R^d)}^{s}\Big)\|f\|_{H^{s}(\R^d)}\3f\3^2,
\end{equation*}
hence, by~\eqref{eq:smallness-Hs},
\begin{equation}\label{eq:L2bndLB}
\tfrac12\partial_t\|f\|_{\Ld^2(\R^d)}^2+\tfrac1{C_V}\3(\Id-\pi_0)[f]\3^2\,\le\,\tfrac1{2C_V}\3f\3^2.
\end{equation}
In view of~\eqref{eq:conserv}, as by assumption $\pi_0[f^\circ]=0$, the Lenard--Balescu equation~\eqref{eq:g-reform-re} ensures that the solution $f$ satisfies
\[\pi_0[f]=0\qquad\text{on $[0,T]$},\]
so that~\eqref{eq:L2bndLB} becomes
\begin{equation}\label{eq:L2-dissip}
\tfrac12\partial_t\|f\|_{\Ld^2(\R^d)}^2+\tfrac1{2C_V}\3f\3^2\,\le\,0.
\end{equation}
Combining this with~\eqref{eq:telescop-Hs}, we conclude
\begin{equation*}
\tfrac12\partial_t\sum_{r=0}^s(4+4C_VC_{V,s})^{-r}\sum_{|\alpha|\le r}\|\nabla^\alpha f\|_{\Ld^2(\R^d)}^2
\,\le\,0,
\end{equation*}
which can now be used to propagate the small local solution globally.
\end{proof}

\section{Convergence to equilibrium}
This section is devoted to the proof of Theorem~\ref{cor:conv}.
We start with the entropic convergence and the proof of~\eqref{eq:conventropy}, which quickly follows from techniques developed by Toscani and Villani for the Landau equation in~\cite{Toscani-Villani-99,Toscani-Villani-00}. In view of Lemma~\ref{lem:eps-bndd-R}(ii), the constructed solution $F$ satisfies \mbox{$|\e(k,k\cdot v;\nabla F)|\gtrsim_V1$} globally in time. Combining this with~\eqref{eq:expl-Landau-pr} in form of
\[B(v,v-v_*;\nabla F)\,\gtrsim_V\,\tfrac1{|v-v_*|}\Big(\Id-\tfrac{(v-v_*)\otimes(v-v_*)}{|v-v_*|^2}\Big),\]
the $H$-theorem~\eqref{eq:H-princ} takes the form
\begin{equation}\label{eq:lowerbnd-D}
-\partial_tH(F|\mu)
~\gtrsim_V~D(F),
\end{equation}
in terms of the entropy dissipation functional
\[D(F)\,:=\,\tfrac12\iint _{\R^d\times\R^d}\tfrac{FF_*}{|v-v_*|}\Big|\Big(\Id-\tfrac{(v-v_*)\otimes(v-v_*)}{|v-v_*|^2}\Big)\big(\tfrac{\nabla F}F-\tfrac{\nabla_* F_*}{F_*}\big)\Big|^2.\]
Next, given $\ell\ge1$, we appeal to the following logarithmic Sobolev-type inequality, which is obtained in~\cite[Proposition~4]{Toscani-Villani-00},
\begin{equation*}
D(F) \,\gtrsim\, C(F)K_\ell(F)^{-\frac3\ell}H(F|\mu)^{1+\frac3\ell},
\end{equation*}
where $C(F)$ only depends on $F$ via an upper bound on $H(F|\mu)$, and where $K_\ell(F)$ is given by
\begin{equation*}
K_\ell (F)\,:=\, \int_{\R^d}  \langle v\rangle^{\ell+2}\big(|\nabla \sqrt{F}|^2 + F\big).
\end{equation*}
Inserting this into~\eqref{eq:lowerbnd-D}, and noting that the $H$-theorem and the choice~\eqref{eq:smallness-st} entail
\[H(F|\mu)\,\le\,H(F^\circ|\mu)\,\le\,\int_{\R^d}|f^\circ|^2,\]
we deduce
\begin{equation}\label{eq:lowerbnd-Dre}
-\partial_tH(F|\mu)~\gtrsim_{V,f^\circ}~K_\ell(F)^{-\frac3\ell}H(F|\mu)^{1+\frac3\ell}.
\end{equation}
It remains to estimate $K_\ell(F)$. For that purpose, we appeal to the following functional inequalities from~\cite[p.1297]{Toscani-Villani-00},
\begin{eqnarray*}
K_\ell(F)&\le&\int_{\R^d}\Big|\nabla\sqrt{\langle v\rangle^{\ell+2} F}\Big|^2+C\ell^2\int_{\R^d}\langle v\rangle^{\ell+2}F,\\
\int_{\R^d}|\nabla\sqrt F|^2&\lesssim&\Big(\int_{\R^d}\langle v\rangle^{d+1}|\langle\nabla\rangle^2F|^2\Big)^\frac12,
\end{eqnarray*}
which we combine in form of
\begin{equation*}
K_\ell(F)\,\lesssim\,
\ell^2\Big(\int_{\R^d}\langle v\rangle^{2\ell+d+5}|\langle\nabla\rangle^2F|^2\Big)^\frac12
+\ell^2\int_{\R^d}\langle v\rangle^{\ell+2}F.
\end{equation*}
For the constructed solution $F=\mu+\sqrt\mu f$, noting that $\int_{\R^d}\langle v\rangle^n\mu\lesssim(Cn)^{\frac n2}$, we get
\begin{equation*}
K_\ell(F)\,\lesssim\,(C\ell)^{\frac\ell2}\big(1+\|f\|_{H^2(\R^d)}\big)\,\lesssim\,(C\ell)^{\frac\ell2}.
\end{equation*}
Inserting this into~\eqref{eq:lowerbnd-Dre}, we deduce
\[-\partial_tH(F|\mu)~\gtrsim_{V}~\ell^{-\frac32}H(F|\mu)^{1+\frac3\ell},\]
and the conclusion~\eqref{eq:conventropy} easily follows after time integration and optimization in $\ell$.

\medskip
We turn to the $\Ld^2$-convergence and the proof of~(i)--(ii), for which we take inspiration from the work of Strain and Guo~\cite{Guo-Strain-08}.
By an approximation argument, note that the convergence $f^t\to0$ in $\Ld^2(\R^d)$ follows from the quantitative estimates in~\mbox{(i)--(ii)}, hence we may focus on the latter.
So as to prove both estimates at once, we consider the mixed weight function
\[w_{\ell,\theta,K}(v):=\rv^\ell\exp(K\rv^\theta).\]
Let parameters $\ell,\theta,K\ge0$ be fixed either with $\theta<2$, or with $\theta=2$ and $K\ll_V1$. We split the proof into two main steps.

\medskip
\step1 Compactness estimates: if initial data $f^\circ$ satisfies
\[\int_{\R^d}w_{\ell,\theta,K}^2|f^\circ|^2<\infty,\]
and if in addition $\|f^\circ\|_{H^2(\R^d)}\ll_{V,\ell,\theta,K}1$,
then the solution $f$ of~\eqref{eq:g-reform-re} satisfies for all $t\ge0$,
\begin{equation}\label{eq:compact-est}
\int_{\R^d}w_{\ell,\theta,K}^2|f^t|^2\,\lesssim_{V,\ell,\theta,K}\,\int_{\R^d}w_{\ell,\theta,K}^2|f^\circ|^2.
\end{equation}
We focus on the case $\theta<2$, and we emphasize that the argument below ensures that the multiplicative factor is bounded uniformly in the limit $\theta\uparrow2$ provided $K\ll_V1$. From this, the critical case $\theta=2$ with $K\ll_V1$ is deduced a posteriori by letting~$\theta\uparrow2$ in~\eqref{eq:compact-est}.
Next, up to an approximation argument, adding a small constant diffusion in equation~\eqref{eq:g-reform-re}, we note that we may assume the solution $f$ to have some Gaussian decay, which allows to justify all computations below for $\theta<2$.
From equation~\eqref{eq:g-reform-re}, we can decompose
\begin{equation*}
\tfrac12\partial_t\int_{\R^d}w_{\ell,\theta,q}^2|f|^2+I_1(f)\,=\,I_2(f)+I_3(f)+I_4(f),
\end{equation*}
in terms of
\begin{eqnarray*}
I_1(g)&:=&\int_{\R^d}(\nabla+v)(w_{\ell,\theta,q}^2g)\cdot A(\nabla+v)g,\\
I_2(g)&:=&\int_{\R^d}\sqrt\mu(\nabla+v)(w_{\ell,\theta,q}^2g)\cdot \Bc(\nabla F_g)[(\nabla+v)g],\\
I_3(g)&:=&-\int_{\R^d}(\nabla+v)(w_{\ell,\theta,q}^2g)\cdot \Big(\Bc(\nabla F_g)[g]\nabla g-g\,\Bc(\nabla F_g)[\nabla g]\Big),\\
I_4(g)&:=&\int_{\R^d}(\nabla+v)(w_{\ell,\theta,q}^2g)\cdot \Big(\Bc(\nabla F_g)[\sqrt\mu]-\Bc(\nabla \mu)[\sqrt\mu]\Big)(\nabla+v)g.
\end{eqnarray*}
We split the proof into four further substeps.

\medskip
\substep{1.1} Weighted energy dissipation norm $\3\cdot\3_{\ell,\theta,q}$.\\
We define the following weighted version of the energy dissipation norm~\eqref{eq:def-3norm},
\[\3g\3_{\ell,\theta,K}^2\,:=\,\int_{\R^d}w_{\ell,\theta,K}^2\Big(\nabla g\cdot A\nabla g+vg\cdot Avg\Big),\]
which satisfies as in~\eqref{eq:dissipation},
\begin{multline}\label{eq:dissipation-w} 
\|\rv^{-\frac12}w_{\ell,\theta,K} g\|_{\Ld^2(\R^d)}+\| \rv^{-\frac32}w_{\ell,\theta,K}\nabla g\|_{\Ld^2(\R^d)}\\
\,\lesssim\,\| \rv^{-\frac12}w_{\ell,\theta,K} g\|_{\Ld^2(\R^d)}+\| \rv^{-\frac32}w_{\ell,\theta,K}  P_v \nabla g\|_{\Ld^2(\R^d)}+\| \rv^{-\frac12}w_{\ell,\theta,K} P_v^\bot \nabla g\|_{\Ld^2(\R^d)}\\
\,\simeq_V\,\3g\3_{\ell,\theta,K}.
\end{multline}

\substep{1.2} Proof that
\begin{equation}\label{eq:bnd-I1}
I_1(g)\,\ge\,\tfrac1{2}\3g\3_{\ell,\theta,K}^2
-C_{V,\ell,\theta,K}\3g\3^2.
\end{equation}
By integration by parts, we can decompose
\begin{equation*}
I_1(g)
\,=\,\3g\3_{\ell,\theta,K}^2-\int_{\R^d}g^2\,\nabla\cdot(w_{\ell,\theta,q}^2 Av) +\int_{\R^d}g(\nabla w_{\ell,\theta,q}^2)\cdot A(\nabla+v)g.
\end{equation*}
In view of Lemma~\ref{lem:A}, as obviously $\nabla w_{\ell,\theta,K}^2=P_v\nabla w_{\ell,\theta,K}^2$, we get
\begin{equation}\label{eq:pre-bnd-I1}
I_1(g)
\,\ge\,\3g\3_{\ell,\theta,K}^2-C_V\int_{\R^d}\rv^{-2}g^2\big(w_{\ell,\theta,q}^2+|\nabla w_{\ell,\theta,q}^2|\big)-C_V\int_{\R^d}\rv^{-3}|g||\nabla g||\nabla w_{\ell,\theta,q}^2|.
\end{equation}
Noting that $|\nabla w_{\ell,\theta,K}^2|\le2(\ell+\theta K\rv^{\theta} )\rv^{-1}w_{\ell,\theta,K}^2$, and splitting the cases $|v|\le M$ and~\mbox{$|v|>M$}, we can bound for all $M>0$,
\begin{multline*}
\rv^{-1}\big(w_{\ell,\theta,K}^2+|\nabla w_{\ell,\theta,K}^2|\big)\le\big(1+2\ell+2\theta K\big)\langle M\rangle^{2\ell}e^{2K\langle M\rangle^\theta}\\
+\big(M^{-1}+2\ell M^{-2}+2\theta K M^{\theta-2} \big)w_{\ell,\theta,K}^2,
\end{multline*}
hence, for all $N>0$, choosing $M\gg_{V,\ell,\theta,K,N}1$ large enough (and also \mbox{$K\ll_{V,N}1$} small enough in case $\theta=2$),
\[\rv^{-1}\big(w_{\ell,\theta,K}^2+|\nabla w_{\ell,\theta,K}^2|\big)\le C_{V,\ell,\theta,K,N}+\tfrac1{N}w_{\ell,\theta,K}^2.\]
Inserting this into~\eqref{eq:pre-bnd-I1}, we get for all $N>0$,
\begin{multline*}
I_1(g)
\,\ge\,\3g\3_{\ell,\theta,K}^2-C_{V,\ell,\theta,K,N}\Big(\int_{\R^d}\rv^{-1}g^2+\int_{\R^d}\rv^{-2}|g||\nabla g|\Big)\\
-\tfrac1NC_V\Big(\int_{\R^d}\rv^{-1}w_{\ell,\theta,q}^2g^2+\int_{\R^d}\rv^{-2}w_{\ell,\theta,q}^2|g||\nabla g|\Big),
\end{multline*}
hence, in view of~\eqref{eq:dissipation} and~\eqref{eq:dissipation-w},
\begin{equation*}
I_1(g)
\,\ge\,\3g\3_{\ell,\theta,K}^2-C_{V,\ell,\theta,K,N}\3g\3^2
-\tfrac1NC_V\3g\3_{\ell,\theta,K}^2.
\end{equation*}
Choosing $N\ge 2C_V$, this yields the claim~\eqref{eq:bnd-I1}.

\medskip
\substep{1.3} Proof that
\begin{equation}\label{eq:bnd-I234}
|I_2(g)|+|I_3(g)|+|I_4(g)|\,\lesssim_{V,\ell,\theta,K}\,\3g\3^2+\|g\|_{H^2(\R^d)}\3g\3_{\ell,\theta,K}^2.
\end{equation}
We start with the bound on $I_2$. By definition of $\Bc$ and of the collision kernel $B$, cf.~\eqref{eq:def-Bc} and~\eqref{eq:LB-kernel}, and by Lemma~\ref{lem:eps-bndd-R}(ii), we get
\begin{equation*}
|I_2(g)|\,\lesssim_V\,\iint_{\R^d\times\R^d}|v-v_*|^{-1}\sqrt\mu\sqrt\mu_*|(\nabla+v)(w_{\ell,\theta,K}^2g)||((\nabla+v)g)_*|\,dvdv_*.
\end{equation*}
Appealing to the Hardy--Littlewood--Sobolev inequality and using the Gaussian decay of $\sqrt\mu w_{\ell,\theta,K}^2$ (provided $K\ll1$ small enough in case $\theta=2$), we deduce for all $r\ge0$,
\begin{eqnarray*}
|I_2(g)|
&\lesssim_V&\|\sqrt\mu (\nabla+v)(w_{\ell,\theta,K}^2g)\|_{\Ld^\frac{2d}{2d-1}(\R^d)}\|\sqrt\mu(\nabla+v)g\|_{\Ld^\frac{2d}{2d-1}(\R^d)}\\
&\lesssim_{\ell,\theta,K,r}&\|\rv^{-r} (\nabla+v)g\|_{\Ld^2(\R^d)}^2,
\end{eqnarray*}
and thus, by~\eqref{eq:dissipation},
\begin{equation*}
|I_2(g)|
\,\lesssim_{V,\ell,\theta,K,r}\,\3g\3^2.
\end{equation*}
We turn to the bound on $I_3$.
By~\eqref{eq:lemB-1} in Lemma~\ref{lem:B}, we get
\begin{multline*}
|I_3(g)|\,\lesssim_V\,\|w_{\ell,\theta,K}^{-1}(\nabla+v)(w_{\ell,\theta,K}^2g)\|_{\Ld^2_A(\R^d)}\|g\|_{H^2(\R^d)}\\
\times\Big(\|w_{\ell,\theta,K}\nabla g\|_{\Ld^2_A(\R^d)}+\|\rv^{-\frac12}w_{\ell,\theta,K} g\|_{\Ld^2(\R^d)}\Big),
\end{multline*}
and thus, by Lemma~\ref{lem:A}(i), noting again that
\begin{equation}\label{eq:prop-wellthetK}
\nabla w_{\ell,\theta,K}^2=P_v\nabla w_{\ell,\theta,K}^2,\qquad|\nabla w_{\ell,\theta,K}^2|\lesssim_{\ell,\theta,K}\rv w_{\ell,\theta,K}^2,
\end{equation}
and recalling~\eqref{eq:dissipation-w}, we deduce
\begin{equation*}
|I_3(g)|\,\lesssim_{V,\ell,\theta,K}\,\|g\|_{H^2(\R^d)}\3g\3_{\ell,\theta,K}^2.
\end{equation*}
We turn to the bound on $I_4$.
By definition of $\Bc$ and of the collision kernel $B$, cf.~\eqref{eq:def-Bc} and~\eqref{eq:LB-kernel}, Lemma~\ref{lem:eps-bndd-R}(i)--(ii) yields
\begin{multline*}
|I_4(g)|\,\lesssim_V\,\int_{\R^d}\bigg(\iint_{\R^d\times\R^d}|k|^2\widehat V(k)^2\delta(k\cdot(v-v_*))\Big|\int_{\R^d}\tfrac{k\cdot\nabla(\sqrt\mu g)_{**}}{k\cdot(v-v_{**})-i0}dv_{**}\Big|\,\dbar k\,\mu_*dv_*\bigg)\\
\times|(\nabla+v)(w_{\ell,\theta,q}^2g)||(\nabla+v)g|\,dv,
\end{multline*}
hence, by the Sobolev inequality and the boundedness of the Hilbert transform,
\begin{eqnarray*}
|I_4(g)|&\lesssim_V&\|g\|_{H^2(\R^d)}\int_{\R^d}\rv^{-1}|(\nabla+v)(w_{\ell,\theta,q}^2g)||(\nabla+v)g|\,dv\\
&\lesssim&\|g\|_{H^2(\R^d)}\|\rv^{-\frac12}w_{\ell,\theta,K}^{-1}(\nabla+v)(w_{\ell,\theta,K}^2g)\|_{\Ld^2(\R^d)}\|w_{\ell,\theta,K}(\nabla+v)g\|_{\Ld^2(\R^d)}.
\end{eqnarray*}
Improving on the weights $\rv^{-\frac12}$ as in~\eqref{eq:improved-Gbnd0}, we deduce
\begin{equation*}
|I_4(g)|\,\lesssim\,\|g\|_{H^2(\R^d)}\|w_{\ell,\theta,K}^{-1}(\nabla+v)(w_{\ell,\theta,K}^2g)\|_{\Ld^2_A(\R^d)}\|w_{\ell,\theta,K}(\nabla+v)g\|_{\Ld^2_A(\R^d)}.
\end{equation*}
hence, by~\eqref{eq:dissipation-w} and~\eqref{eq:prop-wellthetK},
\begin{equation*}
|I_4(g)|\,\lesssim_{V,\ell,\theta,K}\,\|g\|_{H^2(\R^d)}\3g\3^2_{\ell,\theta,K}.
\end{equation*}
This proves the claim~\eqref{eq:bnd-I234}.

\medskip
\substep{1.4} Proof of~\eqref{eq:compact-est}.\\
Combining~\eqref{eq:bnd-I1} and~\eqref{eq:bnd-I234}, we get for the solution $f$ of~\eqref{eq:g-reform-re},
\[\tfrac12\partial_t\int_{\R^d}w_{\ell,\theta,K}^2|f|^2+\tfrac12\3f\3_{\ell,\theta,K}^2\,\lesssim_{V,\ell,\theta,K}\,\3f\3^2+\|f\|_{H^2(\R^d)}\3f\3^2_{\ell,\theta,K}.\]
Provided $\|f^\circ\|_{H^2(\R^d)}\ll_{V,\beta_0,\ell,\theta,K}1$, the stability estimate in Theorem~\ref{th:global} yields $\|f^t\|_{H^2(\R^d)}\ll_{V,\ell,\theta,K}1$ for all $t\ge0$, hence the above becomes
\[\tfrac12\partial_t\int_{\R^d}w_{\ell,\theta,K}^2|f|^2+\tfrac14\3f\3_{\ell,\theta,K}^2\,\lesssim_{V,\ell,\theta,K}\,\3f\3^2,\]
or alternatively, after integration,
\[\int_{\R^d}w_{\ell,\theta,K}^2|f^t|^2+\tfrac12\int_0^t\3f^s\3_{\ell,\theta,K}^2ds\,\le\,\int_{\R^d}w_{\ell,\theta,K}^2|f^\circ|^2+C_{V,\ell,\theta,K}\int_0^t\3f^s\3^2ds.\]
Recalling that~\eqref{eq:L2-dissip} yields
\begin{equation*}
\|f\|_{\Ld^2(\R^d)}^2+\tfrac1{C_V}\int_0^t\3f^s\3^2ds\,\le\,\|f^\circ\|_{\Ld^2(\R^d)}^2,
\end{equation*}
the claim~\eqref{eq:compact-est} follows.

\medskip
\step2 Proof of~(i)--(ii).\\
Given $\e>0$, we can estimate in view of~\eqref{eq:dissipation},
\[\3g\3\,\gtrsim_V\,\int_{\R^d}\rv^{-1}|g|^2\,\ge\,\rt^{-\e}\int_{\R^d}\mathds1_{\rv\le\rt^\e}|g|^2,\]
so that \eqref{eq:L2-dissip} becomes
\[\partial_t\int_{\R^d}|f|^2+\tfrac1{C_V}\rt^{-\e}\int_{\R^d}\mathds1_{\rv\le\rt^\e}|f|^2\,\le\,0,\]
or equivalently,
\[\partial_t\int_{\R^d}|f|^2+\tfrac1{C_V}\rt^{-\e}\int_{\R^d}|f|^2\,\le\,\tfrac1{C_V}\rt^{-\e}\int_{\R^d}\mathds1_{\rv>\rt^\e}|f|^2.\]
By integration, this entails
\begin{multline*}
\int_{\R^d}|f^t|^2\,\le\,e^{-\frac1{1-\e}\frac1{C_V}\rt^{1-\e}}\int_{\R^d}|f^\circ|^2\\
+\tfrac1{C_V}\int_0^t\rs^{-\e}e^{-\frac1{1-\e}\frac1{C_V}(\rt^{1-\e}-\rs^{1-\e})}\Big(\int_{\R^d}\mathds1_{\rv>\rs^\e}|f^s|^2\Big)ds.
\end{multline*}
Appealing to~\eqref{eq:compact-est} in form of
\begin{equation*}
\int_{\R^d}\mathds1_{\rv>\rs^\e}|f^s|^2\,\lesssim_{V,\ell,\theta,K}\,\rs^{-2\e\ell}e^{-2K\rs^{\e\theta}}\int_{\R^d}w_{\ell,\theta,K}^2|f^\circ|^2,
\end{equation*}
we deduce
\begin{multline*}
\int_{\R^d}|f^t|^2\,\le\,\Big(e^{-\frac1{1-\e}\frac1{C_V}\rt^{1-\e}}
+C_{V,\ell,\theta,K}\int_0^te^{-\frac1{1-\e}\frac1{C_V}(\rt^{1-\e}-\rs^{1-\e})}\rs^{-2\e\ell-\e}e^{-2K\rs^{\e\theta}}ds\Big)\\
\times\int_{\R^d}w_{\ell,\theta,K}^2|f^\circ|^2.
\end{multline*}
This yields the conclusion~(i)--(ii) after straightforward computations.
\qed

\section{Local well-posedness away from equilibrium}\label{sec:away-equ}
This section is devoted to the proof of Theorem~\ref{th:local}.
The proof follows the same lines as the proof of Proposition~\ref{prop:loc-wp-equ}: most steps are adaptations of their counterparts in Section~\ref{sec:global}, hence for the sake of conciseness we shall focus on the points that differ.
Given initial data~$F^\circ$, we decompose the solution $F$ as
\begin{equation}\label{eq:F-decomp-F0}
F \,=\, F^\circ+G.
\end{equation}
In order to be able to control the nonlinear term, we shall construct the small perturbation~$G$ locally in time in some Sobolev space with polynomial weight.
We introduce short-hand notation for these standard weighted Sobolev spaces:
for $m\in \R$, $s\ge0$, and $1\le p<\infty$,
let $\Ld^p_m(\R^d)$ and $H^s_m(\R^d)$ be the weighted Lebesgue and Sobolev spaces with norms
\begin{eqnarray*}
\|F\|^p_{\Ld^p_m(\R^d)} &=& \int_{\R^d} \langle v\rangle^m|F|^p,\\
\|F\|^2_{H^s_m(\R^d)} &=& \int_{\R^d}\langle v\rangle^m |\langle\nabla\rangle^s F|^2,
\end{eqnarray*}
and we define the analogue of the energy dissipation norm~\eqref{eq:def-L2A} in this functional setting,
\begin{equation*}
\|F\|^2_{\Ld^2_{A,m}(\R^d)}= \int_{\R^d} \langle v\rangle^m\,F\cdot A F.
\end{equation*}
We make abundant use of the following elementary estimate: given $r\ge0$ and $1\le p\le 2$, we have for all $m>2r+d\frac{2-p}p$,
\begin{align*}
\| \rv^r F \|_{\Ld^p(\R^d)} \lesssim_m \|F\|_{\Ld^2_m(\R^d)}.
\end{align*}
We start by stating that the estimates for the dispersion function $\e$ in Lemma~\ref{lem:eps-bndd-R} carry over to these weighted spaces; the proof is omitted.

\begin{samepage}
\begin{lem}\label{lem:dielectric2}
Let $V\in\Ld^1(\R^d)$.
\begin{enumerate}[(i)]
\item \emph{Non-degeneracy:}
Provided $F^\circ\in\Ld^1(\R^d)$ satisfies the following non-degeneracy and boundedness conditions for some $\delta_0,M_0>0$,
\[\inf_{k,v\in\R^d}|\e(k,k\cdot v;\nabla F^\circ)|\,\ge\,\tfrac1{M_0},\qquad\|\langle\nabla\rangle^{\frac32+\delta_0}F^\circ\|_{\Ld^2_{d-1+\delta_0}(\R^d)}\le M_0,\]
and provided $G\in\Ld^1(\R^d)$ satisfies the following smallness condition for some large enough constant $C_0$,
\begin{equation*}
\|\langle\nabla\rangle^{\frac32+\delta_0}G\|_{\Ld^2_{d-1+\delta_0}(\R^d)}\,\leq\,\tfrac{1}{C_0},
\end{equation*}
we have for all $k,v\in\R^d$,
\begin{align*}
|\e(k,k\cdot v;\nabla (F^\circ +G))| \,\simeq_{V,M_0,\delta_0}\,1.
\end{align*}
\item \emph{Boundedness:} For all $\alpha>0$ and $\delta>0$, we have for all $k,v\in\R^d$,
\begin{equation*}
|\nabla^\alpha_v \e(k,k\cdot v;\nabla F)| \,\lesssim_{V,\alpha,\delta}\,\|\langle \nabla \rangle^{|\alpha|+\frac32+\delta} F\|_{\Ld^2_{d-1+\delta}(\R^d)}.\qedhere
\end{equation*}
\end{enumerate} 
\end{lem}
\end{samepage}

As we no longer use Maxwellian weights in this section, cf.~\eqref{eq:F-decomp-F0}, we consider the following linear operator $\Bo(\nabla F)$, instead of $\Bc(\nabla F)$ in~\eqref{eq:def-Bc},
\begin{align}\label{eq:defBo}
\Bo(\nabla F)[H](v) \,:=\, \int_{\R^d}B(v,v-v_*;\nabla F) \,H_*\,dv_*.
\end{align}
The following analogue of the coercivity estimate of Lemma~\ref{lem:A} is easily obtained.

\begin{lem}\label{lem:coercivity}
Let $V\in\Ld^1(\R^d)$ be isotropic, let $F^\circ\in\Ld^1(\R^d)$ be nonnegative and satisfy for some $M_0>0$ and $v_0\in\R^d$,
\[\inf_{|\cdot-v_0|\le\frac1{M_0}}F^\circ\ge\tfrac1{M_0},\]
and let $F\in\Ld^1(\R^d)$ satisfy
\begin{align*}
|\e(k,k\cdot v;\nabla F)| \,\le\,M_0.
\end{align*}
Then the following coercivity estimate holds,
\begin{equation*}
e\cdot\Bo(\nabla F)[F^\circ](v) \,e \,\gtrsim_{V,M_0,v_0}\, \rv^{-1}|P^\perp_v e|^2+\rv^{-3}|P_v e|^2.\qedhere
\end{equation*}
\end{lem}

\begin{proof}
The definition of $\Bo$, the lower bound on $F^\circ$, the condition on $F$, and the identity~\eqref{eq:expl-Landau-pr} yield
\begin{eqnarray*}
e\cdot\Bo(\nabla F)[F^\circ](v)\,e&\gtrsim_{M_0}&\int_{|v_*-v_0|\le\frac1{M_0}}\bigg(\int_{\R^d}|e\cdot k|^2\pi\widehat V(k)^2\delta(k\cdot(v-v_*))\,\dbar k\bigg)\,dv_*\\
&\simeq_{V}&\int_{|v_*-v_0|\le\frac1{M_0}}\tfrac1{|v-v_*|}|P_{v-v_*}^\bot e|^2\,dv_*.
\end{eqnarray*}
Arguing as in the proof of Lemma~\ref{lem:A}, we may then deduce
\begin{eqnarray*}
e\cdot\Bo(\nabla F)[F^\circ](v)\,e&\gtrsim_{V,M_0}&\langle v-v_0\rangle^{-1}|P_{v-v_0}^\bot e|^2+\langle v-v_0\rangle^{-3}|P_{v-v_0}e|^2.
\end{eqnarray*}
Using that $\langle v-v_0\rangle\lesssim_{v_0}\langle v\rangle$ and that $|P_{v-v_0}e-P_{v}e|\lesssim_{v_0}\langle v\rangle^{-1}|e|$, the conclusion easily follows.
\end{proof}

Finally, we will also need the following analogue of~\eqref{eq:lemB-1} and~\eqref{eq:lemB-3} in Lemma~\ref{lem:B} for boundedness properties of the operator~$\Bo$; the proof is similar and omitted.

\begin{lem}
Let $d>2$, let $V\in\Ld^1\cap\dot H^2(\R^d)$ be isotropic, and assume $xV\in\Ld^2(\R^d)$.
Provided $F^\circ\in\Ld^1(\R^d)$ satisfies the following non-degeneracy and boundedness conditions for some $M_0>0$,
\[\inf_{k,v\in\R^d}|\e(k,k\cdot v;\nabla F^\circ)|\,\ge\,\tfrac1{M_0},\qquad\|F^\circ\|_{H^2_{d}(\R^d)}\le M_0,\]
and provided $G\in\Ld^1(\R^d)$ satisfies the following smallness condition for some large enough constant $C_0$,
\begin{equation*}
\|G\|_{H^2_{d}(\R^d)}\,\leq\,\tfrac{1}{C_0},
\end{equation*}
we have for all vector fields $h_1,h_2$, all $\alpha\ge0$, and $m>d+7$,
\begin{multline}\label{eq:Gbound1}
\Big| \int_{\R^d} h_1\cdot \big(\nabla^\alpha \Bo (\nabla (F^\circ+G))[H]\big) \,h_2 \Big|
\,\lesssim_{V,m,M_0,\alpha} \,\|h_1\|_{\Ld^2_{A}(\R^d)} \|h_2\|_{\Ld^2_{A}(\R^d)}\\
\times\|H\|_{H^{|\alpha|+1}_{m}(\R^d)}\Big(1+\|G\|^{|\alpha|}_{H^{|\alpha|+1}_{d}(\R^d)}+ \mathds1_{\alpha>0} \|G\|_{H^{|\alpha|+2}_{d}(\R^d)}\Big),
\end{multline}
and alternatively, if needed to loose less derivatives on $G,H$,
\begin{align}\label{eq:Gbound3}
&\Big|\int_{\R^d}h_1\cdot \big(\nabla^\alpha\overline\Bc(\nabla (F^\circ+G))[H]\big)\,h_2\Big|
\,\lesssim_{V,m,M_0,\alpha}\,\|h_1\|_{\Ld^2_A(\R^d)}\|\rnabla h_2\|_{\Ld^2_A(\R^d)}\\
&\hspace{7cm}\times\|H\|_{H_m^{|\alpha|\vee2}(\R^d)}
\Big(1+\|G\|_{H^{|\alpha|\vee3}_d(\R^d)}^{|\alpha|\vee2}\Big).\nonumber\qedhere
\end{align}
\end{lem}

With the above estimates at hand, we are now in position to prove local well-posedness away from equilibrium, that is, Theorem~\ref{th:local}.

\begin{proof}[Proof of Theorem~\ref{th:local}]
We proceed by constructing a sequence $\{F_n\}_n$ of approximate solutions: we choose the $0$th-order approximation as the initial data,
\[F_0:=F^\circ,\]
and next, for all $n\ge0$, given the $n$th-order approximation $F_n$, we iteratively define $F_{n+1}$ as the solution of the following linear Cauchy problem,
\begin{equation}\label{eq:Fn-recurr}
\left\{\begin{array}{l}
\partial_t F_{n+1} = \nabla\cdot \left(\Bo_n [F_n] \nabla F_{n+1} -F_{n+1}\Bo_n [\nabla F_n] \right),\\
F_{n+1}|_{t=0} = F^ \circ,
\end{array}\right.
\end{equation}
where we have set for abbreviation
$\Bo_n[G]:=\Bo(\nabla F_n)[G]$.
As $F^\circ$ is nonnegative by assumption, the maximum principle ensures that $F_n$ remains nonnegative for all $n$.
Decompose
\[F_n\,=\,F^\circ + G_n,\qquad G_n|_{t=0}=0,\]
let $s\ge2$ and $m>d+7$ be fixed,
and, given a constant $C_0>0$ to be later chosen large enough, define the maximal time
\begin{multline}\label{eq:def-Tn2}
T_n\,:=\,1\wedge\max\Big\{T\ge0~:~\|G_n\|_{\Ld^\infty([0,T];H^s_m(\R^d))}^2\\
+\tfrac1{C_0}\|\nabla\rnabla^s G_n\|_{\Ld^2([0,T];\Ld^2_{A,m}(\R^d))}^2\le \tfrac1{C_0^2}\Big\}.
\end{multline}
Equation~\eqref{eq:Fn-recurr} yields for all $|\alpha|\le s$,
\begin{equation}\label{eq:evol-L2m}
\tfrac12\partial_t \|\nabla^\alpha G_{n+1}\|^2_{\Ld^2_m(\R^d)} \,=\,\sum_{\gamma\le \alpha}\binom\alpha\gamma \Big(I^{\alpha}_\gamma(G_{n+1};F_n) +J^{\alpha}_\gamma(G_{n+1};F_n)\Big)+K^{\alpha}(G_{n+1};F_n),
\end{equation}
in terms of
\begin{eqnarray*}
I^{\alpha}_\gamma(G_{n+1};F_n) &:=& -\int_{\R^d} \langle v\rangle^m(\nabla\nabla^\alpha G_{n+1})\cdot (\nabla^{\alpha-\gamma}\Bo_n[F_n])\,\nabla \nabla^\gamma  (F^\circ+ G_{n+1}), \\
J^{\alpha}_\gamma(G_{n+1};F_n) &:=& \int_{\R^d} \langle v\rangle^m(\nabla \nabla^\alpha G_{n+1})\cdot (\nabla^{\gamma} \Bo_n[\nabla F_n])\,\nabla^{\alpha-\gamma} (F^\circ+ G_{n+1}),\\
K^{\alpha}(G_{n+1};F_n)  &:=& -\int_{\R^d}(\nabla\langle v\rangle^m)\,(\nabla^\alpha G_{n+1})\,\nabla^\alpha \Big(\Bo_n[F_n] \nabla F_{n+1} -F_{n+1}\Bo_n[\nabla F_n] \Big).
\end{eqnarray*}
We start by extracting the dissipation contained in the term $I^\alpha_\alpha$. For that purpose, we decompose
\begin{multline}\label{eq:decomp-Iaa}
I^{\alpha}_\alpha(G_{n+1};F_n) \,=\, -\int_{\R^d}\rv^m(\nabla\nabla^\alpha G_{n+1})\cdot\Bo_n[F^\circ] (\nabla\nabla^\alpha  G_{n+1})\\
-\int_{\R^d}\rv^m(\nabla\nabla^\alpha G_{n+1})\cdot\Bo_n[G_n] \nabla\nabla^\alpha (F^\circ+G_{n+1})\\
-\int_{\R^d}\rv^m(\nabla\nabla^\alpha G_{n+1})\cdot\Bo_n[F^\circ] (\nabla\nabla^\alpha F^\circ).
\end{multline}
In view of assumptions~\eqref{eq:as-F0} on $F^\circ$ and in view of the definition~\eqref{eq:def-Tn2} of $T_n$, provided $C_0$ is chosen large enough, Lemma~\ref{lem:dielectric2}(i) yields for $t\le T_n$,
\[|\e(k,k\cdot v;\nabla F_n)|\,\simeq_{V,M_0}\,1,\]
and Lemma~\ref{lem:coercivity} can then be applied to the effect of
\[-\int_{\R^d}\rv^m(\nabla\nabla^\alpha G_{n+1})\cdot\Bo_n[F^\circ] (\nabla\nabla^\alpha  G_{n+1})\,\le\,-\tfrac1{C_{V,M_0}}\|\nabla\nabla^\alpha G_{n+1}\|_{\Ld^2_{A,m}(\R^d)}^2.\]
Further using the bound~\eqref{eq:Gbound1} to estimate the last two right-hand side terms in~\eqref{eq:decomp-Iaa}, we deduce for $t\le T_n$, in view of~\eqref{eq:def-Tn2},
\begin{multline*}
I^{\alpha}_\alpha(G_{n+1};F_n) \,\le\,-\tfrac1{C_{V,M_0}}\|\nabla\nabla^\alpha G_{n+1}\|_{\Ld^2_{A,m}(\R^d)}^2\\
+C_{V,M_0,m}\|\nabla\nabla^\alpha G_{n+1}\|_{\Ld^2_{A,m}(\R^d)}\Big(1+\tfrac1{C_0}\|\nabla\nabla^\alpha G_{n+1}\|_{\Ld^2_{A,m}(\R^d)}\Big).
\end{multline*}
Note that this makes use of the control on $F^\circ$ in $H^{s+1}_m(\R^d)$ rather than only in $H^{s}_m(\R^d)$.
Choosing~$C_0$ large enough, this easily yields
\begin{equation*}
I^{\alpha}_\alpha(G_{n+1};F_n) +\tfrac1{2C_{V,M_0}} \|\nabla\nabla^\alpha G_{n+1}\|^2_{\Ld^2_{A,m}(\R^d)}\,\le\, C_{V,M_0,m}.
\end{equation*}
Now using~\eqref{eq:Gbound1} and~\eqref{eq:Gbound3} to estimate $I^\alpha_\gamma,J^\alpha_\gamma,K^\alpha$ in~\eqref{eq:evol-L2m},
we easily arrive at the following, for $t\le T_n$,
\begin{multline}\label{eq:est-Hs-der-t-GG}
\partial_t \|G_{n+1}\|^2_{H^s_m(\R^d)}+\tfrac1{C_{V,M_0}} \|\nabla\rnabla^s G_{n+1}\|^2_{\Ld^2_{A,m}(\R^d)}\\
\,\lesssim_{V,M_0,m,s}\,
\Big(1+\|G_{n+1}\|_{H^{s}_m(\R^d)}^2\Big)\Big(1+\|G_{n}\|_{H^{s\vee3}_m(\R^d)}^{2(s\vee2+1)}\Big).
\end{multline}
Note that the difference with~\eqref{eq:Hs-evol-fn1} is essential: the norm $\|\nabla\rnabla^sG_n\|_{\Ld^2_{A,m}(\R^d)}$ does not appear in the present right-hand side. This is permitted by the use of the improved estimate~\eqref{eq:Gbound3}, to be compared with~\eqref{eq:lemB-3} in Lemma~\ref{lem:B}, which was not needed in~\eqref{eq:Hs-evol-fn1} for the proof of local well-posedness close to equilibrium.

\medskip\noindent
Now from~\eqref{eq:est-Hs-der-t-GG},
after time integration, we deduce for $t\le T_n$, in view of~\eqref{eq:def-Tn2} and $G_{n+1}|_{t=0}=0$,
\begin{equation*}
\|G_{n+1}^t\|^2_{H^s_m(\R^d)}+\tfrac1{C_{V,M_0}} \|\nabla\rnabla^s G_{n+1}\|^2_{\Ld^2([0,t];\Ld^2_{A,m}(\R^d))}
\,\le\, t\,C_{V,M_0,m,s}\,e^{t\,C_{V,M_0,m,s}}.
\end{equation*}
Choosing $C_0\ge C_{V,M_0}$, this entails
\[T_{n+1}\ge T_n\wedge(eC_{V,M_0,m,s}C_0^2)^{-1}.\]
Therefore, setting $T_0:=(eC_{V,M_0,m,s}C_0^2)^{-1}$, we obtain by iteration a sequence $(F_n)_n$ satisfying the sequence of approximate equations~\eqref{eq:Fn-recurr} on the time interval~$[0,T_0]$ with
\[F_n=F^\circ+G_n,\qquad\|G_n\|_{\Ld^\infty([0,T_0];H^s_m(\R^d))}\,\le\,\tfrac1{C_0}.\]
The conclusion now follows easily by a compactness argument, passing to the limit $n\uparrow\infty$ in approximate equations~\eqref{eq:Fn-recurr}.
\end{proof}

\section{Landau approximation}
This section is devoted to the proof of Theorem~\ref{cor:Landau}.
As the Fourier transform of the rescaled potential~\eqref{eq:Vscale} takes the form
\begin{equation*}
\widehat V_{\delta}(k) \,=\, \delta^{d-a}\widehat V(\delta k),
\end{equation*}
the Lenard--Balescu equation~\eqref{eq:LB} reads as follows for the time-rescaled velocity density~$\tilde F_{\delta}=\mu_\beta+\sqrt{\mu_\beta}\tilde f_{\delta}$,
\[\partial_t\tilde F_{\delta}=\nabla\cdot\int_{\R^d}B_{\delta}(v,v-v_*;\nabla\tilde F_{\delta})\,\big(\tilde F_{\delta,*}\nabla\tilde F_{\delta}-\tilde F_{\delta}\nabla_*\tilde F_{\delta,*}\big),\]
in terms of
\begin{eqnarray*}
B_{\delta}(v,v-v_*;\nabla F)&:=&\int_{\R^d}(k\otimes k)\,\pi\widehat V(k)^2\tfrac{\delta(k\cdot(v-v_*))}{|\e_{\delta}(k,k\cdot v;\nabla F)|^2}\,\dbar k,\\
\e_{\delta}(k,k\cdot v;\nabla F)&:=&1+\delta^{d-a}\widehat V(k)\int_{\R^d}\tfrac{k\cdot\nabla F(v_*)}{k\cdot(v-v_*)-i0}\,dv_*.
\end{eqnarray*}
For $a< d$, global well-posedness can now obviously be deduced as in Theorem~\ref{th:global} uniformly with respect to $0< \delta\le1$. 
As the proof of Lemma~\ref{lem:eps-bndd-R}(ii) yields
\[\|\e_{\delta}(\cdot,\cdot;\nabla \tilde F_{\delta})-1\|_{\Ld^\infty(\R^+;\Ld^\infty(\R^d))}\,\lesssim\,\delta^{d-a}\|V\|_{\Ld^1(\R^d)}\big(1+\|\tilde f_{\delta}\|_{H^2(\R^d)}\big),\]
the conclusion follows from the explicit computation~\eqref{eq:expl-Landau-pr}.\qed

\section*{Acknowledgements}
MD acknowledges financial support from the CNRS-Momentum program. RW acknowledges support  of Universit\'e de Lyon through the IDEXLYON Scientific Breakthrough Project ``Particles drifting and propelling in turbulent flows'', 
and the hospitality of the UMPA ENS Lyon.

\bibliographystyle{plain}
\bibliography{biblio}

\def\cprime{$'$}\def\cprime{$'$} \def\cprime{$'$}
\begin{thebibliography}{10}

\bibitem{Alexandre-Villani-04}
R.~Alexandre and C.~Villani.
\newblock On the {L}andau approximation in plasma physics.
\newblock {\em Ann. Inst. H. Poincar\'{e} Anal. Non Lin\'{e}aire},
  21(1):61--95, 2004.

\bibitem{Balescu-60}
R.~Balescu.
\newblock Irreversible {P}rocesses in {I}onized {G}ases.
\newblock {\em Phys. Fluids}, 3(1):52--63, 1960.

\bibitem{Balescu-63}
R.~Balescu.
\newblock {\em Statistical mechanics of charged particles}.
\newblock Wiley Interscience, New York, 1963.

\bibitem{Bobylev-Potapenko-19}
A.~Bobylev and I.~Potapenko.
\newblock Long wave asymptotics for the {V}lasov-{P}oisson-{L}andau kinetic
  equation.
\newblock {\em J. Stat. Phys.}, 175:1--18, 2019.

\bibitem{Calderon-Zygmund}
A.~P. Calder\'on and A.~Zygmund.
\newblock On singular integrals.
\newblock {\em Amer. J. Math.}, 78:289--309, 1956.

\bibitem{Decoster}
A.~Decoster, P.~A. Markowich, and B.~Perthame.
\newblock {\em Modeling of collisions}, volume~2 of {\em Series in Applied
  Mathematics (Paris)}.
\newblock Gauthier-Villars, \'{E}ditions Scientifiques et M\'{e}dicales
  Elsevier, Paris, North-Holland, Amsterdam, 1998.
\newblock Edited and with a foreword by P. A. Raviart.

\bibitem{Degond-Lemou-97}
P.~Degond and M.~Lemou.
\newblock Dispersion relations for the linearized {F}okker-{P}lanck equation.
\newblock {\em Arch. Rational Mech. Anal.}, 138(2):137--167, 1997.

\bibitem{desvillettes2020new}
L.~Desvillettes, L.-B. He, and J.-C. Jiang.
\newblock {A new monotonicity formula for the spatially homogeneous {L}andau
  equation with {C}oulomb potential and its applications}.
\newblock Preprint, arXiv:2011.00386.

\bibitem{D-21a}
M.~Duerinckx.
\newblock On the size of chaos via {G}lauber calculus in the classical
  mean-field dynamics.
\newblock {\em Commun. Math. Phys.}, 382:613--653, 2021.

\bibitem{DSR}
M.~Duerinckx and L.~Saint-Raymond.
\newblock Lenard--{B}alescu correction to mean-field theory.
\newblock {\em Prob. Math. Phys.}, 2(1):27--69, 2021.

\bibitem{golse2019partial}
F.~Golse, M.~P. Gualdani, C.~Imbert, and A.~Vasseur.
\newblock {Partial Regularity in Time for the Space Homogeneous {L}andau
  Equation with {C}oulomb Potential}.
\newblock Preprint, arXiv:1906.02841.

\bibitem{Guernsey-60}
R.~L. Guernsey.
\newblock The kinetic theory of fully ionized gases.
\newblock PhD thesis, University of Michigan, 1960.

\bibitem{Guernsey-62}
R.~L. Guernsey.
\newblock Kinetic {E}quation for a {C}ompletely {I}onized {G}as.
\newblock {\em Phys. Fluids}, 5:322--328, 1962.

\bibitem{Guo-02}
Y.~Guo.
\newblock The {L}andau {E}quation in a {P}eriodic {B}ox.
\newblock {\em Commun. Math. Phys.}, 231:391--434, 2002.

\bibitem{Lenard-60}
A.~Lenard.
\newblock On {B}ogoliubov's kinetic equation for a spatially homogeneous
  plasma.
\newblock {\em Ann. Phys.}, 10:390--400, 1960.

\bibitem{Merchant-Liboff-73}
A.~H. Merchant and R.~L. Liboff.
\newblock Spectral properties of the linearized {B}alescu-{L}enard operator.
\newblock {\em J. Math. Phys.}, 14(1):119--129, 1973.

\bibitem{Nicholson}
D.~R. Nicholson.
\newblock {\em Introduction to {P}lasma {T}heory}.
\newblock John Wiley \& Sons Inc., 1983.

\bibitem{Nota-21-2}
A.~Nota, J.~Vel\'{a}zquez, and R.~Winter.
\newblock Interacting particle systems with long range interactions:
  approximation by tagged particles in random fields.
\newblock Preprint, arXiv:2003.11605.

\bibitem{Nota-21-1}
A.~Nota, J.~Vel\'{a}zquez, and R.~Winter.
\newblock Interacting particle systems with long-range interactions: scaling
  limits and kinetic equations.
\newblock {\em Rend. Lincei-Mat. Appl.}, 32:335--377, 2021.

\bibitem{Penrose}
O.~Penrose.
\newblock {E}lectrostatic instabilities of a uniform non-maxwellian plasma.
\newblock {\em Phys. Fluids}, 3(2):258--265, 1960.

\bibitem{Strain-07}
R.~M. Strain.
\newblock On the linearized {B}alescu-{L}enard equation.
\newblock {\em Comm. Partial Differential Equations}, 32(10-12):1551--1586,
  2007.

\bibitem{Guo-Strain-08}
R.~M. Strain and Y.~Guo.
\newblock Exponential decay for soft potentials near {M}axwellian.
\newblock {\em Arch. Ration. Mech. Anal.}, 187(2):287--339, 2008.

\bibitem{Toscani-Villani-99}
G.~Toscani and C.~Villani.
\newblock Sharp entropy dissipation bounds and explicit rate of trend to
  equilibrium for the spatially homogeneous {B}oltzmann equation.
\newblock {\em Comm. Math. Phys.}, 203(3):667--706, 1999.

\bibitem{Toscani-Villani-00}
G.~Toscani and C.~Villani.
\newblock On the trend to equilibrium for some dissipative systems with slowly
  increasing a priori bounds.
\newblock {\em J. Stat. Phys.}, 98:1279--1309, 2000.

\bibitem{VW-18}
J.~J.~L. Vel\'{a}zquez and R.~Winter.
\newblock The two-particle correlation function for systems with long-range
  interactions.
\newblock {\em J. Stat. Phys.}, 173(1):1--41, 2018.

\bibitem{Villani-02}
C.~Villani.
\newblock A review of mathematical topics in collisional kinetic theory.
\newblock In {\em Handbook of mathematical fluid dynamics, {V}ol. {I}}, pages
  71--305. North-Holland, Amsterdam, 2002.

\bibitem{Winter-21}
R.~Winter.
\newblock Convergence to the {L}andau equation from the truncated {BBGKY}
  hierarchy in the weak-coupling limit.
\newblock {\em J. Differential Equations}, 283:1--36, 2021.

\end{thebibliography}

\end{document}